\documentclass[10pt]{article}
\usepackage[colorlinks,citecolor=blue,urlcolor=blue]{hyperref}
\usepackage{amssymb,amsmath,amsthm}
\usepackage[english]{babel}
\usepackage[round]{natbib}
\usepackage{booktabs}
\usepackage{caption}
\usepackage{enumitem}
\usepackage{geometry}
\usepackage{graphicx}
\usepackage[utf8]{inputenc}
\usepackage{makecell}
\usepackage{mathtools}
\usepackage{tikz}
\usepackage{url}

\RequirePackage[colorlinks,citecolor=blue,urlcolor=blue]{hyperref}
\usepackage{cleveref}

\usetikzlibrary{shapes,arrows,positioning}

\tikzstyle{block} = [rectangle, draw,
    text width=16em, text centered, rounded corners, minimum height=4em]
\tikzstyle{block2} = [rectangle, draw,
    text width=8em, text centered, rounded corners, minimum height=2.5em]
\tikzstyle{block3} = [rectangle,
    text width=9em, text centered, minimum height=4em]
\tikzstyle{line} = [draw, -implies, double equal sign distance]

\usepackage{stevemath}

\newcommand{\tth}{$t$\textsuperscript{th} }
\newcommand{\kth}{$k$\textsuperscript{th} }
\newcommand{\Vhat}{\widehat{V}}

\newcommand{\Xbar}{\bar{X}}
\newcommand{\Sigmahat}{\widehat{\Sigma}}

\newcommand{\Xhat}{\widehat{X}}
\newcommand{\Yhat}{\widehat{Y}}
\newcommand{\Mtil}{\textup{DM}}

\newcommand{\flil}{f^{\text{LIL}}}
\newcommand{\lambdahat}{\hat{\lambda}}

\newcommand{\CI}{\textup{CI}}
\renewcommand{\ATE}{\textup{ATE}}

\newcommand{\NM}{\textup{NM}}

\newcommand{\GP}{\textup{GP}}
\newcommand{\GE}{\textup{GE}}

\newcommand{\uppergamma}{\bar{\gamma}}
\newcommand{\subpsiclass}{\mathbb{S}^{l_0}_{\psi}}
\newcommand{\overlambda}{\overline{\lambda}}
\newcommand{\Zbar}{\bar{Z}}

\begin{document}

\title{Time-uniform, nonparametric, nonasymptotic confidence sequences}
\author{Steven R. Howard$^1$ \and Aaditya Ramdas$^{2,3}$
  \and Jon McAuliffe$^{1,4}$ \and Jasjeet Sekhon$^{1,5}$ \and
  Departments of Statistics$^1$ and Political Science$^5$, UC Berkeley\\
  Departments of Statistics and Data Science$^2$ and Machine Learning$^3$,
  Carnegie Mellon\\
  The Voleon Group$^4$ \\
  \texttt{\{stevehoward,jonmcauliffe,sekhon\}@berkeley.edu},
  \texttt{aramdas@stat.cmu.edu}}
\maketitle

\begin{abstract}
  A confidence sequence is a sequence of confidence intervals that is uniformly
  valid over an unbounded time horizon. Our work develops confidence
  sequences whose widths go to zero, with nonasymptotic coverage guarantees
  under nonparametric conditions. We draw connections between the
 Cram\'er-Chernoff method for exponential concentration, the
  law of the iterated logarithm (LIL), and the sequential probability ratio
  test---our confidence sequences are time-uniform extensions of the first; 
  provide tight, nonasymptotic characterizations of the second; and
  generalize the third to nonparametric settings, including sub-Gaussian and
  Bernstein conditions, self-normalized processes, and matrix martingales. We
  illustrate the generality of our proof techniques by deriving an
  empirical-Bernstein bound growing at a LIL rate, as well as a novel upper LIL
  for the maximum eigenvalue of a sum of random matrices. Finally, we apply our
  methods to covariance matrix estimation and to estimation of sample average
  treatment effect under the Neyman-Rubin potential outcomes model.
\end{abstract}

\section{Introduction}\label{sec:introduction}

It has become standard practice for organizations with online presence to run
large-scale randomized experiments, or ``A/B tests'', to improve product
performance and user experience. Such experiments are inherently sequential:
visitors arrive in a stream and outcomes are typically observed quickly relative
to the duration of the test. Results are often monitored continuously using
inferential methods that assume a fixed sample, despite the known problem
that such monitoring inflates Type~I error substantially
\citep{armitage_repeated_1969, berman_p-hacking_2018}. Furthermore, most A/B
tests are run with little formal planning and fluid decision-making, compared
to clinical trials or industrial quality control, the traditional applications
of sequential analysis.

This paper presents methods for deriving \emph{confidence sequences} as a
flexible tool for inference in sequential experiments
\citep{darling_confidence_1967, lai_incorporating_1984,
  jennison_interim_1989}. For $\alpha \in (0,1)$, a $(1-\alpha)$-confidence
sequence is a sequence of confidence sets $(\CI_t)_{t=1}^\infty$, typically
intervals $\CI_t = (L_t, U_t) \subseteq \R$, satisfying a uniform coverage
guarantee: after observing the \tth unit, we calculate an updated confidence set
$\CI_t$ for the unknown quantity of interest $\theta_t$, with the uniform
coverage property
\begin{align}
\P(\forall t \geq 1 : \theta_t \in \CI_t) \geq 1 - \alpha.
\label{eq:conf_seq_defn}
\end{align}
With only a uniform lower bound $(L_t)$, i.e., if $U_t \equiv \infty$, we have a
\emph{lower confidence sequence}. Likewise, if $L_t \equiv -\infty$ we have an
\emph{upper confidence sequence} given by
$(U_t)$. \Cref{th:stitching,th:discrete_mixture,th:inverted_stitching} and
\cref{th:basic_mixture} are our key tools for constructing confidence
sequences. All build upon the general framework for uniform exponential
concentration introduced in \citet{howard_exponential_2018}, which means our
techniques apply in diverse settings: scalar, matrix, and Banach-space-valued
observations, with possibly unbounded support; self-normalized bounds applicable
to observations satisfying weak moment or symmetry conditions; and
continuous-time scalar martingales. Our methods allow for flexible control of
the ``shape'' of the confidence sequence, that is, how the sequence of intervals
shrinks in width over time. As a simple example, given a sequence of i.i.d.\
observations $(X_t)_{t=1}^\infty$ from a 1-sub-Gaussian distribution whose mean
$\mu$ we would like to estimate, \cref{th:stitching} yields the following
$(1-\alpha)$-confidence sequence for $\mu$, a special case of the more general
bound \eqref{eq:poly_stitching}:
\begin{align}\label{eq:stitching-intro}
{
\frac{\sum_{i=1}^t X_i}{t} 
  \pm 1.7 \sqrt{\frac{\log \log(2t) + 0.72 \log(10.4 / \alpha)}{t}}.}%
\end{align}
The $\Ocal(\sqrt{t^{-1} \log \log t})$ asymptotic rate of this bound matches the
lower bound implied by the law of the iterated logarithm (LIL), and
nonasymptotic bounds of this form are called \emph{finite LIL bounds}
\citep{jamieson_lil_2014}.

We develop confidence sequences that possess the following properties:

\begin{enumerate}
\item[(P1)] \textbf{Nonasymptotic and nonparametric}: our confidence sequences
  offer coverage guarantees for all sample sizes, without exact distributional
  assumptions or asymptotic approximations.
\item[(P2)] \textbf{Unbounded sample size}: our methods do not require a final
  sample size to be chosen ahead of time. They may be tuned for a planned sample
  size but always permit additional sampling.
\item[(P3)] \textbf{Arbitrary stopping rules}: we make no assumptions on the
  stopping rule used by an experimenter to decide when to end the experiment, or
  when to act on certain inferences.
\item[(P4)] \textbf{Asymptotically zero width}: the interval widths of our
  confidence sequences shrink towards zero at a $1/\sqrt{t}$ rate, ignoring log
  factors, just as with pointwise confidence intervals.
\end{enumerate}

\begin{figure}
\centering
\includegraphics{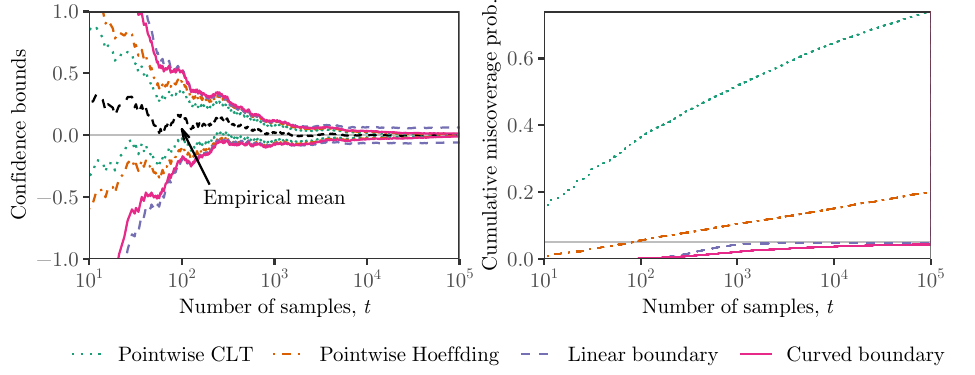}
\caption{
Left panel shows 95\% pointwise confidence intervals and uniform
  confidence sequences for the mean of a Rademacher random variable, using one
  simulation of 100,000 i.i.d.\ draws. Right panel shows cumulative chance of
  miscoverage based on 10,000 replications; flat grey line shows the nominal target level
  0.05. The CLT intervals are asymptotically pointwise valid (these are similar
  to the exact binomial confidence intervals, which are nonasymptotically
  pointwise valid). The pointwise Hoeffding intervals are nonasymptotically
  pointwise valid. The confidence sequence based on a linear boundary, as in
  \cref{th:uniform_chernoff}, is valid uniformly over time and
  nonasymptotically, but does not shrink to zero width. Finally, the confidence
  sequence based on a curved boundary is valid uniformly and nonasymptotically,
  while also shrinking towards zero width; here we use the two-sided normal
  mixture boundary, \eqref{eq:two_sided_normal_mixture}, qualitatively similar
  to the stitched bound \eqref{eq:stitching-intro}. \label{fig:intro}}
  
\end{figure}

These properties give us strong guarantees and broad applicability. An
experimenter may always choose to gather more samples, and may stop at any time
according to any rule---the resulting inferential guarantees hold under the
stated assumptions without any approximations. Of course, this flexibility comes
with a cost: our intervals are wider than those that rely on asymptotics or make
stronger assumptions, for example, a known stopping rule. Typical, fixed-sample
confidence intervals derived from the central limit theorem do not satisfy any
of (P1)-(P3), and accommodating any one property necessitates wider intervals;
we illustrate this in \cref{fig:intro}. It is perhaps surprising that
these four properties come at a numerical cost of less than doubling the fixed-sample,
asymptotic interval width---the discrete mixture bound illustrated in
\cref{fig:finite_lil_boundaries} stays within a factor of two of the
fixed-sample CLT bounds over five orders of magnitude in time.

\subsection{Related work}\label{sec:related_work}


The idea of a confidence sequence goes back at least to
\citet{darling_confidence_1967}. They are called \emph{repeated confidence
  intervals} by \citet{jennison_repeated_1984, jennison_interim_1989} (with a
focus on finite time horizons) and \emph{always-valid confidence intervals} 
by \citet{johari_always_2015}. They are sometimes labeled
\emph{anytime confidence intervals} in the machine learning literature
\citep{jamieson_bandit_2018}.

Prior work on sequential inference is often phrased in terms of a sequential
hypothesis test, defined as a stopping rule and an accept/reject decision
variable, or in terms of an always-valid $p$-value \citep{johari_always_2015}. In
\cref{sec:hypothesis_testing} we discuss the duality between confidence
sequences, sequential hypothesis tests, and always-valid $p$-values. We show in \cref{th:equiv_uniform_defns} that definition
\eqref{eq:conf_seq_defn} is equivalent to requiring
$\P(\theta_\tau \in \CI_\tau) \geq 1 - \alpha$ for all stopping times $\tau$, or
even for all random times $\tau$, not necessarily stopping times. Hence the
choice of definition \eqref{eq:conf_seq_defn} over related definitions in
the literature is one of convenience.

Recent interest in confidence sequences has come from the literature on best-arm
identification with fixed confidence for multi-armed bandit
problems. \Citet{garivier2013informational}, \citet{jamieson_lil_2014}, \citet{kaufmann_complexity_2014}, and
\citet{zhao_adaptive_2016} present methods satisfying properties (P1)-(P4) for
independent, sub-Gaussian observations.
Our results are sharper and more general, and our
Bernstein confidence sequence scales with the true variance
in nonparametric settings. Confidence sequences
are a key
ingredient in best-arm selection algorithms \citep{jamieson_best-arm_2014} and related methods for sequential testing with multiple
comparisons \citep{yang_framework_2017, malek_sequential_2017,
  jamieson_bandit_2018}. Our results improve and generalize such methods.

\Citet{maurer_empirical_2009} and \citet{audibert_explorationexploitation_2009}
prove empirical-Bernstein bounds for fixed times or finite time horizons. Our
empirical-Bernstein bound holds uniformly over infinite
time. \citet{balsubramani_sharp_2014} takes a different approach to deriving
confidence sequences satisfying properties (P1)-(P4) by lower bounding a mixture
martingale. This work was extended in \citet{balsubramani_sequential_2016} to an
empirical-Bernstein bound, the only infinite-horizon, empirical-Bernstein
confidence sequence we are aware of in prior work. Our result removes a
multiplicative pre-factor and yields sharper bounds. We emphasize that our proof
technique is quite different from all three of these existing
empirical-Bernstein bounds; see \cref{sec:proof_empirical_variance}.

The simplest confidence sequence satisfying properties (P1)-(P3) follows by
inverting a suitably formulated sequential probability ratio test (SPRT,
\citep{wald_sequential_1945}), such as in Section 3.6 of
\citet{howard_exponential_2018}. Wald worked in a parametric setting, though it
is known that the normal SPRT depends only on sub-Gaussianity (e.g.,
\citet{robbins_statistical_1970}). The resulting confidence sequence does not
shrink towards zero width as $t \to \infty$ (property P4), a problem which stems
from the choice of a single point alternative $\lambda$. Numerous extensions
have been developed to remedy this defect, and our work is most closely tied to
two approaches. First, in the method of mixtures, one replaces the likelihood
ratio with a mixture $\int \prod_i [f_\lambda(X_i) / f_0(X_i)] \d F(\lambda)$,
which is still a martingale \citep{ville_etude_1939, wald_sequential_1945,
  darling_further_1968, robbins_probability_1969, robbins_boundary_1970,
  robbins_statistical_1970, lai_confidence_1976,
  de_la_pena_pseudo-maximization_2007, balsubramani_sharp_2014,
  bercu_concentration_2015,kaufmann2018mixture}. Second, epoch-based analyses choose a sequence of
point alternatives $\lambda_1, \lambda_2, \dots$ approaching the null value,
with corresponding error probabilities $\alpha_1, \alpha_2, \dots$ approaching
zero so that a union bound yields the desired error control
\citep{darling_iterated_1967, robbins_iterated_1968, kaufmann_complexity_2014}.

The literature on self-normalized bounds makes extensive use of the method of
mixtures, sometimes called pseudo-maximization
\citep{de_la_pena_self-normalized_2004, de_la_pena_pseudo-maximization_2007,
  de_la_pena_theory_2009, de_la_pena_self-normalized_2009, garivier2013informational}; these works
introduced the idea of using a mixture to bound a quantity with a random
intrinsic time $V_t$. These results are mostly given for fixed samples or finite
time horizon, though \citet[Eq. 4.20]{de_la_pena_self-normalized_2004} includes
an infinite-horizon curve-crossing bound. \Citet{lai_confidence_1976} treats
confidence sequences for the parameter of an exponential family using mixture
techniques similar to those of \cref{sec:conjugate_mixtures}. Like most work
 on the method of mixtures, Lai's work focused on the parametric
setting (which we discuss in \cref{sec:one_param_exp}), while we focus on the
application of mixture bounds to nonparametric settings.

\citet{johari_peeking_2017} adopt the mixture approach for a commercial A/B
testing platform, where properties (P2) and (P3) are critical to provide an
``off-the-shelf'' solution for a variety of clients. Their application relies on
asymptotics which lack rigorous justification. In \cref{sec:ate} we give
nonasymptotic justification for a similar confidence sequence under a
finite-sample randomization inference model, and in \cref{sec:simulations} we
demonstrate how our methods control Type I error in situations where asymptotics
fail.

\subsection{Outline}

We organize our results using the sub-Gaussian, sub-gamma, sub-Bernoulli, sub-Poisson and
sub-exponential settings defined in \cref{sec:prelims}.
\begin{enumerate}
  
\item The \emph{stitching} method gives new closed-form sub-Gaussian or sub-gamma
  boundaries 
  (\cref{th:stitching}). Our sub-gamma treatment
  extends prior sub-Gaussian work to cover any martingale whose increments have
  finite moment-generating function in a neighborhood of zero; see
  \cref{th:universal}. Our proof is transparent and flexible, accommodating
  a variety of boundary shapes, including those growing at the rate
  $\Ocal(\sqrt{t \log \log t})$ with a focus on tight constants, 
  though we do not recommend this bound in practice
  unless closed-form simplicity is vital.
\item \emph{Conjugate mixtures} give one- and two-sided boundaries for the
  sub-Bernoulli, sub-Gaussian, sub-Poisson and sub-exponential cases
  (\cref{sec:conjugate_mixtures}) which avoid approximations made for
  analytical convenience. The sub-Gaussian boundaries are unimprovable without
  further assumptions (\cref{sec:admissibility}). These boundaries include a
  common tuning parameter which is critical in practice and we
    discuss why their $\Ocal(\sqrt{t \log t})$ growth rate may be preferable to
    the slower $\Ocal(\sqrt{t \log \log t})$ rate (\cref{sec:tuning}).
\item \emph{Discrete mixtures} facilitate numerical computation of
  boundaries with a great deal of flexibility, at the cost of slightly more
  involved computations (\cref{th:discrete_mixture}). Like conjugate mixture
  boundaries, these boundaries avoid unnecessary approximations and are
  unimprovable in the sub-Gaussian case.
\item Finally, for sub-Gaussian processes, the \emph{inverted stitching} method
  (\cref{th:inverted_stitching}) gives numerical upper bounds on the crossing
  probability of \emph{any} increasing, strictly concave boundary over a limited
  time range. We show that any such boundary yields a uniform upper tail
  inequality over a finite horizon, and compute its crossing probability.
\end{enumerate}

Building on this foundation, we present a a state-of-the-art empirical-Bernstein
bound (\cref{th:empirical_variance}) for any sequence of bounded observations
using a new self-normalization proof technique.  We illustrate our methods with
two novel applications: the nonasymptotic, sequential estimation of average
treatment effect in the Neyman-Rubin potential outcomes model (\cref{sec:ate}),
and the derivation of uniform matrix bounds and covariance matrix confidence
sequences (\cref{th:matrix_finite_lil,sec:matrix_application}). We give
simulation results in \cref{sec:simulations}.  \Cref{sec:hypothesis_testing}
discusses the relationship of our work to existing concepts of sequential
testing.  
Proofs of main results are in \cref{sec:main_proofs}, with others
deferred to \cref{sec:other_proofs}.

\section{Preliminaries: linear boundaries}\label{sec:prelims}

Given a sequence of real-valued observations $(X_t)_{t=1}^\infty$, suppose we
wish to estimate the average conditional expectation
$\mu_t \defineas t^{-1} \sum_{i=1}^t \E_{i-1} X_i$ at each time $t$ using the
sample mean $\Xbar_t \defineas t^{-1} \sum_{i=1}^t X_i$; here we assume an
underlying filtration $(\Fcal_t)_{t=1}^\infty$ to which $(X_t)$ is adapted, and
$\E_t$ denotes expectation conditional on $\Fcal_t$. Let
$S_t \defineas \sum_{i=1}^t (X_i - \E_{i-1} X_i)$, the zero-mean deviation of
our sample sum from its estimand at time $t$. Given $\alpha \in (0,1)$, suppose
we can construct a uniform upper tail bound
$u_\alpha: \R_{\geq 0} \to \R_{\geq 0}$ satisfying
\begin{align}
\P\big(\exists t \geq 1 : S_t \geq u_\alpha(V_t)\big) \leq \alpha
\label{eq:uniform_boundary}
\end{align}
for some adapted, real-valued \emph{intrinsic time} process
$(V_t)_{t=1}^\infty$, an appropriate time scale to measure the (squared)
deviations of $(S_t)$. This uniform upper bound on the centered sum $(S_t)$
yields a lower confidence sequence for $(\mu_t)$ with radius
$t^{-1} u_\alpha(V_t)$:
$\P\eparen{\forall t \geq 1 : \Xbar_t - t^{-1} u_\alpha(V_t) \leq \mu_t} \geq 1
- \alpha$.

Note that an assumption on the upper tail of $(S_t)$ yields a lower
confidence sequence for $(\mu_t)$; a corresponding assumption on the lower tail
of $(S_t)$ yields an upper confidence sequence for $(\mu_t)$. In this paper we
formally focus on upper tail bounds, from which lower tail bounds can be derived
by examining $(-S_t)$ in place of $(S_t)$. In general, the left and right tails
of $(S_t)$ may behave differently and require different sets of assumptions, so
that our upper and lower confidence sequences may have different
forms. Regardless, we can always combine upper and
lower confidence sequences using a union bound to obtain a two-sided confidence
sequence \eqref{eq:conf_seq_defn}.

When the $(X_t)$ are independent with common mean $\mu$, the resulting
confidence sequence estimates $\mu$, but the setup requires neither independence
nor a common mean. In general, the estimand $\mu_t$ may be changing at each time
$t$; \cref{sec:ate} gives an application to causal inference in which this
changing estimand is useful. In principle, $\mu_t$ may also be random, although
none of our applications involve random $\mu_t$.

To construct uniform boundaries $u_\alpha$ satisfying inequality
\eqref{eq:uniform_boundary}, we build upon the following general condition
\citep[Definition 1]{howard_exponential_2018}:

\begin{definition}[Sub-$\psi$ condition]
\label{th:canonical_assumption}
Let $\smash{(S_t)_{t = 0}^\infty, (V_t)_{t = 0}^\infty}$ be real-valued processes
adapted to an underlying filtration $(\Fcal_t)_{t = 0}^\infty$ with $\smash{S_0=V_0=0}$
and~$\smash{V_t\geq0}$ for all $t$. For a function
$\psi: [0, \lambda_{\max}) \to \R$ and a scalar $l_0 \in [1, \infty)$, we say
$(S_t)$ is \emph{$l_0$-sub-$\psi$ with variance process $(V_t)$} if, for each
$\lambda \in [0, \lambda_{\max})$, there exists a supermartingale
$(L_t(\lambda))_{t = 0}^\infty$ w.r.t. $(\Fcal_t)$ such that
$\E L_0(\lambda) \leq l_0$ and
  \begin{align}
    \expebrace{\lambda S_t - \psi(\lambda) V_t } \leq L_t(\lambda)
    \text{ a.s.\ for all } t.
  \end{align}
  For given $\psi$ and $l_0$, let $\subpsiclass$ be the class of pairs of
  $l_0$-sub-$\psi$ processes $(S_t, V_t)$:
  \begin{align}
    \subpsiclass \defineas \ebrace{
      (S_t, V_t):
      (S_t) \text{ is $l_0$-sub-$\psi$ with variance process } (V_t)}.
  \end{align}
\end{definition}
When stating that a process is sub-$\psi$, we typically omit $l_0$ from our
terminology for simplicity. In scalar
cases, we always have $l_0 = 1$, while in matrix cases $l_0 = d$, the dimension
of the (square) matrices.

Where does \cref{th:canonical_assumption} come from?  The jumping-off point is
the martingale method for concentration inequalities
(\citep{hoeffding_probability_1963, azuma_weighted_1967,
  mcdiarmid_concentration_1998}; \citep[section
2.2]{raginsky_concentration_2012}), itself based on the classical
Cram\'er-Chernoff method (\citep{cramer_sur_1938, chernoff_measure_1952};
\citep[section 2.2]{boucheron_concentration_2013}). The martingale method starts
off with an assumption of the form
$\E_{t-1}e^{\lambda(X_t - \E_{t-1} X_t)}\leq e^{\psi(\lambda)\sigma_t^2}$
for all $t \geq 1, \lambda \in \R$. Then, denoting
$S_t \defineas \sum_{i=1}^t (X_i - \E_{i-1} X_i)$ and
$V_t \defineas \sum_{i=1}^t \sigma_i^2$, the process
$\expebrace{\lambda S_t - \psi(\lambda) V_t}$ is a supermartingale for each
$\lambda \in \R$. Unlike the martingale method assumption,
\cref{th:canonical_assumption} allows the exponential process to be upper
bounded by a supermartingale, and it permits $(V_t)$ to be adapted rather than
predictable. We also restrict our attention to $\lambda \geq 0$ to derive
one-sided bounds.

Intuitively, the process $\expebrace{\lambda S_t - \psi(\lambda) V_t}$ measures
how quickly $S_t$ has grown relative to intrinsic time $V_t$, and the free
parameter $\lambda$ determines the relative emphasis placed on the tails of the
distribution of $S_t$, i.e., on the higher moments. Larger values of $\lambda$
exaggerate larger movements in $S_t$, and $\psi$ captures how much we must
correspondingly exaggerate $V_t$.  $\psi$ is related to the heavy-tailedness of
$S_t$ and the reader may think of it as a cumulant-generating function (CGF, the
logarithm of the moment-generating function). For example, suppose $(X_t)$ is a
sequence of i.i.d., zero-mean random variables with CGF
$\psi(\lambda) \defineas \log \E e^{\lambda X_1}$ which is finite for all
$\lambda \in [0, \lambda_{\max})$. Then, setting $V_t \defineas t$, we see that
$L_t(\lambda) \defineas \expebrace{\lambda S_t - \psi(\lambda) V_t}$ is itself a
martingale, for all $\lambda \in [0, \lambda_{\max})$. Indeed, in all scalar
cases we consider, $L_t(\lambda)$ is just equal to
$\expebrace{\lambda S_t - \psi(\lambda) V_t}$. See
Appendix~\Cref{tab:scalar_suff_cond,tab:matrix_suff_cond}, drawn from
\citet{howard_exponential_2018}, for a catalog of sufficient conditions for a
process to be sub-$\psi$ using the five $\psi$ functions defined below. We use
many of these conditions in what follows.

We organize our uniform boundaries according to the $\psi$ function used in
\cref{th:canonical_assumption}. First recall
the Cram\'er-Chernoff bound: if $(X_t)$ are independent zero-mean with bounded
CGF $\log \E e^{\lambda X_t} \leq \psi(\lambda)$ for all $t \geq 1$ and
$\lambda \in \R$, then writing $S_t = \sum_{i=1}^t X_i$, we have
$\P(S_t \geq x) \leq e^{-t \psi^\star(x/t)}$ for any $x > 0$, where $\psi^\star$
denotes the Legendre-Fenchel transform of $\psi$. Equivalently, writing
$z_\alpha(t) \defineas t {\psi^\star}^{-1}(t^{-1} \log \alpha^{-1})$, we have
$\P(S_t \geq z_\alpha(t)) \leq \alpha$ for any fixed $t$ and $\alpha \in
(0,1)$. In other words, the function $z_\alpha$ gives a high-probability upper
bound at any fixed time $t$ for \emph{any} sum of independent random variables
with CGF bounded by $\psi$. When we extend this concept to boundaries holding
uniformly over time, there is no longer a unique, minimized boundary, and the
following definition captures the class of valid boundaries.

\begin{definition}\label{th:uniform_boundary}
  Given $\smash{\psi: [0, \lambda_{\max}) \to \R}$ and
  $\smash{l_0\geq 1}$, a function~$\smash{u:\R\to\R}$ is called an
  \emph{$l_0$-sub-$\psi$ uniform boundary} with crossing probability
  $\alpha$ if
  \begin{align}
    \sup_{(S_t,V_t) \in \subpsiclass} \P(\exists t \geq 1: S_t \geq u(V_t))
      \leq \alpha.
  \end{align}
  
\end{definition}

Although $u$ does depend on the constant $l_0$ in
\cref{th:canonical_assumption}, for simplicity we typically omit this dependence
from our notation, writing simply that $u$ is a sub-$\psi$ uniform boundary.

Five particular $\psi$ functions play important roles in our development; below,
we take $1 / 0 = \infty$ in the upper bounds on $\lambda$:
\begin{itemize}
\item
  $\psi_{B,g,h}(\lambda) \defineas \frac{1}{gh} \log\pfrac{ge^{h\lambda} +
    he^{-g\lambda}}{g+h}$ on $0 \leq \lambda < \infty$, the scaled CGF of a
  centered random variable (r.v.) supported on two points, $-g$ and $h$, for some
  $g,h > 0$, for example a centered Bernoulli r.v. when
  $g + h = 1$.
\item $\psi_N(\lambda) \defineas \lambda^2/2$ on $0 \leq \lambda < \infty$, the
  CGF of a standard Gaussian r.v.
\item $\psi_{P,c}(\lambda) \defineas c^{-2}(e^{c\lambda} - c\lambda - 1)$ on
  $0 \leq \lambda < \infty$ for some scale parameter $c \in \R$, which is the
  CGF of a centered unit-rate Poisson r.v. when $c=1$. By taking the limit, we
  define $\psi_{P,0} = \psi_N$.
\item $\psi_{E,c}(\lambda) \defineas c^{-2}(-\log(1-c\lambda) - c\lambda)$ on
  $0 \leq \lambda < 1 / (c \bmax 0)$ for some scale $c \in \R$, which is the CGF
  of a centered unit-rate exponential r.v. when $c = 1$. By taking the limit, we
  define $\psi_{E,0} = \psi_N$.
\item $\psi_{G,c}(\lambda) \defineas \lambda^2 / (2 (1 - c\lambda))$ on
  $0 \leq \lambda < 1 / (c \bmax 0)$ (taking $1/0 = \infty$) for some scale
  parameter $c \in \R$, which we refer to as the sub-gamma case following
  \citet{boucheron_concentration_2013}. This is not the CGF of a gamma r.v. but
  is a convenient upper bound which also includes the sub-Gaussian case at
  $c = 0$ and permits analytically tractable results below.
\end{itemize}

\begin{figure}[h!]
  \centering
  \begin{tikzpicture}[node distance = 4.5cm, auto]
    \node[block2] (ber) {Sub-Bernoulli};
    \node[block2, below of=ber, node distance=1.4cm] (norm) {Sub-Gaussian};
    \node[block2, right of=ber] (poi) {Sub-Poisson};
    \node[block2, right of=poi] (gam) {Sub-gamma};
    \node[block2, below of=gam, node distance=1.4cm] (exp)
      {Sub-exponential};
    \path [line] (norm) -- (ber);
    \path [line] (poi.170) -- (ber.10);
    \path [line] (ber.350) -- node[below] {$c < 0$} (poi.190);
    \path [line] (poi.220) to[out=270, in=0] (norm.5);
    \path [line] (norm.350) to[out=0, in=270] node[below right] {$c < 0$}
      (poi.270);
    \path [line] (gam.170) -- (poi.10);
    \path [line] (poi.350) -- node[below] {$c < 0$} (gam.190);
    \path [line] (exp.60) -- (gam.300);
    \path [line] (gam.240) -- (exp.120);
  \end{tikzpicture}
  \caption{
    Relations among sub-$\psi$ boundaries: each arrow indicates that a
    sub-$\psi$ boundary at the source node can also serve as a sub-$\psi$
    boundary at the destination node, with appropriate modifications to their
    parameters. Details are in
    \cref{th:psi_boundary_relations}. \label{fig:psi_boundary_relations}
    }
\end{figure}
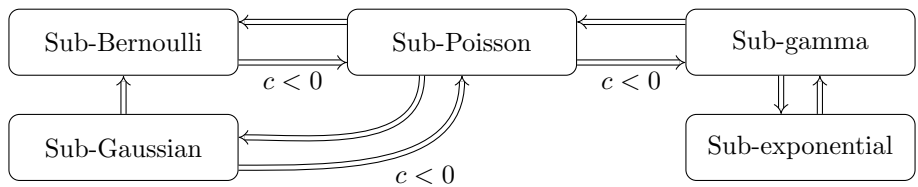

One may freely scale $\psi$ by any positive constant and divide $V_t$ by the
same constant so that \cref{th:canonical_assumption} remains satisfied; by
convention, we scale all $\psi$ functions above so that $\psi''(0_+) = 1$. When
we speak of a \emph{sub-gamma} process (or uniform boundary) with scale parameter
$c$, we mean a sub-$\psi_{G,c}$ process (or uniform boundary), and likewise for
 other cases. We often write $\psi_B$, $\psi_P$, etc., dropping the range and
scale parameters from our notation. As we summarize in
\cref{fig:psi_boundary_relations} and detail in
\cref{th:psi_boundary_relations}, certain general implications hold among
sub-$\psi$ boundaries. In particular, any sub-Gaussian boundary can also serve
as a sub-Bernoulli boundary; any sub-Poisson boundary serves as a sub-Gaussian
or sub-Bernoulli boundary; and, importantly, any sub-gamma or sub-exponential
boundary can serve as a sub-$\psi$ boundary in any of the other four cases.
Indeed, a sub-gamma or sub-exponential boundary applies to many cases of
practical interest, as detailed below.  
\begin{proposition}\label{th:universal}
  Suppose $\psi$ is twice-differentiable and $\smash{\psi(0) = \psi'(0_+) =
  0}$. Suppose, for each $c>0$, $u_c(v)$ is a sub-gamma or sub-exponential
  uniform boundary with crossing probability $\alpha$ for scale $c$. Then
  $v \mapsto u_{k_1}(k_2v)$ is a sub-$\psi$ uniform boundary for some constants
  $k_1,k_2 > 0$ depending only on $\psi$.
\end{proposition}

\Cref{th:universal} restates
\citet[Proposition 1]{howard_exponential_2018}, which shows that any process $(S_t)$ which is
sub-$\psi$ is also sub-gamma and sub-exponential, if $\psi$ satisfies
the conditions of \cref{th:universal}. Note that these conditions are satisfied 
for any mean-zero random variable if the CGF exists in a neighborhood of zero, 
so the conditions are quite weak
\citep[Theorem 2.3]{jorgensen_theory_1997}.

\begin{example}[Confidence sequence for the variance of a Gaussian distribution
  with unknown mean] \label{ex:variance} Suppose $X_1, X_2, \dots$ are i.i.d.\
  draws from a $\Normal(\mu, \sigma^2)$ distribution and we wish to sequentially
  estimate $\sigma^2$ when $\mu$ is also unknown. Let
  $S_t \defineas \sigma^{-2} \sum_{i=1}^{t+1} (X_i - \Xbar_{t+1})^2 - t$ for
  $t = 1, 2, \dots$, where $\Xbar_t \defineas t^{-1} \sum_{i=1}^t X_i$ is the
  sample mean. This $S_t$ is a centered and scaled sample variance, and as in
  \citet{darling_confidence_1967}, we use the fact that $S_t$ is a cumulative
  sum of independent, centered Chi-squared random variables each with one degree
  of freedom (see \cref{sec:example_details} for details). Such a centered
  Chi-squared distribution has variance two and CGF equal to $2
  \psi_{E,2}$. 

  Thus $(S_t)$ is 1-sub-exponential with variance process
  $V_t = 2t$ and scale parameter $c = 2$. We may uniformly bound the upper
  deviations of $S_t$ using any sub-exponential uniform boundary, for example
  the gamma-exponential mixture boundary of \cref{th:gamma_mixture}. Or, we can
  use \cref{th:psi_boundary_relations} to deduce that $(S_t)$ is sub-gamma with
  scale $c = 2$ (and the same variance process) and use the closed-form stitched
  boundary of \cref{th:stitching}.

  The above constructions yield lower confidence sequences for the
  variance. To obtain an upper confidence sequence, we use the fact that
  $(-S_t)$ is 1-sub-exponential with scale parameter $\smash{c = -2}$. Now
  \cref{th:psi_boundary_relations} implies that $(-S_t)$ is sub-gamma with scale
  $c = -1$, so the stitched boundary again applies, while
  \cref{th:psi_boundary_relations} implies that $(-S_t)$ is also sub-Gaussian,
  so we may alternatively use the normal mixture boundary of
  \cref{th:normal_mixture}. Since $\psi_{G,-1}$ is uniformly smaller
  than $\psi_N$, the above analysis yields tighter bounds than the
  sub-Gaussian approach of \citet{darling_confidence_1967}.
\end{example}

The simplest uniform boundaries are linear with positive intercept and slope.
This is formalized in \citet{howard_exponential_2018}, partially
restated below.

\begin{lemma}[\citep{howard_exponential_2018}, Theorem 1]
\label{th:uniform_chernoff}
For any $\lambda \in [0, \lambda_{\max})$ and $\alpha \in (0,1)$,
\begin{align}
  u(v) \defineas \frac{\log(l_0 / \alpha)}{\lambda}
    + \frac{\psi(\lambda)}{\lambda} \cdot v
\end{align}
is a sub-$\psi$ uniform boundary with crossing probability $\alpha$.
\end{lemma}

While \cref{th:uniform_chernoff} provides a versatile building block, the $\Ocal(V_t)$ growth of $u(V_t)$ may be undesirable. Indeed, from a concentration point of
view, the typical deviations of $S_t$ tend to be only $\Ocal(\sqrt{V_t})$, 
so the bound will rapidly become loose for large $t$. From a confidence sequence point of view,
 recall that the confidence radius for the mean is given by
  $u(V_t) / t$. Typically, $V_t = \Theta(t)$ a.s. as $t \to \infty$, so the
  confidence radius will be asymptotically zero width if and only if
  $u(v) = o(v)$. In other words, we cannot achieve arbitrary estimation
  precision with arbitrarily large samples unless the uniform boundary is
  sublinear. We address this problem in \cref{sec:main_results}, building upon
\cref{th:uniform_chernoff} to construct \emph{curved} sub-$\psi$ uniform
boundaries.

\section{Curved uniform boundaries}\label{sec:main_results}

We present our four methods for computing curved uniform boundaries in
\cref{sec:stitching,sec:conjugate_mixtures,sec:discrete_mixture,sec:inverted_stitching}. In
\cref{sec:tuning}, we discuss how to tune boundaries, a necessity for good
performance in practice, and we describe the unimprovability of sub-Gaussian
mixture bounds in \cref{sec:admissibility}.

\subsection{Closed-form boundaries via stitching}\label{sec:stitching}
Our analytical ``stitched'' bound is useful in the sub-Gaussian case or, more
generally, the sub-gamma case with scale $c$. We require three user-chosen
parameters:
\begin{itemize} 
\item a scalar $\eta > 1$ determines the geometric spacing of intrinsic time,
\item a scalar $m > 0$ which gives the intrinsic time at which the uniform
  boundary starts to be nontrivial, and
\item an increasing function $h : \R_{\geq 0} \to \R_{>0}$  such that
  ${\sum_{k=0}^\infty 1 / h(k) \leq 1}$, which determines the shape of the
  boundary's growth after time $m$.
\end{itemize}
Recalling the scale parameter $c$ for the $\psi_G$ function above and the
constant $l_0$ in \cref{th:canonical_assumption}, we define the stitching
function $\Scal_\alpha$ as
\begin{equation}
\small
\Scal_\alpha(v) \defineas
  \sqrt{k_1^2 v \ell(v) + k_2^2 c^2 \ell^2(v)} \\
  + k_2 c \ell(v),
\text{ where }
\begin{cases}
\ell(v) \defineas \log h(\log_{\eta}(\frac{v}{m})) + \log(\frac{l_0}{\alpha}),\\
k_1 \defineas (\eta^{1/4} + \eta^{-1/4}) / \sqrt{2},\\
k_2 \defineas (\sqrt{\eta} + 1) / 2,
\end{cases}
\label{eq:stitching_operator}
\end{equation}
and define the stitched boundary as $u(v) = \Scal_\alpha(v \bmax m)$. Note
$\Scal_\alpha(v) \leq k_1 \sqrt{v \ell(v)} + 2 c k_2 \ell(v)$ when $c > 0$,
while $\Scal_\alpha(v) \leq k_1 \sqrt{v \ell(v)}$ when $c \leq 0$, with equality
in the sub-Gaussian case ($c = 0$). These simpler expressions may sometimes be
preferable. For notational simplicity we suppress the dependence of
$\Scal_\alpha$ on $h$, $\eta$, $l_0$, and $c$; we will discuss specific choices
as necessary. In the examples we consider, $\ell(v)$ grows as $\Ocal(\log v)$ or
$\Ocal(\log \log v)$ as $v \uparrow \infty$, so the first term,
$k_1 \sqrt{V_t \ell(V_t)}$, dominates for sufficiently large $V_t$, specifically
when $V_t / \ell(V_t) \gg 2c^2\sqrt{\eta}$.

\begin{theorem}[Stitched boundary]\label{th:stitching}
  For any $\smash{c \geq 0, \alpha \in (0,1), \eta > 1, m>0}$, and
  $\smash{h : \R_{\geq 0} \to \R_{\geq 0}}$ increasing such that
  $\smash{\sum_{k=0}^\infty 1 / h(k) \leq 1}$, the function
  $v \mapsto \Scal_\alpha(v \bmax m)$ is a sub-gamma uniform boundary with
  crossing probability $\alpha$.  Further, for any sub-$\psi_G$ process
  $(S_t)$ with variance process $(V_t)$ and any $v_0 \geq m$,
  \begin{align}
    \P\eparen{
      \exists t \geq 1: V_t \geq v_0 \text{ and } S_t \geq \Scal_\alpha(V_t)}
    \leq \sum_{k=\floor{\log_\eta(v_0/m)}}^\infty \frac{\alpha}{h(k)}.
    \label{eq:late_crossing_decay}
  \end{align}
\end{theorem}

\begin{figure}
\centering
\includegraphics{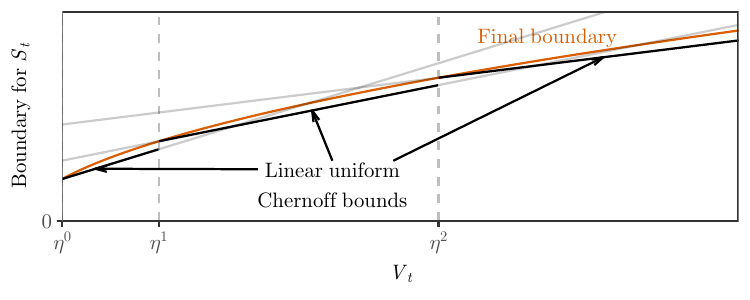}
\caption{
Illustration of \cref{th:stitching}, stitching together linear
  boundaries to construct a curved boundary. We break time into
  geometrically-spaced epochs $\eta^k \leq V_t < \eta^{k+1}$, construct a linear
  uniform bound using \cref{th:uniform_chernoff} optimized for each epoch, and
  take a union bound over all crossing events. The final boundary is a smooth
  analytical upper bound to the piecewise linear bound. \label{fig:stitching}}
  
\end{figure}

The first sentence above says that the probability of $S_t$ crossing
$\Scal_\alpha(V_t \bmax m)$ at least once is at most $\alpha$, while the second
says that, even if it does happen to cross once or more, the probability of
further crossings decays to zero beyond larger and larger intrinsic times
$v_0$. Note that \eqref{eq:late_crossing_decay} implies
$\P\eparen{\sup_t V_t = \infty \text{ and } S_t \geq \Scal_\alpha(V_t) \text{
    infinitely often}} = 0$. The proof of \cref{th:stitching},
  given with discussion in \cref{sec:proof_stitching}, follows by taking a union
  bound over a carefully chosen family of linear boundaries, one for each of a
  sequence of geometrically-spaced epochs; see \cref{fig:stitching}. The
  high-level proof technique is standard, often referred to as ``peeling'' in
  the bandit literature, and closely related to chaining elsewhere in
  probability theory. Our proof generalizes beyond the sub-Gaussian case and
  involves careful parameter choices in order to achieve tight constants. In
  brief, within each epoch, there are many possible linear boundaries, and we
  have found that optimizing the linear boundary for the geometric mean of the
  epoch endpoints strikes a good balance between tight constants and analytical
  simplicity in the final boundary.  \Cref{sec:finite_lil_details} gives a
  detailed comparison of constants arising from our bound with similar bounds
  from the literature.

 The boundary shape is determined by choosing the function $h$
  and setting the nominal crossing probability in the \kth epoch to equal
  $\alpha/h(k)$. Then \cref{th:stitching} gives a curved boundary which grows at
  a rate $\Ocal\eparen{\sqrt{V_t \log h(\log_\eta V_t)}}$ as
  $V_t \uparrow \infty$.  The more slowly $h(k)$ grows as $k \uparrow \infty$,
  the more slowly the resulting boundary will grow as $V_t \uparrow \infty$. A
  simple choice is exponential growth, $h(k) = \eta^{sk} / (1 - \eta^{-s})$ for
  some $s > 1$, yielding $\Scal_\alpha(v) = \Ocal(\sqrt{v \log v})$. A more
interesting example is $h(k) = (k + 1)^s \zeta(s)$ for some $s > 1$, where
$\zeta(s)$ is the Riemann zeta function. Then, when $l_0 = 1$,
\cref{th:stitching} yields the \emph{polynomial stitched boundary}: for
$c \geq 0$,
\begin{align}
  \Scal_\alpha(v) =
    k_1 \sqrt{
      v \eparen{s \log\log\pfrac{\eta v}{m}
      + \log \frac{\zeta(s)}{\alpha \log^s \eta}}
    }
    + c k_2
      \eparen{s \log\log\pfrac{\eta v}{m}
      + \log \frac{\zeta(s)}{\alpha \log^s \eta}},
  \label{eq:poly_stitching}
\end{align}
where the second term is neglected in the sub-Gaussian case since $c=0$. This is
a ``finite LIL bound'', so-called because
$\Scal_\alpha(v) \sim \sqrt{s k_1^2 v \log \log v}$, matching the form of the
law of the iterated logarithm
\citep{stout_hartman-wintner_1970}. We can bring $s k_1^2$
  arbitrarily close to $2$ by choosing $\eta$ and $s$ sufficiently close to one,
  at the cost of inflating the additive term $\log(\zeta(s) / (\log^s
  \eta))$. Briefly, increasing $\eta$ increases the size of each epoch in the
  aforementioned peeling argument, which reduces the looseness of the union
  bound over epochs. But the larger we make the epochs, the further each linear
  boundary deviates from the ideal curved shape at the ends of the epochs, which
  inflates our final boundary. The choice of $s$ involves a similar tradeoff:
  increasing $s$ causes us to exhaust more of our total error probability budget
  on earlier epochs, decreasing the constant term (which matters most for early
  times), at the cost of a union bound over smaller error probabilities in later
  epochs, which shows up as an increase in the leading constant. We discuss
  parameter tuning in more practical terms in \cref{sec:tuning}. For example,
take $\smash{\eta = 2, s = 1.4, m = 1}$; if $S_t$ is a sum of independent,
zero-mean, 1-sub-Gaussian observations, we obtain
\begin{align}
\P\eparen{
  \exists t \geq 1 :
  S_t \geq 1.7 \sqrt{t \eparen{\log \log(2t) + 0.72 \log \pfrac{5.2}{\alpha}}}
} \leq \alpha. \label{eq:poly_stitching_explicit}
\end{align}

\Cref{fig:finite_lil_boundaries} in \cref{sec:finite_lil_details} compares a
sub-Gaussian stitched boundary to a numerically-computed discrete mixture bound
with a mixture distribution roughly corresponding to $h(k) \propto (k+1)^{1.4}$,
as described in \cref{sec:stitching_mixture}. This discrete mixture boundary
acts as a lower bound (see \cref{sec:admissibility}) and shows that not too much
is lost by the approximations involved in the stitching construction.
\Cref{fig:finite_lil_all} compare the same stitched boundary to related bounds
from the literature; our bound shows slightly improved constants over the best
known bounds.

Although our stitching construction begins with a sub-gamma assumption, it
applies to other sub-$\psi$ cases, including sub-Bernoulli, sub-Poisson and
sub-exponential cases; see \cref{fig:psi_boundary_relations} and
\cref{th:universal}. Further, our stitched bounds apply equally well
in continuous-time settings to Brownian motion, continuous martingales,
martingales with bounded jumps, and martingales whose jumps satisfy a Bernstein condition; 
see \cref{th:continuous_time}.

While our focus is on nonasymptotic results, \cref{th:stitching} makes it easy
to obtain the following general upper asymptotic LIL, proved in
\cref{sec:proof_asymptotic_lil}:
\begin{corollary}\label{th:asymptotic_lil}

  Suppose $(S_t)$ is sub-$\psi$ with variance process $(V_t)$ and
  $\psi(\lambda) \sim \lambda^2 / 2$ as $\lambda \downarrow 0$. Then
  \begin{align}
    \limsup_{t \to \infty} \frac{S_t}{\sqrt{2 V_t \log \log V_t}} \leq 1
    \quad \text{on } \ebrace{\sup_t V_t = \infty}.
  \end{align}
\end{corollary}

\subsection{Conjugate mixture boundaries}\label{sec:conjugate_mixtures}

For appropriate choice of mixing distribution $F$, the integral
$\int \expebrace{\lambda S_t - \psi(\lambda) V_t} \d F(\lambda)$ will be
analytically tractable. Since, under \cref{th:canonical_assumption}, this
mixture process is upper bounded by a mixture supermartingale
$\int L_t(\lambda) \d F(\lambda)$, such mixtures yield closed-form or
efficiently computable curved boundaries, which we call conjugate mixture
boundaries. This approach is known as the method of mixtures, one of the most
widely-studied techniques for constructing uniform bounds
\citep{ville_etude_1939, wald_sequential_1945, darling_further_1968,
  robbins_statistical_1970, robbins_probability_1969, robbins_boundary_1970,
  lai_confidence_1976,kaufmann2018mixture}. Unlike the stitched bound of
\cref{th:stitching}, which involves a small amount of looseness in the
analytical approximations, mixture boundaries involve no such approximations
and, in the sub-Gaussian case, are unimprovable in the sense described in
\cref{sec:admissibility}. We restate the following standard idea behind the
method of mixtures using our definitions, with a proof in
\cref{sec:proof_mixtures}. The proof details a technical condition on product
measurability which we require of $L_t$.

\begin{lemma}\label{th:basic_mixture}
  For any probability distribution $F$ on $[0, \lambda_{\max})$ and
  $\alpha \in (0,1)$,
\begin{align}
\Mcal_\alpha(v) := \sup\Bigg\lbrace
  s \in \R: \underbrace{\int \expebrace{\lambda s - \psi(\lambda) v} \d F(\lambda)}_{\defineright m(s,v)}
  < \frac{l_0}{\alpha}
\Bigg\rbrace \label{eq:mixture_operator}
\end{align}
is a sub-$\psi$ uniform boundary with crossing probability $\alpha$, so long as
the supermartingale $(L_t)$ of \cref{th:canonical_assumption} is product
measurable when the underlying probability space is augmented with the
independent random variable $\lambda$.
\end{lemma}

For each of our conjugate mixture bounds, we compute $m(s, v)$ in closed-form. The boundary $u(v)$ can then be computed by numerically solving the
equation $m(s, v) = l_0 / \alpha$ in $s$, as we show in
\cref{sec:root_finding}. When an identical sub-$\psi$ condition applies to
$(-S_t)$ as well as $(S_t)$, we may apply a uniform boundary to both tails and
take a union bound, obtaining a two-sided confidence sequence. However, mixing
over $\lambda \in \R$ rather than $\lambda \in \R_{\geq 0}$ yields a two-sided
bound directly, so in some cases we present two-sided variants along with their
one-sided counterparts. We give details for the following conjugate mixture
boundaries in \cref{sec:proof_mixtures}:
\begin{itemize}
\item one-, two-sided \emph{normal mixture} boundaries
  (sub-Gaussian case);
\item one-, two-sided \emph{beta-binomial mixture} boundaries (sub-Bernoulli case);
\item one-sided \emph{gamma-Poisson mixture} boundary (sub-Poisson
  case); and
\item one-sided \emph{gamma-exponential mixture} boundary (sub-exponential case).
\end{itemize}

The two-sided normal mixture boundary has a closed form
expression:
\begin{align}
  u(v) \defineas \sqrt{
  (v + \rho)
  \log\eparen{\frac{l_0^2 (v + \rho)}{\alpha^2 \rho}}
  }. \label{eq:two_sided_normal_mixture}
\end{align}
The one-sided normal mixture boundary has a similar, closed-form upper bound,
making these especially convenient. It is clear from
\eqref{eq:two_sided_normal_mixture} that the normal mixture boundary grows as
$\Ocal(\sqrt{v \log v})$ asymptotically, and this rate is shared by all of our
conjugate mixture boundaries. Indeed, \cref{th:mixture_rate}
  below, proved in \cref{sec:proof_mixture_rate}, shows that such a rate holds
  for any mixture boundary as given by \eqref{eq:mixture_operator} whenever the
  mixing distribution is continuous with positive density at and around
  the origin, a property which holds for all mixture distributions used in our
  conjugate mixture boundaries, subject to regularity conditions on $\psi$ which
  hold for the CGF of any nontrivial, mean-zero r.v. and specifically
  for the five $\psi$ functions in \cref{sec:prelims}.
  
\begin{proposition}\label{th:mixture_rate}
  Assume $(i)$ $\psi$ is nondecreasing, $\psi(0) = \psi'(0_+) = 0$,
  $\psi''(0_+) = c > 0$, and $\psi$ has three continuous derivatives on a
  neighborhood including the origin; and $(ii)$ $F$ has density $f$ (w.r.t.
  Lebesgue) which is continuous and positive on a neighborhood including the
  origin. Then
  \begin{align}
    \Mcal_\alpha(v) = \sqrt{
      v \ebracket{c\log\pfrac{c l_0^2 v}{2\pi \alpha^2 f^2(0)} + o(1)}}
    \quad \text{as } v \to \infty.
    \label{eq:mixture_rate}
  \end{align}
\end{proposition}
Note that $f$ need not place mass on all of $[0, \lambda_{\max})$, only near the
origin, for the asymptotic rate to hold. \Cref{th:mixture_rate} shows how the
asymptotic behavior of any such mixture bound depends only on the behavior of
$\psi$ and $f$ near the origin, a result reminiscent of the central limit
theorem. Analogous, related results for the sub-Gaussian special case using
$\psi(\lambda) = \lambda^2 / 2$ can be found in \citet[Section
4]{robbins_boundary_1970} and \citet[Theorem 2]{lai_boundary_1976}, in some
cases under weaker assumptions on $F$.

In contrast to previous derivations of conjugate mixture
  boundaries in the literature, all of our conjugate mixture boundaries include a common tuning
parameter $\rho > 0$ which controls the sample size for which the boundary is
optimized. Such tuning is critical in practice, as we explain in
\cref{sec:tuning}, but has been ignored in much prior work. Additionally, with
the exception of the sub-Gaussian case, most prior work on the method of
mixtures has focused on parametric settings. We instead emphasize the
applicability of these bounds to nonparametric settings. For example, when the
observations are bounded, one may construct a confidence sequence making use of
empirical-Bernstein estimates (\cref{th:empirical_variance}) based on our
gamma-exponential mixture (\cref{th:gamma_mixture}). See
\cref{sec:reference_tables} for other conditions in which mixture bounds yield
nonparametric uniform boundaries.

\subsection{Numerical bounds using discrete mixtures}
\label{sec:discrete_mixture}

In applications, one may not need an explicit closed-form expression so long as
the bound can be easily computed numerically. Our discrete mixture method is an
efficient technique for numerical computation of curved boundaries for processes
satisfying \cref{th:canonical_assumption}. It permits arbitrary mixture
densities, thus producing boundaries growing at the rate
$\Ocal(\sqrt{v \log \log v})$.
Recall that the shape of the stitched bound was determined by the user-specified
function $h$. For the discrete mixture bound, one instead specifies a
probability density $f$ over finite support $(0, \overlambda]$ for some
$\overlambda \in (0, \lambda_{\max})$. We first discretize $f$ using a series of
support points $\lambda_k$, geometrically spaced according to successive powers
of some $\eta > 1$, and an associated set of weights $w_k$: 
\begin{align}
    \lambda_k \defineas \frac{\overlambda}{\eta^{k+1/2}}
    \quad \text{and} \quad
    w_k \defineas
      \frac{\overlambda (\eta - 1) f(\lambda_k\sqrt{\eta})}{\eta^{k+1}}
    \quad \text{for} \quad k = 0, 1, 2, \dots.
  \label{eq:discretization}
\end{align}
\begin{theorem}[Discrete mixture bound]\label{th:discrete_mixture}
  Fix $\psi: [0, \lambda_{\max}) \to \R$, $\alpha~\in~(0,1)$,
  $\overlambda \in (0, \lambda_{\max})$, and a probability density $f$ on
  $(0, \overlambda]$ that is nonincreasing and positive. For supports
  $\lambda_k$ and weights $w_k$ defined in \eqref{eq:discretization},
  \begin{equation}
    \Mtil_\alpha(v) \defineas \sup\ebrace{
      s \in \R :
      \sum_{k=0}^\infty w_k \expebrace{\lambda_k s - \psi(\lambda_k) v}
      < \frac{l_0}{\alpha}
    }, \label{eq:discrete_mixture_sum}
  \end{equation}
  is a sub-$\psi$ uniform boundary with crossing probability
  $\alpha$.
\end{theorem}

We suppress the dependence of $\Mtil_\alpha$ on $f$, $l_0$, $\overlambda$ and
$\eta$ for notational simplicity. Though \cref{th:discrete_mixture} is a
straightforward consequence of the method of mixtures, our choice of
discretization \eqref{eq:discretization} makes it effective, broadly applicable, and easy to
implement. See \cref{sec:proof_discrete_mixture} for the proof of this
result. \Cref{fig:finite_lil_boundaries} includes an example bound,
demonstrating a slight advantage over stitching.
\cref{sec:stitching_mixture} describes a connection between the stitching and
discrete mixture methods, including a correspondence between the alpha-spending
function $h$ and the mixture density $f$. Finally, we note that the method can be
applied even when $f$ is not monotone; one must simply choose the
discretization \eqref{eq:discretization} more carefully, using known
properties of $f$.

\subsection{Inverted stitching for arbitrary boundaries}
\label{sec:inverted_stitching}

In the method of mixtures, we choose a mixing distribution $F$ and the machinery
yields a boundary $\Mcal_\alpha$. Likewise, in the stitching construction of
\cref{th:stitching}, we choose an error decay function $h$ and obtain a boundary
$\Scal_\alpha$. Here, we invert the procedure: we choose a boundary
function $g(v)$ and numerically compute an upper bound on its $S_t$-upcrossing
probability using a stitching-like construction.

\begin{theorem}\label{th:inverted_stitching}
  For any nonnegative, strictly concave
  function $g:~\R ~\to~\R$ and $v_{\max} > 1$, the function
  \begin{align}
    u(v) \defineas \begin{cases}
      g(1 \bmax v), & v \leq v_{\max}, \\
      \infty, & \text{otherwise}
    \end{cases} \label{eq:inverted_stitching_boundary}
  \end{align}
  is a sub-Gaussian uniform boundary with crossing probability at most
  \begin{align}
    l_0 \inf_{\eta > 1} \sum_{k=0}^{\ceil{\log_\eta v_{\max}}} \expebrace{
      -\frac{2(g(\eta^{k+1}) - g(\eta^k))(\eta g(\eta^k) - g(\eta^{k+1}))}
          {\eta^k(\eta - 1)^2}
    }. \label{eq:inverted_stitching_sum}
  \end{align}
\end{theorem}

The proof is in \cref{sec:proof_inverted_stitching}. For simplicity we restrict
to the sub-Gaussian case; examination of the proof will show that the method
applies in other sub-$\psi$ cases as well, since we simply apply
\cref{th:uniform_chernoff} to appropriately chosen lines, but more involved
numerical calculations will be necessary, as the closed-form
\eqref{eq:inverted_stitching_sum} no longer applies. A similar idea was
considered by \citet{darling_further_1968}, using a mixture integral
approximation instead of an epoch-based construction to derive closed-form
bounds. \Cref{th:inverted_stitching} requires numerical summation but yields
tighter bounds with fewer assumptions. As an example,
\cref{th:inverted_stitching} with $\eta = 2.99$ shows that
{\small
\begin{align}
\P\eparen{
  \exists t : 1 \leq V_t \leq 10^{20} \text{ and }
  S_t \geq 1.7 \sqrt{V_t (\log\log(e V_t) + 3.46)}
} \leq 0.025. \label{eq:lil_inverted}
\end{align}
}
This boundary is illustrated in Figure~\ref{fig:finite_lil_boundaries}.

\subsection{Tuning boundaries in practice}\label{sec:tuning}

All uniform boundaries involve a tradeoff of tightness at different intrinsic
times: making a bound tighter for some range of times requires making it looser
at other times. Roughly speaking, the choice of a uniform
  boundary involves choosing both what time the bound should be optimized for
  (e.g., should the bound be tightest around 100 observations or around 100,000
  observations?) as well as how quickly the bound degrades as we move away from
  the optimized-for time (e.g., if we optimize for 100 samples, will the bound
  be twice as wide when we reach 1,000 samples, or will it stay within a factor
  of two until we reach 1,000,000 samples?). A boundary which decays more slowly
  will necessarily not be as tight around the optimized-for time. In brief,
  linear boundaries decay the most quickly, conjugate mixture boundaries decay
  substantially more slowly, and polynomial stitched boundaries decay even more
  slowly; we feel that mixture boundaries strike a good balance in
  practice. 

  Here, we explain how to optimize uniform boundaries for a particular
  time and discuss the above tradeoff in more detail.
Let $W_{-1}(x)$ be the lower branch of the Lambert $W$ function, the most
negative real-valued solution in $z$ to $z e^z = x$.  Consider the unitless
process $S_t / \sqrt{V_t}$, and the corresponding uniform boundary
$v \mapsto u(v) / \sqrt{v}$. Since all of our uniform boundaries $u(v)$ have
positive intercept at $v = 0$, and all grow at least at the rate
$\sqrt{v \log \log v}$ as $v \to \infty$, the normalized boundary
$u(v) / \sqrt{v}$ diverges as $v \to 0$ and $v \to \infty$. For the two-sided
normal mixture \eqref{eq:two_sided_normal_mixture}, there is a unique time $m$
at which $u(v) / \sqrt{v}$ is minimized; $m$ is proportional to tuning
parameter $\rho$ as follows:

\begin{figure}
\centering
\includegraphics{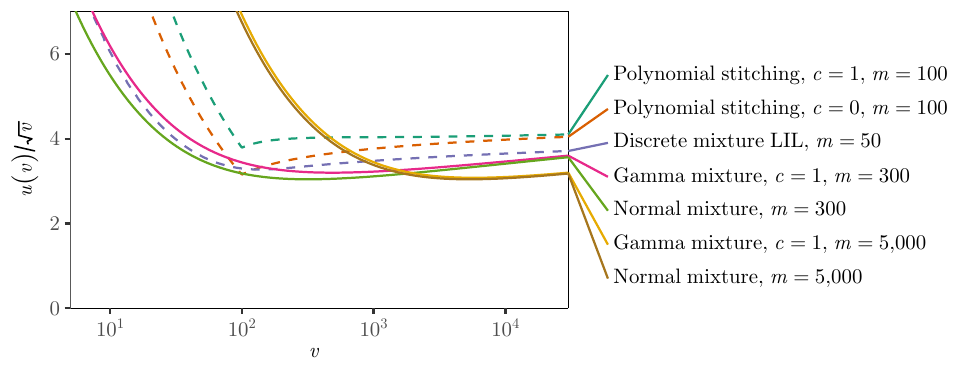}
\caption{Comparison of normalized uniform boundaries $u(v) / \sqrt{v}$ optimized
  for different intrinsic times. Normal mixture uses Appendix
  \cref{th:normal_mixture}, while gamma mixture uses
  Appendix \cref{th:gamma_mixture}. Polynomial stitched boundary is given in
  \eqref{eq:poly_stitching}, with $\eta = 2$ and $s = 1.4$. Discrete mixture
  applies \cref{th:discrete_mixture} to the density
  $f(\lambda) = 0.4 \cdot \indicator{0 \leq \lambda \leq 0.38} / [\lambda
  \log^{1.4}(0.38e / \lambda)]$ with $\eta = 1.1$, and $\lambda_{\max} = 0.38$;
  see \cref{sec:stitching_mixture} for motivation. All boundaries use
  $\alpha = 0.025$. \label{fig:normalized_boundaries}}
  
\end{figure}

\begin{proposition}\label{th:normal_mixture_rho}
  Let $u(v)$ be the two-sided normal mixture boundary~\eqref{eq:two_sided_normal_mixture} with parameter $\rho > 0$.
  \begin{enumerate}[label=(\alph*)]
  \item For fixed $\rho > 0$, the function $v \mapsto u(v) / \sqrt{v}$ is
    uniquely minimized at $v = m$ with $m$ given by
    \begin{align}
      \frac{m}{\rho} = -W_{-1}\eparen{-\frac{\alpha^2}{e l_0^2}} - 1.
      \label{eq:m_and_rho}
    \end{align}
  \item For fixed $m > 0$, the choice of $\rho$ which minimizes the boundary
    value $u(m)$ is also determined by \eqref{eq:m_and_rho}.
  \end{enumerate}
\end{proposition}
 The above result is proved in
\cref{sec:proof_normal_mixture_rho}; it is a matter of elementary calculus, but
addresses a question that has received little attention in the literature.
\Cref{fig:normalized_boundaries} includes the normalized versions of two normal
mixture boundaries optimized for different times, $m = 300$ and $m = $
5,000. Optimizing for the range of values of $V_t$ most relevant in a particular
application will yield the tightest confidence sequences. However, as the figure
shows, one need not have a very precise range of times, so long as one uses a
conservatively low value for $m$, because $u(v) / \sqrt{v}$ grows slowly after
time $m$. Indeed, for the normal mixture boundary with $\alpha = 0.05$ and
$l_0 = 1$, we have $u(m) / \sqrt{m} \approx 3.0$ and
$u(100m) / \sqrt{100m} \approx 3.6$, so that the penalty for being off by two
orders of magnitude is modest.

The one-sided normal mixture boundary of Appendix \cref{th:normal_mixture} with
crossing probability $\alpha$ is nearly identical to the two-sided normal
mixture boundary with crossing probability $2\alpha$, so one may choose $\rho$
as in \cref{th:normal_mixture_rho} with $\alpha$ doubled. For the
gamma-exponential mixture and other non-sub-Gaussian uniform boundaries,
\cref{th:normal_mixture_rho} provides a good approximation in
practice. \Cref{fig:normalized_boundaries} includes gamma-exponential mixture
boundaries with the same $\rho$ values as each corresponding normal mixture
boundary. Though the normalized gamma-exponential mixture boundary with
$m = 300$ clearly reaches its minimum at $v > m$, this choice of $\rho$ seems
reasonable. Discrete mixtures can be similarly tuned by adjusting the
precision of the mixing distribution, but require additional considerations (\cref{sec:discrete_mix_details}).

Comparing the sub-Gaussian stitched boundary, discrete mixture boundary, and
 normal mixture boundary optimized for $m = 300$ in
\cref{fig:normalized_boundaries} illustrates another important point for
practice: although the normal mixture bound grows more quickly than the others
as $v \to \infty$, it remains smaller over about three orders of magnitude. This
makes it preferable for many real-world applications, as the longest feasible
duration of an experiment is rarely more than two orders of magnitude larger
than the earliest possible stopping time. For example, many online experiments
run for at least one week to account for weekly seasonality effects, and very
few such experiments last longer than 100 weeks. As both the normal mixture and
the discrete mixture are unimprovable in general (\cref{sec:admissibility}), the
difference is attributable to the choice of mixture, or alternatively, to the
fact that the normal mixture trades tightness around the optimized-for time in
exchange for looseness at much later times. The lesson is that the
$\Ocal(v \log \log v)$ rate, while asymptotically optimal in certain settings
and useful for theory and some applications, may not be preferable in all
real-world scenarios.

\subsection{Unimprovability of uniform boundaries}\label{sec:admissibility}

\Cref{th:uniform_boundary} of a sub-$\psi$ boundary $u$ involves only an upper
bound on the $u$-crossing probability of any sub-$\psi$ process $(S_t)$. One may
reasonably ask for corresponding lower bounds on the $u$-crossing probability to
quantify how tight this boundary is. In the ideal case, we might desire a
boundary $u$ such that the true $u$-crossing probability of some process $(S_t)$
is equal to the upper bound. In nonparametric settings, we cannot achieve this
goal for every sub-$\psi$ process. However, we might still ask that there exists
\emph{some} sub-$\psi$ process for which the true $u$-crossing probability is
arbitrarily close to the upper bound, so that the upper bound on crossing
probability is unimprovable in general. That is, we might ask that the
inequality on the supremum in \cref{th:uniform_boundary} holds with equality.

The fact we wish to point out, known in various forms, is that in the scalar,
sub-Gaussian case, exact mixture bounds are unimprovable in the above sense. It
is in this sense that the discrete mixture bound in
\cref{fig:finite_lil_boundaries} provides a lower bound, showing that the
sub-Gaussian polynomial stitched bound cannot be improved by much. The following
result shows that, for any exact, sub-Gaussian mixture boundary $\Mcal_\alpha$,
as defined in \cref{th:basic_mixture} for $\psi = \psi_N$, there exists a
sub-Gaussian process whose true $\Mcal_\alpha$-crossing probability is
arbitrarily close to $\alpha$. The result is similar to Theorem 2 of
\citet{robbins_boundary_1970}, which gives a more general invariance principle,
but requires conditions on the boundary that appear difficult to verify for
arbitrary mixture boundaries $\Mcal_\alpha$. Recall that $\mathbb{S}_{\psi_N}^1$
is the class of pairs of processes $(S_t,V_t)$ such that $(S_t)$ is
1-sub-Gaussian with variance process $(V_t)$.

\begin{proposition}\label{th:unimprovable}
  For any exact, 1-sub-Gaussian mixture boundary $\Mcal_\alpha$,
  \begin{align}
    \sup_{(S_t,V_t) \in \mathbb{S}_{\psi_N}^1}
      \P(\exists t \geq 1: S_t \geq \Mcal_\alpha(V_t))
    = \alpha.
  \end{align}
\end{proposition}

We prove \cref{th:unimprovable} in \cref{sec:proof_unimprovable}. In general,
for each $\alpha$ there is an infinite variety of boundaries that are
unimprovable in the above sense, differing in when they are loose and
tight. These different boundaries will yield confidence sequences which are
loose or tight at different sample sizes, or, equivalently, are efficient for
detecting different effect sizes. Such a boundary cannot be tightened everywhere
without increasing the crossing probability.

\section{Applications}\label{sec:applications}

After presenting an empirical-Bernstein confidence sequence for bounded
observations, we apply our uniform boundaries to causal effect estimation and
matrix martingales. We also consider estimation for a general, one-parameter
exponential family.

\subsection{An empirical-Bernstein confidence sequence}

The following novel result is proved in \cref{sec:proof_empirical_variance}
using a self-normalization argument, which leads to its attractive simplicity.
Recall the estimand $\mu_t \defineas t^{-1} \sum_{i=1}^t \E_{i-1} X_i$, the
average conditional expectation.  
\begin{theorem}\label{th:empirical_variance}
  Suppose $X_t \in [a,b]$ a.s. for all $t$. Let $(\Xhat_t)$ be any
  $[a,b]$-valued predictable sequence, and let $u$ be any sub-exponential
  uniform boundary with crossing probability $\alpha$ for scale $c = b -
  a$. Then
  \begin{align}
    \P\eparen{
      \forall t \geq 1:
      \eabs{\Xbar_t - \mu_t}
      < \frac{u\eparen{\sum_{i=1}^t (X_i - \Xhat_i)^2}}{t}
    } \geq 1 - 2\alpha.
    \label{eq:empirical_bernstein}
  \end{align}
\end{theorem}

This is an empirical-Bernstein bound because it uses the sum of observed squared
deviations to estimate the true variance, much like a classical $t$-test. Hence
the confidence radius scales with the true standard deviation for
sufficiently large samples, regardless of the support diameter $b - a$, and with
no prior knowledge of the true variance. Note also that this bound does not
require that observations share a common mean.

The confidence statement \eqref{eq:empirical_bernstein} holds for \emph{any}
sequence of predictions $(\Xhat_i)$, but predictions closer to the conditional
expectations, $\Xhat_i \approx \E_{i-1} X_i$, will yield smaller confidence
intervals on average. A simple choice is the mean,
$\Xhat_t = (t-1)^{-1} \sum_{i=1}^{t-1} X_i$, which will be effective when the
samples are i.i.d., for example. But the predictions $(\Xhat_i)$ can also make use of trends,
seasonality, stratification or regression (in the presence of covariates),
machine learning algorithms, or any other information that
may aid with prediction.

For an explicit example, assume $X_i \in [0, 1]$ and define the empirical
variance as $\Vhat_t \defineas \sum_{i=1}^t (X_i - \Xbar_{i-1})^2$. Invoking
\cref{th:empirical_variance} with the polynomial stitched bound
\eqref{eq:poly_stitching} using $c=1$, $\eta=2$, $m=1$, and
$h(k) \propto k^{1.4}$, we have the following 95\%-confidence sequence for
$\mu_t$:

\begin{align}
\small
\smash{
\Xbar_t \pm \frac{
  1.7 \sqrt{(\Vhat_t \bmax 1) (\log \log(2(\Vhat_t \bmax 1)) + 3.8)}
  + 3.4 \log \log(2(\Vhat_t \bmax 1)) + 13
}{t}.}
\label{eq:closed_form_empirical}
\end{align}
When a closed form is not required, the gamma-exponential mixture (\cref{th:gamma_mixture}) may yield tighter bounds than stitching; simulations in \cref{sec:simulations} demonstrate the use of
\cref{th:empirical_variance} with this mixture.

\subsection{Estimating ATE in the Neyman-Rubin model}\label{sec:ate}

As one illustration of \cref{th:empirical_variance}, we consider the sequential
estimation of average treatment effect under the Neyman-Rubin potential outcomes
model \citep{neyman_application_1990, rubin_estimating_1974,
  imbens_causal_2015}. We imagine a sequence of experimental units,
each with real-valued potential outcomes under control and treatment denoted by
$\{Y_t(0),Y_t(1)\}_{t \in \N}$, respectively. These potential
outcomes are fixed, but we observe only one outcome for each unit in the
experiment. We assign a randomized treatment to each unit, denoted by the
$\brace{0,1}$-valued random variable $Z_t \in \Fcal_t$, observing
$Y_t^\obs \defineas Y_t(Z_t)$. Here treatment is assigned by flipping a coin for
each subject, with a bias possibly depending on previous observations. This
treatment assignment is the only source of randomness.  Specifically, let
$P_t \defineas E_{t-1} Z_t$ and suppose $0 < P_t < 1$ a.s. for all $t$; then we
permit $P_t$ to vary between individuals and to depend on past outcomes. This
accommodates Efron's biased coin design \cite{efron_forcing_1971} and
related covariate balancing methods.

At each step $t$, having treated and observed units $1, \dots, t$, we wish to
draw inference about the estimand
$\ATE_t \defineas t^{-1} \sum_{i=1}^t [Y_i(1) - Y_i(0)]$. In particular,
we seek a confidence sequence for $(\ATE_t)_{t=1}^\infty$.  To construct our
estimator, we may utilize any predictions $\Yhat_t(0)$ and $\Yhat_t(1)$ for each
unit's potential outcomes; these random variables must be
$\Fcal_{t-1}$-measurable, for each $t$. We then employ the inverse probability
weighting estimator
\begin{align}
X_t &\defineas
  \Yhat_t(1) - \Yhat_t(0)
  + \pfrac{Z_t - P_t}{P_t(1 - P_t)} (Y_t^\obs - \Yhat_t(Z_t)),
  \label{eq:horvitz}
\end{align}
which is (conditionally) unbiased for the individual treatment effect
$Y_t(1) - Y_t(0)$. As with \cref{th:empirical_variance}, better predictions will
lead to shorter confidence intervals, but the coverage guarantee holds for any
choice of predictions, and a reasonable choice would be the average of
past observed outcomes. See
\citet{aronow_class_2013} for a similar strategy for fixed-sample
estimation.

We assume bounded potential outcomes; for simplicity we assume
$Y_t(k) \in [0, 1]$ for all $t \geq 1, k = 0, 1$, and we assume predictions are
likewise bounded. We further assume that treatment probabilities are uniformly
bounded away from zero and one. Then, an empirical-Bernstein confidence sequence
for $\ATE_t$ follows from \cref{th:empirical_variance}, where we use
$\Xhat_t = \Yhat_t(1) - \Yhat_t(0)$ so that
\begin{align}
  V_t ~\defineas~ \sum_{i=1}^t (X_i - \Xhat_i)^2
  ~=~ \sum_{i=1}^t \pfrac{Z_i - P_i}{P_i(1 - P_i)}^2 (Y_i^\obs - \Yhat_i(Z_i))^2.
\end{align}

\begin{corollary}\label{th:ate_bernoulli}
  Suppose $P_t \in [p_{\min}, 1 - p_{\min}]$ a.s., $Y_t(k) \in [0, 1]$ and
  $\Yhat_t(k) \in [0, 1]$ for all $t \geq 1, k = 0, 1$. Let $u$ be any
  sub-exponential uniform boundary with scale $2 / p_{\min}$ and crossing
  probability $\alpha$. Then
  \begin{align}
    \P\eparen{
      \forall t \geq 1:
      \eabs{\Xbar_t - \ATE_t}
      < \frac{u(V_t)}{t}
    } \geq 1 - 2\alpha.
  \end{align}
\end{corollary}
\begin{figure}
\centering
\includegraphics{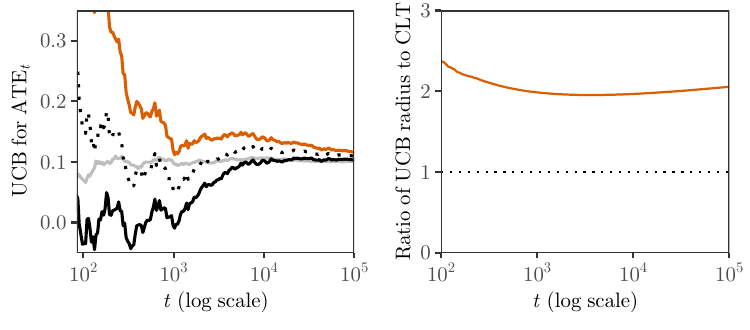}
\caption{
Upper half of 95\% empirical-Bernstein confidence sequence for
  $\ATE_t$ under Bernoulli randomization based on one simulated sequence of i.i.d.
  observations, $P_t \equiv 0.5$, $Y_i(0) \sim \text{Ber}(0.5)$,
  $Y_i(1) = \xi_i \bmax Y_i(0)$ where $\xi_i \sim \text{Ber}(0.2)$. Grey line
  shows estimand $\ATE_t$. Dotted line shows fixed-sample confidence bounds
  based on difference-in-means estimator and normal approximation; these bounds
  fail to cover the true $\ATE_t$ at many times. Our bound uses
  $\Yhat_t(k) = \sum_{i=1}^{t-1} Y_i^\obs \indicator{Z_i = k} / \sum_{i=1}^{t-1}
  \indicator{Z_i = k}$, $\alpha = 0.05$ and a gamma-exponential mixture bound
  with $\rho=12.6$, chosen to optimize for intrinsic time $V_t = 100$.
  \label{fig:bernoulli_ate}} 
\end{figure}

For $u$, one may choose the gamma-exponential mixture boundary
(\cref{th:gamma_mixture}) or the stitched boundary
\eqref{eq:poly_stitching} with $c~=~\frac{2}{p_{\min}}$. Figure~\ref{fig:bernoulli_ate}
illustrates our strategy on simulated data. Over the range $t = 100$ to
$t = $100,000 displayed, our bound is about twice as wide as the fixed-sample
CLT bound, with the ratio growing at a slow $\Ocal(\sqrt{\log t})$ rate
thereafter. Of course the fixed-sample CLT bound provides no uniform coverage
guarantee.

\subsection{Matrix iterated logarithm bounds}\label{sec:matrix_application}

Our second application is the construction of iterated logarithm bounds for
random matrix sums and their use in sequential covariance matrix estimation. The
curved uniform bounds given in \cref{sec:main_results} may be applied to matrix
martingales by taking $(S_t)$ to be the maximum eigenvalue process of the
martingale and $(V_t)$ the maximum eigenvalue of the corresponding matrix
variance process. \Citet[Section 2]{howard_exponential_2018} give sufficient
conditions for \cref{th:canonical_assumption} to hold in this matrix case. Then
\Cref{th:stitching} yields a novel matrix finite LIL; here we give an example
for bounded increments. We denote the space of symmetric, real-valued,
$d \times d$ matrices by $\mathbb{S}^d$; $\gamma_{\max}(\cdot)$ denotes the
maximum eigenvalue;
$ \ell_{\eta,s}(v) = s \log\log(\eta v / m) + \log \frac{d\, \zeta(s)}{\alpha
  \log^s \eta}$; and $k_1(\eta), k_2(\eta)$ are defined in
\eqref{eq:stitching_operator}.

\begin{corollary}\label{th:matrix_finite_lil}
  Suppose $(Y_t)_{t=1}^\infty$ is a $\mathbb{S}^d$-valued matrix martingale such
  that $\smash{\gamma_{\max}(Y_t - Y_{t-1}) \leq b}$ a.s. for all $t$. Let
  $V_t \defineas \gamma_{\max}(\sum_{i=1}^t \E_{t-1} (Y_t - Y_{t-1})^2)$ and $S_t \defineas \gamma_{\max}(Y_t)$.  Then
  for any $\smash{\eta >1, s>1, m>0,\alpha\in(0,1)}$, we have
\begin{equation}
\small
\P\eparen{
  \exists t \geq 1 :
  S_t \geq
  k_1(\eta) \sqrt{(V_t \bmax m) \ell_{\eta,s}(V_t \bmax m)}
  + \frac{b k_2(\eta)}{3} \ell_{\eta,s}(V_t \bmax m)
} \leq \alpha. \label{eq:matrix_finite_lil}
\end{equation}
\end{corollary}

The result follows using the polynomial stitched boundary after invoking Fact 1(c)
and Lemma 2 of \citet{howard_exponential_2018}
(cf. \citep{tropp_freedmans_2011}), which show that $(S_t)$ is sub-gamma with
variance process $(V_t)$, scale $c = b/3$, and $l_0 = d$. 
Beyond bounded increments, the same bound holds for any sub-gamma
process. As evidenced by \Cref{th:universal}, this is a very general condition.

Taking $\eta$ and $s$ arbitrarily close to one and using the final result of
\cref{th:stitching}, we obtain the following asymptotic matrix upper LIL, proved in \cref{sec:proof_matrix_asymptotic_lil}. Here
we denote the martingale increments by $\Delta Y_t \defineas Y_t - Y_{t-1}$.

\begin{corollary}\label{th:matrix_asymptotic_lil}
  Let $(Y_t)_{t=1}^\infty$ be a $\mathbb{S}^d$-valued, square-integrable
  martingale, and define
  $V_t = \gamma_{\max}\eparen{\sum_{i=1}^t \E_{i-1} \Delta Y_t^2}$. Then
\begin{align}
  \limsup_{t \to \infty}
  \frac{\gamma_{\max}\eparen{Y_t}}
       {\sqrt{2 V_t \log \log V_t}}
  \leq 1
  \quad \text{a.s. on } \ebrace{\sup_t V_t = \infty}
  \label{eq:matrix_lil}
\end{align}
whenever either (1) the increments $(\Delta Y_t)$ are i.i.d., or (2) the
increments $(\Delta Y_t)$ satisfy a Bernstein condition on higher moments: for
some $c > 0$, for all $t$ and all $k > 2$,
$\E_{t-1} (\Delta Y_t)^k \preceq (k!/2) c^{k-2} \E_{t-1} \Delta Y_t^2$.
\end{corollary}

The
Bernstein condition holds if the increments are uniformly bounded,
$\gamma_{\max}(\Delta Y_t) \leq c$ for some $c > 0$. Also, in the i.i.d.\ case,
$\P(V_t \to \infty) = 1$ and then \eqref{eq:matrix_lil} states that
$\smash{\limsup_{t \to \infty} \gamma_{\max}\eparen{Y_t} / \sqrt{2 \gamma_{\max}(\E
  \Delta Y_1^2) t \log \log t} \leq 1}$, a.s. on $\brace{\sup_t V_t = \infty}$.
 When $d = 1$, this recovers the classical upper LIL,
  showing that \cref{th:matrix_asymptotic_lil} cannot be improved uniformly, but
  we are not aware of an appropriate lower bound for the general matrix case.

We now consider the nonasymptotic sequential estimation of a covariance matrix
based on bounded vector observations \citep{rudelson_random_1999,
  vershynin_introduction_2012, gittens_tail_2011, tropp_introduction_2015,
  koltchinskii_concentration_2017}. In particular, we observe a sequence of
independent, mean zero, $\R^d$-valued random vectors $x_t$ with common
covariance matrix $\Sigma = \E x_t x_t^T$. We wish to estimate $\Sigma$ using an
operator-norm confidence ball centered at the empirical covariance matrix
$\Sigmahat_t \defineas t^{-1} \sum_{i=1}^t x_i x_i^T$. For fixed-sample
estimation, when $\norm{x_i}_2 \leq \sqrt{b}$ a.s. for all $i \in [t]$, the
analysis of \citet[section 1.6.3]{tropp_introduction_2015} implies
\begin{align}
{\small
\P\eparen{
  \opnorm{\Sigmahat_t - \Sigma} \geq
  \sqrt{\frac{2 b \opnorm{\Sigma} \log(2d/\alpha)}{t}}
  + \frac{4 b \log(2d/\alpha)}{3t}
} \leq \alpha.}
\label{eq:fixed_sample_covariance}
\end{align}

We use a sub-Poisson uniform boundary to obtain a uniform analogue:

\begin{corollary}\label{th:empirical_covariance}
  Let $(x_t)_{t=1}^\infty$ be a sequence of $\R^d$-valued, independent
  random vectors with $\E x_t = 0$, $\norm{x_t}_2 \leq \sqrt{b}$ a.s. and
  $\E x_t x_t^T = \Sigma$ for all $t$. If $u$~is a sub-Poisson uniform boundary
  with crossing probability $\alpha$ and scale $2b$, then
\begin{align}
  \P\eparen{
    \exists t \geq 1 : \opnorm{\Sigmahat_t - \Sigma} \geq
    \frac{1}{t} u\eparen{b t \opnorm{\Sigma}}
  } \leq \alpha.
\end{align}
\end{corollary}
 For example, using the polynomial stitched bound with scale
$c = 2b/3$ and $m = b \opnorm{\Sigma}$, \cref{th:empirical_covariance} gives a
$(1-\alpha)$-confidence sequence for $\Sigma$ with operator norm radius
$\Ocal(\sqrt{t^{-1} \log \log t})$. This bound has the closed form
\begin{equation}
{\small \P\eparen{
  \exists t \geq 1 :
  \opnorm{\Sigmahat_t - \Sigma} \geq
  k_1 \sqrt{\frac{b \opnorm{\Sigma} \ell(t)}{t}}
  + \frac{4b k_2 \ell(t)}{3t}
} \leq \alpha,} \label{eq:covariance_finite_lil}
\end{equation}
where
$\ell(t) = s \log\log(\eta t) + \log \frac{d\, \zeta(s)}{\alpha \log^s \eta}$,
and $k_1,k_2$ are defined in \eqref{eq:stitching_operator}.

\begin{figure}
\centering
\includegraphics{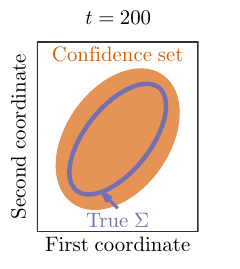}
\includegraphics{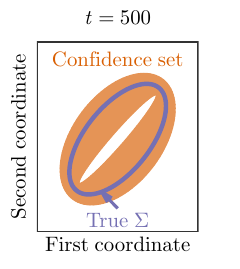}
\includegraphics{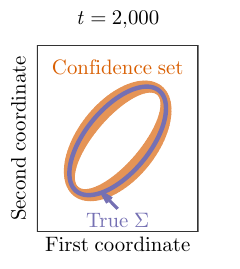}
\caption[Illustration of covariance matrix confidence sequence]{
  The matrix confidence sequence of
  \cref{th:empirical_covariance} based on one simulated sequence. Observations
  are drawn i.i.d.\ taking values $\pm (\sqrt{2} \enskip \sqrt{2})^T$,
  $\pm (1/\sqrt{2} \enskip -1/\sqrt{2})^T$ each with probability 1/4, with
  covariance matrix
  $\Sigma = \frac{1}{4} \eparen{\begin{smallmatrix}5 & 3 \\ 3 &
      5\end{smallmatrix}}$, which is represented by the ellipse
  $x^T \Sigma^{-1} x = 1$. Confidence ball with level $\alpha = 0.05$ is
  represented by shaded area between ellipses corresponding to elements of the
  confidence ball with minimal and maximal trace. Confidence sequence from
  \cref{th:empirical_covariance} uses $\smash{b = 4}$ and a discrete mixture
  boundary with $\psi = \psi_G$ using $c = 2b/3$, mixture density $\flil_{1.4}$
  from \eqref{eq:lil_mixture} with $s = 1.4$ matching \eqref{eq:poly_stitching_explicit},
  $\eta = 1.1$ and $\overlambda = 0.262$ chosen as described in
  \cref{sec:discrete_mix_details}. \label{fig:covariance}  }
\end{figure}

In other words, with high probability, we have for all $t \geq 1$ that
\begin{align}
{\small
\opnorm{\Sigmahat_t - \Sigma}
  \lesssim \sqrt{\frac{b \log(d\log t)}{t}} + \frac{b \log(d\log t)}{t}.}
\end{align}
Compared to the fixed-sample result \eqref{eq:fixed_sample_covariance}, we
obtain uniform control by adding a factor of $\log \log t$. We are not aware of
other results like these for sequential covariance matrix
estimation. Figure~\ref{fig:covariance} illustrates the confidence sequence of
\cref{th:empirical_covariance} on simulated data using a discrete mixture
boundary with the mixture density $\flil_s$ defined in \eqref{eq:lil_mixture}.

\subsection{One-parameter exponential families}
\label{sec:one_param_exp}

Suppose $(X_t)$ are i.i.d.\ from an exponential family in mean parametrization,
with sufficient statistic $T(X)$ having mean in some set $\Omega$. For each
$\mu \in \Omega$, we write the density as
$\smash{f_{\mu}(x) = h(x) \expebrace{\theta(\mu) T(x) - A(\theta(\mu))}}$ where
$A'(\theta(\mu)) = \mu$. Let $\psi_\mu$ be the cumulant-generating function of
$T(X_1) - \mu$ when $\E T(X_1) = \mu$, that is,
$\smash{\psi_\mu(\lambda) \defineas A(\lambda + \theta(\mu)) - A(\theta(\mu)) - \lambda
\mu}$, with $\psi_\mu(\lambda) \defineas \infty$ if the RHS does not
exist. Writing $S_t(\mu) \defineas \sum_{i=1}^t T(X_i) - t \mu$, the process
$\expebrace{\lambda S_t(\mu) - t \psi_\mu(\lambda)}$ is the likelihood ratio
testing $H_0: \theta = \theta(\mu)$ against
$H_1: \theta = \theta(\mu) + \lambda$, and if we use a method-of-mixtures
uniform boundary, the resulting confidence sequence will be dual to a family of
mixture sequential probability ratio tests, as discussed in
\cref{sec:hypothesis_testing}. To obtain a two-sided confidence sequence, we use
the ``reversed'' CGF $\tilde{\psi}_\mu(\lambda) = \psi_\mu(-\lambda)$. We
summarize these observations as follows; see \citet[Theorem
1]{lai_confidence_1976} for a related result.

\begin{corollary}\label{th:expo_family}
  Suppose, for each $\mu \in \Omega$, $u_\mu$ is a sub-$\psi_\mu$ uniform bound
  with crossing probability $\alpha_1$, and $\tilde{u}_\mu$ is a
  sub-$\tilde{\psi}_\mu$ uniform bound with crossing probability
  $\alpha_2$. Defining
\begin{align}
  \CI_t \defineas
    \ebrace{\mu \in \Omega: -\tilde{u}_\mu(t) < S_t(\mu) < u_\mu(t)},
  \label{eq:expo_family_cs}
\end{align}
we have
$\P(\forall t \geq 1: \E T(X_1) \in \CI_t) \geq 1 - \alpha_1 - \alpha_2$.

\end{corollary}

\section{Simulations}\label{sec:simulations}

\begin{figure}[t!]
\centering
\includegraphics{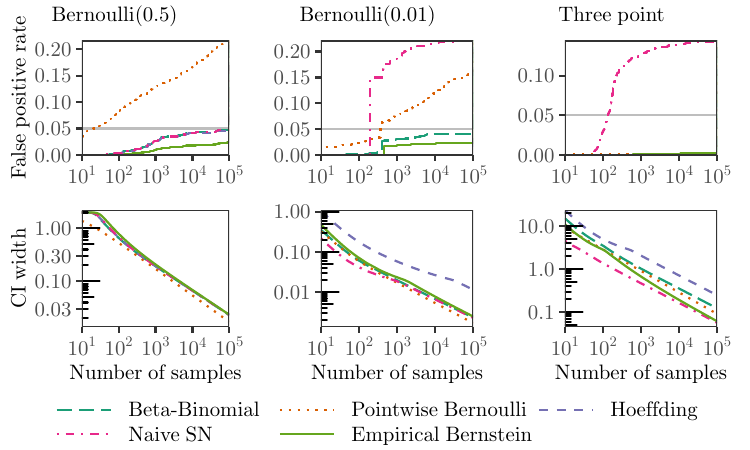}
\caption{Summary of 1,000 simulations, each with 100,000 i.i.d.\
  observations from the indicated distribution. Top panels show the proportion
  of replications in which the 95\%-confidence sequence has excluded the true
  mean by time $t$. Bottom panels show the mean confidence interval
  width. The ``three point'' distribution takes values $-1.408$ and $1$ with
  probability $0.495$ each, and takes value $20$ with probability
  $0.01$. ``Hoeffding'' uses a normal mixture boundary
  \eqref{eq:two_sided_normal_mixture}, while``Beta-Binomial'' uses the
  beta-binomial mixture (\cref{th:two_sided_beta_mixture}). ``Pointwise
  Bernoulli'' uses a nonasymptotic bound based on the Bernoulli KL-divergence
  which is valid pointwise but not uniformly. ``Empirical Bernstein'' uses the
  strategy given in \cref{th:empirical_variance} with a gamma-exponential
  mixture boundary, \cref{th:gamma_mixture}. ``Naive SN'' uses a normal mixture
  boundary with an empirical variance estimate, which does not guarantee
  coverage. In all cases, $\rho$ is chosen to optimize for a sample size of
  $t = 500$. \label{fig:bounded_simulations}}
\end{figure}

In\footnote{The repository \url{https://github.com/gostevehoward/cspaper}
  contains code to reproduce all simulations and plots in this paper. Uniform
  boundaries themselves are implemented in R and Python packages at
  \url{https://github.com/gostevehoward/confseq}.}
\cref{fig:bounded_simulations} we illustrate the error control of some of our
confidence sequences for estimating the mean of an i.i.d.\ sequence of
observations $(X_i)$ with bounded support $[a,b]$. We compare four strategies:
\begin{enumerate}
  
\item The Hoeffding strategy exploits the fact that bounded observations are
  sub-Gaussian (\citealp{hoeffding_probability_1963};
  cf. \citealp{howard_exponential_2018}, Lemma 3(c)).
  We use a two-sided normal mixture boundary
  \eqref{eq:two_sided_normal_mixture} with variance process
  $V_t = (b-a)^2 t / 4$.
\item The beta-binomial strategy uses the stronger condition that bounded
  observations are sub-Bernoulli (\citealp{hoeffding_probability_1963};
  cf. \citealp{howard_exponential_2018}, Fact 1(b)), accounting for the true
  mean as well as the boundedness, but possibly failing to take account of the
  true variance. For hypothesized true mean $\mu$, this strategy uses the
  beta-binomial mixture boundary given in \cref{th:two_sided_beta_mixture}, with
  parameters $g(\mu) = \mu - a$ and $h(\mu) = b - \mu$, and variance process
  $V_t(\mu) = g(\mu) h(\mu) t$. The confidence set for the mean is
  $\brace{\mu \in [a,b]: -f_{g(\mu),h(\mu)}(V_t(\mu)) \leq \sum_{i=1}^t X_i - t
    \mu \leq f_{h(\mu),g(\mu)}(V_t(mu))}$. This is more efficiently computed
  using the mixture supermartingale $m(S_t,V_t)$ of \eqref{eq:beta_mixture}, as
  $\brace{\mu \in [a,b]: m(\sum_{i=1}^t X_i - t \mu, V_t(\mu)) < 1 / \alpha}$.
\item The pointwise Bernoulli strategy uses the same
    sub-Bernoulli condition as the beta-binomial strategy, but relies on a
    fixed-sample Cram\'er-Chernoff bound which is valid pointwise but not
    uniformly over time. Specifically, we reject mean $\mu$ if
    $V_t \psi_B^\star(S_t / V_t) \geq \log \alpha^{-1}$, where $S_t$ is the sum
    of centered observations as usual, $V_t = (\mu - a)(b - \mu) t$, and we set
    $\smash{g = \mu - a, h = b - \mu}$ in $\psi_B$, with $\psi_B^\star$ its
    Legendre-Fenchel transform.
\item The empirical-Bernstein strategy uses an empirical estimate of variance,
  thus achieving a confidence width scaling with the true variance in all three
  cases. Here we use
  \cref{th:empirical_variance} with a gamma-exponential mixture boundary (\cref{th:gamma_mixture}). For predictions, we use the mean of past
  observations: $\Xhat_t = (t-1)^{-1} \sum_{i=1}^{t-1} X_i$.
\item The naive self-normalized (``Naive SN'') strategy plugs the
  empirical variance estimate, the sum of squared prediction errors from
  \cref{th:empirical_variance}, into the two-sided
  normal mixture \eqref{eq:two_sided_normal_mixture}. It ignores the facts
  that the observations are not sub-Gaussian with respect to their true variance
  and that the variance is estimated. This strategy is similar to that of
  \citet{johari_peeking_2017} and does not guarantee coverage. Though it will
  sometimes control false positives, coverage rates can easily be inflated
  for asymmetric, heavy-tailed distributions, as we illustrate.
\end{enumerate}

We present three cases of bounded distributions. The first case is the easiest,
with $\Ber(0.5)$ observations. Here the sub-Gaussian variance parameter based on
the boundedness of the observations is equal to the true variance, so the
Hoeffding strategy performs well. The empirical-Bernstein strategy is only a
little wider, and all four successfully control false positives. The story
changes with the more difficult $\Ber(0.01)$ distribution, however. The
Hoeffding boundary is far too wide, since it fails to make use of information
about the true variance. The beta-binomial bound uses information about variance
provided by the first moment to achieve the correct scaling. The naive
self-normalized strategy, on the other hand, yields confidence intervals that
are too small and fail to control false positive rate. The empirical-Bernstein
strategy, though only slightly wider than the naive bound for large sample
sizes, gives just enough extra width to control the false positive rate and is
nearly as narrow as the beta-binomial bound. The final, three-point
distribution takes values $-1.408$ and $1$ with probability $0.495$ each, and
takes value $20$ with probability $0.01$. Here the beta-binomial
strategy yields confidence intervals that are too wide. In this most difficult
case, only the empirical-Bernstein strategy yields tight intervals while 
controlling false positive rates.

\section{Implications for sequential hypothesis testing}
\label{sec:hypothesis_testing}

We have organized our presentation around confidence sequences and closely
related uniform concentration bounds
due to our belief that they offer a useful ``user interface'' for sequential
inference. However, our methods also yield always-valid $p$-values
\citep{johari_always_2015} for sequential tests. Indeed, a slew of related definitions from the
literature are equivalent or ``dual'' to one another. Here we briefly discuss these
connections.
The following result, proved
in \cref{sec:proof_equiv_uniform_defns}, gives equivalent formulations of
 common definitions in sequential testing.

\begin{lemma}\label{th:equiv_uniform_defns}
  Let $(A_t)_{t=1}^\infty$ be an adapted sequence of events in some filtered
  probability space and let $A_\infty \defineas \limsup_{t \to \infty} A_t$. The
  following are equivalent:
  \begin{enumerate}[label=(\alph*)]
  \item $\P\eparen{\eunion_{t = 1}^\infty A_t} \leq \alpha.$
  \item $\P(A_T) \leq \alpha$ for all random (not necessarily stopping) times $T$.
  \item $\P(A_\tau) \leq \alpha$ for all stopping times $\tau$, possibly
    infinite.
  \end{enumerate}
\end{lemma}

Our definition of confidence sequences \eqref{eq:conf_seq_defn}, based on
\citet{darling_confidence_1967} and \citet{lai_incorporating_1984}, differs from
that \citet{johari_always_2015}, who require that
$\P(\theta_\tau \in \CI_\tau) \geq 1 - \alpha$ for all stopping times
$\tau$. They allow $\tau = \infty$ by defining
$\CI_\infty \defineas \liminf_{t \to \infty} \CI_t$. By taking
$A_t \defineas \brace{\theta_t \notin \CI_t}$ in \cref{th:equiv_uniform_defns},
we see that the distinction is immaterial, and furthermore that we could
equivalently define confidence sequences in terms of arbitrary random times, not
necessarily stopping times. This generalizes Proposition 1 of
\citet{zhao_adaptive_2016}.

\paragraph{Always-valid $p$-values and tests of power one} 
As an alternative to confidence sequences, \citet{johari_always_2015} define an
\emph{always-valid $p$-value process} for some null hypothesis $H_0$ as an
adapted, $[0,1]$-valued sequence $(p_t)_{t=1}^\infty$ satisfying
$\P_0(p_\tau \leq \alpha) \leq \alpha$ for all stopping times $\tau$, where
$\P_0$ denotes probability under the null $H_0$. Taking
$A_t \defineas \brace{p_t \leq \alpha}$ in \cref{th:equiv_uniform_defns} shows
that we may replace this definition with an equivalent one over all random
times, not necessarily stopping times, or with the uniform condition
$\P_0(\exists t \in \N: p_t \leq \alpha) \leq \alpha$. By analogy to the usual
dual construction between fixed-sample $p$-values and confidence
intervals, one can see
that confidence sequences are dual to always-valid $p$-values, and both are dual
to sequential tests, as defined by a stopping time and a binary
random variable indicating rejection \citep[Proposition 5]{johari_always_2015}.
In particular, for the null $H_0: \theta = \theta^\star$, if $(\CI_t)$ is a
$(1-\alpha)$-confidence sequence for $\theta$, it is clear that a test which
stops and rejects the null as soon as $\theta^\star \notin \CI_t$ controls type
I error:
$\P_0(\text{reject } H_0) = \P_0(\exists t \in \N : \theta^\star \notin \CI_t)
\leq \alpha$. Typically, then, a confidence sequence based on any of the curved
uniform bounds in this paper, with radius $u(v) = o(v)$, will yield a \emph{test
  of power one} \citep{darling_iterated_1967, robbins_statistical_1970}. In
particular, for a confidence sequence with limits $\bar{X}_t \pm u(V_t)$, it is
sufficient that $\bar{X}_t \convas \theta$ and
$\limsup_{t \to \infty} V_t / t < \infty$ a.s., conditions that usually
hold. These conditions imply that the radius of the confidence sequence,
$u(V_t) / t$, approaches zero, while the center $\bar{X}_t$ is eventually
bounded away from $\theta^\star$ whenever $\theta \neq \theta^\star$, so that
the confidence sequence eventually excludes $\theta^\star$ with probability
one.

In the one-parameter exponential family case considered in
\cref{sec:one_param_exp}, as noted above, the exponential process
$\expebrace{\lambda S_t(\mu) - t \psi_\mu(t)}$ is exactly the likelihood ratio
for testing $H_0: \theta = \theta(\mu)$ against
$H_1: \theta = \theta(\mu) + \lambda$. From the definitions
\eqref{eq:expo_family_cs} and \eqref{th:basic_mixture} we see that, when using a
mixture uniform boundary, a sequential test which rejects as soon as the
confidence sequence of \cref{th:expo_family} excludes $\mu^\star$ can be seen as
equivalently rejecting as soon as either of the mixture likelihood ratios
$\int \expebrace{\lambda S_t - \psi_{\mu^\star}(\lambda) t} \d F(\lambda)$ or
$\int \expebrace{-\lambda S_t - \psi_{\mu^\star}(-\lambda) t} \d F(\lambda)$
exceeds $2/\alpha$. Thus a sequential hypothesis test built upon a mixture-based
confidence sequence is equivalent to a mixture sequential probability ratio test
\citep{robbins_statistical_1970} in the parametric setting. As discussed
in \cref{sec:stitching_mixture}, stitching can be viewed as an
approximation to certain mixture bounds, so that hypothesis tests based on
stitched bounds are also approximations to mixture SPRTs. Importantly, our
confidence sequences are natural nonparametric
generalizations of the mixture SPRT, recovering various mixture SPRTs in the
parametric settings.

\paragraph{Pros and cons of the running intersection} 
Our definition \eqref{eq:conf_seq_defn} of a confidence sequence allows for the
parameter $\theta_t$ to vary with $t$. It is common in the literature on
sequential testing to assume a single, stationary parameter,
$\theta_t \equiv \theta$, but this assumption has a troublesome consequence in
the context of confidence sequences. If the confidence sequence $(\CI_t)$
satisfies $\P(\forall t: \theta \in \CI_t) \geq 1 - \alpha$, then the running intersection
$\widetilde{\CI}_t \defineas \intersect_{s \leq t} \CI_t$ is also uniformly valid for
$\theta$, is never larger and may be much smaller. 
This was observed by \citet{darling_iterated_1967}, and is used in the implementation of
\citet{johari_peeking_2017}, for example. (In the language of sequential testing, if $(p_t)_{t=1}^\infty$ is an always-valid $p$-value process, then so is $(\min_{s\leq t}p_s)_{t=1}^\infty$.)

However, the intersected intervals $\widetilde{\CI}_t$ may become empty at some
point. This is particularly likely if the underlying parameter is drifting over
time, contrary to the assumption of stationarity or identically-distributed
observations, and such a drift would be the likely interpretation of this event
in practice. In this non-stationary case, the non-intersected sequence is the
more sensible one to use. The solution of \citet{johari_peeking_2017} is to
``reset'' the experiment, discarding data accumulated up to that point, on the
rationale that such an event indicates that previous data are no longer relevant
to estimation of the current parameter of interest. However, this means that our
confidence sequence can go from a very high precision estimate at some time $t$
to knowing almost nothing at time $t+1$, which is difficult for an experimenter
to interpret and could lead to misleading inference just before the
reset. \Citet{jennison_interim_1989} make a case for the non-intersected
intervals on slightly different grounds, arguing that estimation at time $t$
ought to be a function of the sufficient statistic at that time. 
Shifting to the potential outcomes model in \cref{sec:ate}
neatly avoids this issue: because the estimand is changing at each time, the
non-intersected intervals are the only reasonable choice for estimating $\ATE_t$
and no conceptual difficulty remains.

\section{Summary and future work}\label{sec:discussion}

We have discussed four techniques for deriving curved uniform boundaries, each
improving upon past work, with careful attention paid to constants and to
practical issues. By building upon the general framework of
\citet{howard_exponential_2018}, we have emphasized the nonparametric
applicability of our boundaries. A leading example of the utility of this
approach is the general empirical-Bernstein bound, with an application to
sequential causal inference, and we have also shown how our framework
immediately yields novel results for matrix martingales.

\subsection{Other related work}\label{sec:other_related_work}

We introduced the method of mixtures and the epoch-based analyses in
\cref{sec:related_work}. Two other methods of extending the SPRT deserve
mention, though they are distinct from our approaches. First, the approach of
\citet{robbins_class_1972, robbins_expected_1974} examines
$\prod_i f_{\lambdahat_{i-1}}(X_i) / f_0(X_i)$ where $\lambdahat_{i-1}$ is a ``nonanticipating''
estimate based on $X_1, \dots, X_{i-1}$. This is similar to a generalized
likelihood ratio but modified to retain the martingale property
(cf. Wald \citep[section 10.5]{wald_sequential_1947},
\citep{lorden_nonanticipating_2005}). Second, the sequential generalized
likelihood ratio approach examines
$\sup_\lambda \prod_i f_\lambda(X_i) / f_0(X_i)$, which is not a martingale
under the null \citep{siegmund_sequential_1980, lai_optimal_1997,
  kulldorff_maximized_2011}.

The concept of \emph{test (super)martingales} expounded by
\citet{shafer_test_2011} is related to our methods for conducting inference
based on Ville's inequality applied to nonnegative supermartingales. Their main
example is the Beta mixture for i.i.d.\ Bernoulli observations, an example which
originated with \citet{ville_etude_1939} and discussed by
\citet{robbins_statistical_1970} and \citet{lai_confidence_1976}.  A recent
``safe testing'' framework of \Citet{grunwald_safe_2019} is also tightly
related. In terms of these frameworks, our work can be viewed as constructing
``safe confidence intervals'' (and thus safe tests) using nonparametric test
supermartingales.

A very different approach is that of group sequential
methods \citep{pocock_group_1977, obrien_multiple_1979, lan_discrete_1983,
  jennison_group_2000}. These methods rely on either exact discrete
distributions or asymptotics to assume exact normality of group increments,
either of which permits computation of sequential boundaries via numerical
integration. The resulting confidence sequences are tighter than ours, but lack
nonasymptotic guarantees or closed-form results and do not support continuous
monitoring.

A related problem is that of terminal confidence intervals, in which one assumes
a rigid stopping rule and wishes to construct a confidence interval upon
termination. \citet{siegmund_estimation_1978} gave an analytical treatment of
the problem; numerical methods are also available for group sequential tests
\citep[section 8.5]{jennison_group_2000}. However, the idea of a rigid stopping
rule is often restrictive.

\subsection{Future work}

We discuss in \cref{sec:banach} how our work may be extended to
martingales in smooth Banach spaces and real-valued, continuous-time
martingales. It may be fruitful to explore applications in those areas.

Our consideration of optimality has been limited to the discussion in
\cref{sec:admissibility}. It would be valuable to further explore various
optimality properties for nonasymptotic uniform bounds. For example,
%
it is standard in sequential testing to compute the
  expected sample size to reject a null under parametric
  alternatives. Though we target less restrictive assumptions,
  it may be instructive to compute bounds in special cases.
Second, a natural counterpoint to our uniform concentration bounds would be a set
  of uniform anticoncentration bounds. 
  This would yield a nonasymptotic extension of the ``lim inf''
  half of the classical LIL. \citet[Theorem
  3]{balsubramani_sharp_2014} gives one such interesting result.
Last, in practice, one will rarely require updated
inference after every observation, and may be content to
take observations in groups. Further, one may be satisfied with a finite time
horizon \cite{garivier2011context}.  This is the domain in which
group-sequential methods shine, but SPRT-based methods can be made competitive
by estimating the ``overshoot'' of the stopped supermartingale
\citep{lai_nonlinear_1977, lai_nonlinear_1979, siegmund_sequential_1985,
  whitehead_group_1983}.  It would be interesting to understand whether such
improvements work out in nonparametric settings.

\subsection*{Acknowledgments}

Howard thanks ONR Grant N00014-15-1-2367. Sekhon thanks ONR grants N00014-17-1-2176 and
N00014-15-1-2367. Ramdas thanks NSF grant DMS1916320. We thank Boyan Duan and Ian Waudby-Smith as
well as the referees/AE for useful suggestions.

\bibliographystyle{agsm}
\bibliography{best_arm}

\appendix

\section{Proofs of main results}\label{sec:main_proofs}

In this section we give proofs of our main results along with selected
discussion of and intuition for proof techniques.

\subsection{Proof of \cref{th:stitching}}\label{sec:proof_stitching}

The idea behind \cref{th:stitching} is to divide intrinsic time into
geometrically spaced epochs, $\eta^k \leq V_t < \eta^{k+1}$ for some $\eta >
1$. We construct a linear boundary within each epoch using
\cref{th:uniform_chernoff} and take a union bound over crossing events of the
different boundaries. The resulting, piecewise-linear boundary may then be upper
bounded by a smooth, concave function. Figure~\ref{fig:stitching} illustrates
the construction.

As discussed in \cref{sec:stitching}, the function $h$ determines the nominal
crossing probability $\alpha/h(k)$ allocated to the \kth epoch, and we have
mentioned the choices $h(k) = \eta^{sk} / (1 - \eta^{-s})$ and
$h(k) = (k + 1)^s \zeta(s)$. One may substitute a series converging yet more
slowly; for example, $h(k) \propto (k+2) \log^s(k+2)$ for $s > 1$ yields
\begin{align}
  \log h(\log_\eta V_t) = \log \log_\eta(\eta^2 V_t)
    + s \log \log \log_\eta(\eta^2 V_t)
    + \log\pfrac{\log^{1-s}(3/2)}{s-1},
  \label{eq:triple_log}
\end{align}
matching related analysis in \citet{darling_iterated_1967},
\citet{robbins_probability_1969}, \citet{robbins_statistical_1970}, and
\citet{balsubramani_sharp_2014}. In practice, the bound \eqref{eq:triple_log}
appears to behave like bound \eqref{eq:poly_stitching} with worse
constants. However, the fact that the stitching approach can recover key
theoretical results like these gives some indication of its power.

\begin{proof}[Proof of \cref{th:stitching}]
  We prove the result in the case $m = 1$ for simplicity. The general result may
  be obtained by considering $S_t / \sqrt{m}$ in place of $S_t$, $V_t / m$ in
  place of $V_t$, and $c / \sqrt{m}$ in place of $c$. See
  \cref{sec:change_units} for details.

  We first compute $\psi_G^{-1}(u)$ by taking the positive solution to the
  quadratic equation given by $\psi_G(\lambda) = u$, yielding
\begin{align}
  \psi_G^{-1}(u) &= -cu \pm \sqrt{c^2 u^2 + 2u}
    = \frac{2}{c + \sqrt{c^2 + 2/u}}, \label{eq:psi_G_inv}
\end{align}
where we have used the identity $\sqrt{1 + x} - 1 = \frac{x}{\sqrt{1 + x} + 1}$.
Let
\begin{align}
  K(u) \defineas \frac{\sqrt{2u}}{\psi_G^{-1}(u)}
    = \sqrt{1 + \frac{c^2 u}{2}} + c\sqrt{\frac{u}{2}}.
  \label{eq:K_u}
\end{align}
$K(u)$ will appear below. Now we start from the line-crossing inequality of
\cref{th:uniform_chernoff}: reparametrizing $r = \log \alpha^{-1}$, we have for
any $r > 0, \lambda > 0$
\begin{align}
\P\Bigg(
  \exists t \geq 1 : S_t \geq
  \underbrace{\frac{r + \psi_G(\lambda) V_t}{\lambda}}_{g_{\lambda,r}(V_t)}
\Bigg) \leq l_0 e^{-r}.
\label{eq:stitching_basic_bound}
\end{align}

We divide intrinsic time into epochs $\eta^k \leq V_t < \eta^{k+1}$ for each
$k = 0, 1, \dots$, and we will construct a linear boundary over each epoch by
carefully choosing values for $\lambda_k$ and $r_k$ and using the probability
bound \eqref{eq:stitching_basic_bound}. We choose $\lambda_k$ so that the
``standardized'' boundary takes equal values at both endpoints of the epoch:
$g_{\lambda_k,r_k}(\eta^k) / \eta^{k/2} = g_{\lambda_k,r_k}(\eta^{k+1}) /
\eta^{(k+1)/2}$. This equation is solved by
$\lambda_k = \psi_G^{-1}(r_k / \eta^{k+1/2})$, which yields, after some algebra,
\begin{align}
g_{\lambda_k,r_k}(v)
  &= K\pfrac{r_k}{\eta^{k+1/2}}
    \ebracket{\sqrt{\frac{\eta^{k+1/2}}{v}} + \sqrt{\frac{v}{\eta^{k+1/2}}}}
    \sqrt{\frac{r_k v}{2}}
    \label{eq:g_first_expression}
\end{align}
Our goal, after choosing $r_k$ below, is to upper bound this expression by a
function of $v$ alone, independent of $k$. Noting that the term in square
brackets in \eqref{eq:g_first_expression} reaches its maximum over the \kth
epoch at the endpoints, $v = \eta^k$ and $v = \eta^{k+1}$, and substituting the
expression \eqref{eq:K_u} for $K(u)$, we have
\begin{align}
g_{\lambda_k,r_k}(v)
  &\leq \eparen{\sqrt{1 + \frac{c^2 r_k}{2\eta^{k+1/2}}}
                + c \sqrt{\frac{r_k}{2\eta^{k+1/2}}}}
    \frac{\eta^{1/4} + \eta^{-1/4}}{\sqrt{2}} \sqrt{r_k v},
    \quad \text{for all } \eta^k \leq v < \eta^{k+1}.
\end{align}
The inequality $\eta^{k+1/2} \geq v / \sqrt{\eta}$ yields
\begin{align}
g_{\lambda_k,r_k}(v)
  &\leq \frac{\eta^{1/4} + \eta^{-1/4}}{\sqrt{2}}
    \eparen{\sqrt{r_k v + \frac{\sqrt{\eta} c^2 r_k^2}{2}}
            + c \frac{\eta^{1/4} r_k}{\sqrt{2}}} \\
  &= \sqrt{k_1^2 r_k v + k_2^2 c^2 r_k^2}
     + c k_2 r_k,
    \quad \text{for all } \eta^k \leq v < \eta^{k+1},
\end{align}
using the definition \eqref{eq:stitching_operator} of $k_1$ and $k_2$. Now let
$r_k = \log(l_0 h(k) / \alpha)$, which we choose to ensure total error
probability will be bounded by $\alpha$ via a union bound. Note that $h$ is
nondecreasing and ${k \leq \log_\eta V_t}$ over the epoch, so that
$r_k \leq \ell(v)$ over the epoch, recalling the definition
\eqref{eq:stitching_operator} of $\ell(v)$. We conclude
\begin{align}
g_{\lambda_k,r_k}(v)
  &\leq \sqrt{k_1^2 v \ell(v) + k_2^2 c^2 \ell^2(v)} + c k_2 \ell(v)
  = \Scal_\alpha(v),
\end{align}
for all $\eta^k \leq v < \eta^{k+1}$. This final expression no longer depends on
$k$, showing that the final boundary $\Scal_\alpha(v)$ majorizes the
corresponding linear boundary $g_{\lambda_k,r_k}(v)$ over each epoch
$\eta^k \leq v < \eta^{k+1}$ for $k = 0, 1, \dots$. Hence
\begin{align}
\Scal_\alpha(v) \geq \min_{k \geq 0} g_{\lambda_k,r_k}(v)
\quad \text{for all } v \geq 1.
\end{align}
But the first linear boundary $g_{\lambda_0,t_0}(v)$ passes through
$\Scal_\alpha(1)$ and has positive slope, which implies
\begin{align}
\Scal_\alpha(1 \bmax v) \geq \min_{k \geq 0} g_{\lambda_k,r_k}(v)
\quad \text{for all } v > 0.
\label{eq:stitching_majorize}
\end{align}
Now taking a union bound over the probability bounds given by
\eqref{eq:stitching_basic_bound} for $k = 0, 1, \dots$, we have
\begin{align}
\P\eparen{\exists t \geq 1 : S_t \geq \min_{k \geq 0} g_{\lambda_k,r_k}(V_t)}
  \leq l_0 \sum_{k=0}^\infty e^{-r_k} = \alpha \sum_{k=0}^\infty \frac{1}{h(k)}
  \leq \alpha.
\label{eq:stitching_union_bound}
\end{align}
Combining \eqref{eq:stitching_union_bound} with \eqref{eq:stitching_majorize}
proves that $v \mapsto \Scal_\alpha(1 \bmax v)$ is a sub-gamma uniform boundary
with crossing probability $\alpha$.

For the second statement \eqref{eq:late_crossing_decay}, we simply restrict the
union bound to epochs $k \geq \floor{\log_\eta V_t}$, which restricts the sum in
\eqref{eq:stitching_union_bound} accordingly.
\end{proof}

We have given a stitched bound which is constant for $v < m$, but inspection of
the proof shows that one may improve the bound to be linear with positive slope
on $v < m$, by extending the linear bound over the first epoch to cover all
$v > 0$. This seems of limited utility for theoretical work, and we recommend
other bounds over the stitched bound for practice, so we do not pursue this
point further.

The idea of taking a union bound over geometrically spaced epochs is standard in
the proof of the classical law of the iterated logarithm \citep[Theorem
8.5.1]{durrett_probability:_2017}. The idea has been extended to finite-time
bounds by \citet{darling_iterated_1967}, \citet{jamieson_lil_2014},
\citet{kaufmann_complexity_2014}, and \citet{zhao_adaptive_2016}, usually when
the observations are independent and sub-Gaussian; the technique is sometimes
called ``peeling''. Of course, \cref{th:stitching} generalizes these
constructions much beyond the independent sub-Gaussian case, but it also
achieves tighter constants for the sub-Gaussian setting. Here, we briefly
discuss how the improved constants arise.

Both \citet{jamieson_lil_2014} and \citet{zhao_adaptive_2016} construct a
constant boundary rather than a linear increasing boundary over each epoch. They
apply Doob's maximal inequality for submartingales \citep[Theorem
4.4.2]{durrett_probability:_2017}, as in
\citet[eq. 2.17]{hoeffding_probability_1963}, to obtain boundaries similar to
that of \citet{freedman_tail_1975}. As illustrated in \citet[Figure
2]{howard_exponential_2018}, the linear bounds from \cref{th:uniform_chernoff}
are stronger than corresponding Freedman-style bounds, and the additional
flexibility yields tighter constants.

Both \citet{darling_iterated_1967} and \citet{kaufmann_complexity_2014} use
linear boundaries within each epoch analogous to those of
\cref{th:uniform_chernoff}. Both methods share a great deal in common with ours,
and \citeauthor{darling_iterated_1967} give consideration to general
cumulant-generating functions. Recall from \cref{th:uniform_chernoff} that such
linear boundaries may be chosen to optimize for some fixed time $V_t = m$. Our
method chooses the linear boundary within each epoch to be optimal at the
geometric center of the epoch, i.e., at $V_t = \eta^{k+{1/2}}$, so that at both
epoch endpoints the boundary will be equally ``loose'', that is, equal multiples
of $\sqrt{V_t}$. \Citeauthor{darling_iterated_1967} choose the boundaries to be
tangent at the start of the epoch, hence their boundary is looser than ours at
the end of the epoch. \citeauthor{kaufmann_complexity_2014} choose the boundary
as we do, but appear to incur more looseness in the subsequent inequalities used
to construct a smooth upper bound.

\subsection{Proof of \cref{th:asymptotic_lil}}\label{sec:proof_asymptotic_lil}

Fix any $\epsilon > 0$ and choose $a > 0$ small enough that
$\psi(\lambda) \leq (1 + \epsilon) \lambda^2 / 2$ for all $\lambda \in (0,
a)$. Using the fact that $\psi_{G,c}(\lambda) \geq \lambda^2 / 2$ for
$c \geq 0$, we have $\psi(\lambda) \leq (1 + \epsilon) \psi_{G,1/a}(\lambda)$
for all $\lambda \in (0, a)$, so that $(S_t)$ is sub-gamma with scale $c = 1/a$
and variance process $((1 + \epsilon) V_t)$. Now \cref{th:stitching} shows that
\begin{align}
  \P\eparen{\sup_t V_t = \infty \text{ and }
    S_t \geq u((1 + \epsilon) V_t) \text{ infinitely often}} = 0,
\end{align}
where we may choose $u(v) \sim \sqrt{2 (1 + \epsilon) v \log \log v}$ (see
\eqref{eq:poly_stitching} and discussion thereafter), so that
$u((1 + \epsilon) v) \sim \sqrt{2 (1 + \epsilon)^2 v \log \log v}$. It follows
that
\begin{align}
  \limsup_{t \to \infty} \frac{S_t}{\sqrt{2 (1 + \epsilon)^2 V_t \log \log V_t}}
  \leq 1
  \quad \text{on } \ebrace{\sup_t V_t = \infty}.
\end{align}
As $\epsilon > 0$ was arbitrary, we are done. \qed

\subsection{Conjugate mixture proofs}
\label{sec:proof_mixtures}

\begin{proof}[Proof of \cref{th:basic_mixture}]
  Assume $(S_t)$ is sub-$\psi$ with variance process $(V_t)$, so that, for each
  $\lambda \in [0, \lambda_{\max})$, we have
  $\expebrace{\lambda S_t - \psi(\lambda) V_t} \leq L_t(\lambda)$ where
  $(L_t(\lambda))_{t=0}^\infty$ is a nonnegative supermartingale. We will show
  that $M_t \defineas \int L_t(\lambda) \d F(\lambda)$ is a supermartingale with
  respect to $(\Fcal_t)$.

  Formally, for this proof, we augment the underlying probability space with the
  random variable $\lambda$ having distribution $F$ over the Borel
  $\sigma$-field on $\R$, independent of everything else. For each $t$, we
  require $L_t$ to be a random variable on this product space, i.e., it must be
  product measurable. Now \cref{th:canonical_assumption} stipulates that
  $L_t \in \sigma(\lambda, \Fcal_t)$ and
  $\E\condparen{L_t}{\lambda, \Fcal_{t-1}} \leq L_{t-1}$ for each $t \geq 1$,
  and additionally, $\E\condparen{L_0}{\lambda} \leq l_0$ a.s. In other words,
  $(L_t)$ is a supermartingale with respect to the filtration given by
  $\Gcal_t \defineas \sigma(\lambda, \Fcal_t)$ on this augmented space.
  Finally, we have $M_t = \E\condparen{L_t}{\Fcal_t}$. These facts follow
  directly from the definition and properties of conditional expectation.

  We claim that $(M_t)$ is a supermartingale with respect to $(\Fcal_t)$ on this
  augmented space. Indeed,
  \begin{align}
  \E\condparen{M_t}{\Fcal_{t-1}}
    = \E\condparen{\E\condparen{L_t}{\Fcal_t}}{\Fcal_{t-1}}
    = \E\condparen{\E\condparen{L_t}{\lambda, \Fcal_{t-1}}}{\Fcal_{t-1}}
    \leq \E\condparen{L_{t-1}}{\Fcal_{t-1}}
  \end{align}
  by the supermartingale property, and this last expression is equal to
  $M_{t-1}$. Furthermore, $\E M_0 = \E \E\condparen{L_0}{\lambda} \leq l_0$
  since $\E\condparen{L_0}{\lambda} \leq l_0$ a.s., hence
  $\E \abs{M_t} = \E M_t \leq l_0$ for all $t$.

  Now \cref{th:canonical_assumption} and Ville's maximal inequality for
  nonnegative supermartingales \citep[exercise 4.8.2]{durrett_probability:_2017}
  yield
  \begin{align}
    \P\eparen{
      \exists t \geq 1:
      \int \expebrace{\lambda S_t - \psi(\lambda) V_t} \d F(\lambda)
      \geq \frac{l_0}{\alpha}}
    \leq \P\eparen{\exists t \geq 1: M_t \geq \frac{l_0}{\alpha}}
    \leq \alpha.
  \end{align}
  In other words, $\P(\exists t \geq 1: S_t \geq \Mcal_\alpha(V_t)) \leq \alpha$
  by the definition of $\Mcal_\alpha$, which is the desired conclusion.
\end{proof}

In the sub-Gaussian case, the following boundary is well-known \citep[example
2]{robbins_statistical_1970}.

\begin{proposition}[Two-sided normal mixture]\label{th:two_sided_normal_mixture}
  Suppose both $(S_t)$ and $(-S_t)$ are sub-Gaussian with variance process
  $(V_t)$. Fix $\alpha \in (0,1)$ and $\rho > 0$, and define
  \begin{align}
    u(v) \defineas \sqrt{
        (v + \rho)
        \log\eparen{\frac{l_0^2 (v + \rho)}{\alpha^2 \rho}}
      }. \label{eq:two_sided_normal_mixture_2}
  \end{align}
  Then $\P(\forall t \geq 1: \abs{S_t} < u(V_t)) \geq 1 - \alpha$.
\end{proposition}

We have included the bound in
\cref{fig:finite_lil_boundaries,fig:normalized_boundaries}; although its
$\Ocal(\sqrt{V_t \log V_t})$ rate of growth is worse than the finite LIL
discrete mixture bound, it can achieve tighter control over about three orders
of magnitude of intrinsic time. This makes the normal mixture preferable in many
practical situations when a sub-Gaussian assumption applies. When only a
one-sided sub-Gaussian assumption holds, the normal mixture still yields a
sub-Gaussian uniform boundary.

\begin{proposition}[One-sided normal mixture]\label{th:normal_mixture}
  For any $\alpha \in (0,1)$ and $\rho > 0$, the boundary
  \begin{align}
    \NM_\alpha(v) = \sup\ebrace{
      s \in \R : \sqrt{\frac{4 \rho}{v + \rho}}
      \expebrace{\frac{s^2}{2(v + \rho)}} \Phi\pfrac{s}{\sqrt{v + \rho}}
      < \frac{l_0}{\alpha}
    }. \label{eq:normal_mixture}
  \end{align}
  is a sub-Gaussian uniform boundary with crossing probability
  $\alpha$. Furthermore, we have the following closed-form upper bound:
  \begin{align}
    \NM_\alpha(v) \leq \widetilde{\NM}_\alpha(v) \defineas \sqrt{
      2(v + \rho)
      \log\eparen{\frac{l_0}{2 \alpha} \sqrt{\frac{v + \rho}{\rho}} + 1}
    }.\label{eq:normal_mixture_closed}
  \end{align}
\end{proposition}

The boundary $\NM_\alpha$ is easily evaluated to high precision by numerical
root-finding, and the closed-form approximation is excellent: numerical
calculations indicate that $\widetilde{\NM}_{0.025}(v) / \NM_{0.025}(v) < 1.007$
uniformly when $\rho = 1$, for example.

\begin{proof}[Proof of \cref{th:normal_mixture}]
\label{sec:proof_normal_mixture}
To obtain the explicit upper bound $\widetilde{\NM}_\alpha$ in
\eqref{eq:normal_mixture_closed} from the exact boundary
\eqref{eq:normal_mixture}, we use the inequality $1 - \Phi(x) \leq e^{-x^2/2}$
for $x > 0$, which follows from a standard Cram\'er-Chernoff bound. This implies
\begin{align}
\sqrt{\frac{4 \rho}{v + \rho}}
  \expebrace{\frac{s^2}{2(v + \rho)}} \Phi\pfrac{s}{\sqrt{v + \rho}}
\geq \sqrt{\frac{4 \rho}{v + \rho}} \ebracket{
  \expebrace{\frac{s^2}{2(v + \rho)}} - 1
},\quad \text{for } s > 0.
\end{align}
We set the RHS equal to $l_0 / \alpha$ and solve to conclude
\begin{align}
\NM_\alpha(v) \leq \sqrt{
  2(v + \rho)
  \log\eparen{\frac{l_0}{2 \alpha} \sqrt{\frac{v + \rho}{\rho}} + 1}
} = \widetilde{\NM}_\alpha(v),
\end{align}
so long as $\NM_\alpha(v) > 0$. But we are guaranteed that $\NM_\alpha(v) > 0$, because the LHS of
the inequality in \eqref{eq:normal_mixture} is increasing in $s$ on $s \geq 0$ and no larger than
one when $s = 0$, while the RHS $l_0/\alpha \geq 1$.

The fact that $\NM_\alpha$ is a sub-Gaussian uniform boundary follows directly
from \cref{th:basic_mixture}, and therefore $\widetilde{\NM}_\alpha$ is as
well.
\end{proof}

When a sub-Bernoulli condition holds, as with bounded observations, the
following beta-binomial boundary is tighter than the normal mixture. Simpler
versions of this boundary have long been studied for i.i.d.\ Bernoulli sampling
\citep{ville_etude_1939, robbins_statistical_1970, lai_confidence_1976,
  shafer_test_2011}. Below, $B_x(a,b) = \int_0^x p^{a-1} (1-p)^{b-1} \d p$
denotes the incomplete Beta function, whose implementation is available in
statistical software packages; $B_1$ is the ordinary Beta function.

\begin{proposition}[Two-sided beta-binomial mixture]
\label{th:two_sided_beta_mixture}
Suppose $(S_t)$ is sub-Bernoulli with variance process $(V_t)$ and range
parameters $g,h$, while $(-S_t)$ is sub-Bernoulli with variance process $(V_t)$
and range parameters $h,g$. Fix any $\rho > gh$, let $r = \rho - gh$, and define
  \begin{align}
  f_{g,h}(v) &~\defineas~
    \sup\ebrace{s \in \left[0, \frac{r+v}{g}\right) :
                m_{g,h}(s, v) < \frac{l_0}{\alpha}},\\
  \text{ where ~}~ m_{g,h}(s,v) &~\defineas~
    \frac{(g+h)^{v/gh}}{\ebracket{g^{v/h + s} h^{v/g - s}}^{1/(g+h)}} \cdot
    \frac{B_1\eparen{\frac{r+v-gs}{g (g+h)}, \frac{r+v+hs}{h(g+h)}}}
         {B_1\eparen{\frac{r}{g(g+h)}, \frac{r}{h(g+h)}}}.
      \label{eq:beta_mixture}
  \end{align}
  Then
  $\P(\forall t \geq 1: -f_{g,h}(V_t) < S_t < f_{h,g}(V_t)) \geq 1 - \alpha$.
\end{proposition}

As with the normal mixture, we have a one-sided variant as well.

\begin{proposition}[One-sided beta-binomial mixture]
\label{th:one_sided_beta_mixture}
Fix any $g, h > 0$, $\alpha \in (0,1)$, and $\rho > gh$. Let $r = \rho - gh$ and
define
  \begin{align}
  f_{g,h}(v) &~\defineas~
    \sup\ebrace{s \in \left[0, \frac{r+v}{g}\right) :
                m_{g,h}(s, v) < \frac{l_0}{\alpha}},\\
  \text{ where ~}~ m_{g,h}(s,v) &~\defineas~
    \frac{(g+h)^{v/gh}}{\ebracket{g^{v/h + s} h^{v/g - s}}^{1/(g+h)}} \cdot
    \frac{B_{h/(g+h)}\eparen{\frac{r+v-gs}{g (g+h)}, \frac{r+v+hs}{h(g+h)}}}
         {B_{h/(g+h)}\eparen{\frac{r}{g(g+h)}, \frac{r}{h(g+h)}}}.
      \label{eq:one_sided_beta_mixture}
  \end{align}
  Then $f_{g,h}$ is a sub-Bernoulli uniform boundary with crossing probability
  $\alpha$ and range parameters $g,h$.
\end{proposition}

In the sub-Bernoulli case, we first rewrite the exponential process
$\expebrace{\lambda S_t - \psi_B(\lambda) V_t}$ in terms of the transformed
parameter $p = [1 + (h/g)e^{-\lambda}]^{-1}$. This is motivated by the transform
from the canonical parameter to the mean parameter of a Bernoulli family, but
keep in mind that we make no parametric assumption here, these are merely
analytical manipulations. Then a truncated Beta distribution on
$p \in [g/(g+h), 1]$ yields the one-sided beta-binomial uniform boundary, while
an untruncated mixture yields the two-sided boundary.

\begin{proof}[Proof of \cref{th:one_sided_beta_mixture,th:two_sided_beta_mixture}]
\label{sec:proof_beta_mixture}

For simplicity of notation, we will assume here that the problem has been scaled
so that $g + h = 1$, e.g., by replacing $X_t$ with $X_t / (g + h)$. Using the
sub-Bernoulli $\psi$ function
$\psi_B(\lambda) = \frac{1}{gh} \log\eparen{ge^{h\lambda} + he^{-g\lambda}}$,
the exponential integrand in our mixture is
\begin{align}
  \expebrace{
    \lambda s - \frac{v}{gh} \log\eparen{ge^{h\lambda} + he^{-g\lambda}}}
  = \frac{p^{v/h + s} (1-p)^{v/g - s}}{g^{v/h+s} h^{v/g-s}},
\end{align}
after substituting the one-to-one transformation
\begin{align}
  p = p(\lambda) \defineas \frac{ge^{h\lambda}}{ge^{h\lambda} + he^{-g\lambda}},
  \quad \text{so that} \quad
  \lambda = \log\pfrac{ph}{(1-p) g},
\end{align}
followed by some algebra. We wish to integrate against a Beta mixture density on
$p$ with parameters $r/h$ and $r/g$, which has mean $p = g$, corresponding to
$\lambda = 0$. For \cref{th:one_sided_beta_mixture}, we must also truncate to
$\lambda \geq 0$, i.e., to $p \geq g$. The appropriately normalized mixture
integral is then
\begin{align}
  \frac{1}{g^{v/h+s} h^{v/g-s}} \cdot
  \frac{\int_g^1 p^{v/h + s + r/h-1} (1-p)^{v/g - s + r/g-1} \d p}
       {\int_g^1 p^{r/h-1} (1-p)^{r/g-1} \d p}
  = \frac{1}{g^{v/h+s} h^{v/g-s}} \cdot
    \frac{B_h\eparen{\frac{r+v}{g} - s, \frac{r+v}{h} + s}}
         {B_h\eparen{\frac{r}{g}, \frac{r}{h}}},
\end{align}
using the fact that
$B_x(a, b) = \int_0^x p^{a-1} (1-p)^{b-1} \d p = \int_{1-x}^1 p^{b-1} (1 -
p)^{a-1} \d p$. This gives the closed-form mixture
\eqref{eq:one_sided_beta_mixture}. (To obtain the formula for general
$g+h \neq 1$, substitute $g/(g+h)$ for $g$, $h/(g+h)$ or $h$, $s/(g+h)$ for $s$,
$v/(g+h)^2$ for $v$, and $r/(g+h)^2$ for $r$.)

The proof of \cref{th:two_sided_beta_mixture} is nearly identical, but we
integrate over the full Beta mixture rather than truncating.

To verify that our choice of $r$ ensures that $\lambda$ has approximate
precision $\rho$ under the full (not truncated) mixture distribution, we use the
delta method to calculate the approximate variance of $\lambda$ for large $r$
based on the variance of $p$ under the full Beta mixture:
\begin{align}
  \Var \lambda \approx
  \ebracket{\pfrac{1}{p(1-p)}^2}_{p = g}
    \cdot \frac{gh}{\frac{r}{gh} + 1}
  = \frac{1}{r+gh}.
\end{align}
Setting this equal to $1/\rho$ yields $r = \rho - gh$ as desired.
\end{proof}

When tails are heavier than Gaussian, the normal mixture boundary is not
applicable. However, the following sub-exponential mixture boundary, based on a
gamma mixing density, is universally applicable, as described in
\cref{th:universal}.  Like the normal mixture, the gamma-exponential mixture is
unimprovable as described in \cref{sec:admissibility}. Below we make use of the
regularized lower incomplete gamma function
$\gamma(a,x) \defineas (\int_0^x u^{a-1} e^{-u} \d u) / \Gamma(a)$, available in
standard statistical software packages.

\begin{proposition}[Gamma-exponential mixture]\label{th:gamma_mixture}
  Fix $c > 0, \rho > 0$ and define
  \begin{align}
  \GE_\alpha(v) &~\defineas~
    \sup\ebrace{s \geq 0 : m(s, v) < \frac{l_0}{\alpha}},\\
  \text{ where ~}~ m(s,v) &~\defineas~
    \frac{\pfrac{\rho}{c^2}^{\frac{\rho}{c^2}}}
      {\Gamma\pfrac{\rho}{c^2}
           \gamma\eparen{\frac{\rho}{c^2}, \frac{\rho}{c^2}}}
    \frac{\Gamma\pfrac{v+\rho}{c^2}
           \gamma\eparen{\frac{v+\rho}{c^2}, \frac{cs+v+\rho}{c^2}}}
      {\pfrac{cs+v+\rho}{c^2}^{\frac{v+\rho}{c^2}}}
    \expebrace{\frac{cs+v}{c^2}}.
      \label{eq:gamma_mixture}
  \end{align}
Then $\GE_\alpha$ is a sub-exponential uniform boundary
  with crossing probability $\alpha$ for scale $c$.
\end{proposition}

The gamma-exponential mixture is the result of evaluating the mixture integral
in \eqref{eq:mixture_operator} with mixing density
\begin{align}
\frac{\d F}{\d \lambda} =
  \frac{1}{\gamma(\rho/c^2, \rho/c^2)}
  \frac{(\rho/c)^{\rho/c^2}}{\Gamma(\rho/c^2)}
  (c^{-1} - \lambda)^{\rho/c^2-1}
  e^{-\rho(c^{-1}-\lambda)/c}.
\end{align}
This is a gamma distribution with shape $\rho/c^2$ and scale $\rho/c$ applied to
the transformed parameter $u = c^{-1} - \lambda$, truncated to the support
$[0,c^{-1}]$. The distribution has mean zero and variance equal to $1/\rho$,
making it comparable to the normal mixture distribution used above. As
$\rho \to \infty$, the gamma mixture distribution converges to a normal
distribution and concentrates about $\lambda = 0$, the regime in which
$\psi_E(\lambda) \sim \psi_N(\lambda)$, which gives some intuition for why the
gamma-exponential mixture recovers the normal mixture when $\rho \gg c^2$.

\begin{proof}[Proof of \cref{th:gamma_mixture}]
We invoke \cref{th:basic_mixture} with the integral
\begin{align}
  &\int_0^{1/c} \expebrace{\lambda s - \psi_E(\lambda) v}
             f(\lambda) \d \lambda,
            \label{eq:gamma_mixture_raw_integral} \\
  \text{where } f(\lambda) &=
    \frac{1}{\gamma(\rho/c^2, \rho/c^2)}
    \frac{(\rho/c)^{\rho/c^2}}{\Gamma(\rho/c^2)}
    (c^{-1} - \lambda)^{\rho/c^2-1}
    e^{-\rho(c^{-1}-\lambda)/c}.
\end{align}
Because $\psi_E(\lambda) \geq 0$ for all $\lambda$, $f(\lambda)$ is a probability density, and
$l_0 \geq 1$ by \cref{th:canonical_assumption}, the integral \eqref{eq:gamma_mixture_raw_integral}
evaluated at $s = 0$ is no larger than $1 < l_0 / \alpha$ (recall $\alpha < 1$). Also, the integral
\eqref{eq:gamma_mixture_raw_integral} is monotonically increasing in $s$ since the integrand
is. Hence, when applying \cref{th:basic_mixture}, we can restrict the supremum to $s \geq 0$. Then
we need only show that, for $s \geq 0$, $m(s,v)$ as defined in \eqref{eq:gamma_mixture} is equal to
the integral \eqref{eq:gamma_mixture_raw_integral}, and the fact that $\GE_\alpha$ is a
sub-exponential uniform boundary follows from \cref{th:basic_mixture}. (The closed form
\eqref{eq:gamma_mixture} is valid whenever $cs + v + \rho > 0$, and in particular for all
$s \geq 0$; when $cs + v + \rho \leq 0$ the lower incomplete gamma function in
\eqref{eq:gamma_mixture} is undefined, although the integral \eqref{eq:gamma_mixture_raw_integral}
remains finite. The change of variables below likewise requires $cs + v + \rho > 0$.)

Proving that $m(s,v)$ as defined in \eqref{eq:gamma_mixture} is equal to the integral
\eqref{eq:gamma_mixture_raw_integral} is an exercise in calculus. Substituting the definition of
$\psi_E$ and removing common terms, it suffices to show that
\begin{align}
c^{-\rho/c^2}
  \frac{\Gamma\pfrac{v+\rho}{c^2}
        \gamma\eparen{\frac{v+\rho}{c^2}, \frac{cs+v+\rho}{c^2}}}
       {\pfrac{cs+v+\rho}{c^2}^{\frac{v+\rho}{c^2}}}
  e^{(cs+v)/c^2}
= \int_0^{1/c} (1-c\lambda)^{v/c^2} e^{\lambda(s+v/c)}
    (c^{-1}-\lambda)^{\rho/c^2-1} e^{-\rho(c^{-1}-\lambda)/c} \d \lambda.
\end{align}
After change of variables $u = \pfrac{cs + v + \rho}{c}(c^{-1} - \lambda)$, the
right-hand side is equal to
\begin{align}
\pfrac{cs+v+\rho}{c}^{-\frac{v+\rho}{c^2}} c^{v/c^2} e^{(cs+v)/c^2}
  \int_0^{(cs+v+\rho)/c^2} u^{(v+\rho)/c^2-1} e^{-u} \d u.
\end{align}
Now the definition of the regularized lower incomplete gamma function and a bit
of algebra finishes the argument.
\end{proof}

A similar mixture boundary holds in the sub-Poisson case,
making use of the regularized upper incomplete gamma function
$\uppergamma(a, x) \defineas (\int_x^\infty u^{a-1} e^{-u} \d u) / \Gamma(a)$.
\begin{proposition}[Gamma-Poisson mixture]\label{th:gamma_poisson_mixture}
  Fix $c > 0, \rho > 0$ and define
  \begin{align}
  \GP_\alpha(v) &~\defineas~
    \sup\ebrace{s \geq 0 : m(s, v) < \frac{l_0}{\alpha}},\\
  \text{ where ~}~ m(s,v) &~\defineas~
    \frac{\pfrac{\rho}{c^2}^{\rho/c^2}}
      {\Gamma\pfrac{\rho}{c^2}
           \uppergamma\eparen{\frac{\rho}{c^2}, \frac{\rho}{c^2}}}
    \frac{\Gamma\pfrac{cs+v+\rho}{c^2}
           \uppergamma\eparen{\frac{cs+v+\rho}{c^2}, \frac{v+\rho}{c^2}}}
      {\pfrac{v+\rho}{c^2}^{(cs+v+\rho)/c^2}}
    \expebrace{\frac{v}{c^2}}.
      \label{eq:gamma_poisson_mixture}
  \end{align}
  Then $\GP_\alpha$ is a sub-Poisson uniform boundary with crossing probability
  $\alpha$ for scale $c$.
\end{proposition}

\begin{proof}[Proof of \cref{th:gamma_poisson_mixture}]
\label{sec:proof_gamma_poisson_mixture}

The proof follows the same contours as that of
\cref{th:gamma_mixture}. Using the sub-Poisson $\psi$ function
$\psi_P(\lambda) = c^{-2} (e^{c\lambda} - c\lambda - 1)$, the exponential
integrand in our mixture is
\begin{align}
  \expebrace{\lambda s - v \pfrac{e^{c\lambda} - c\lambda - 1}{c^2}}
  = \theta^{(cs+v)/c^2} e^{(1-\theta) v/c^2},
\end{align}
after substituting the one-to-one transformation
$\theta = \theta(\lambda) \defineas e^{c\lambda}$, so that
$\lambda = c^{-1} \log \theta$. We integrate against a gamma mixing distribution
on $\theta$ with shape and scale parameters both equal to
$\beta \defineas \rho / c^2$, truncated to $\theta \geq 1$, so that
$\lambda \geq 0$:
\begin{align}
  e^{v/c^2} \,
  \frac{\int_1^\infty \theta^{(cs+v+\rho)/c^2-1} e^{-(v+\rho)\theta/c^2}
        \d \theta}
       {\int_1^\infty \theta^{\rho/c^2-1} e^{-\rho\theta/c^2} \d \theta}
  = \frac{\pfrac{\rho}{c^2}^{\rho/c^2}}{\Gamma\pfrac{\rho}{c^2}}
    \cdot
    \frac{\Gamma\pfrac{cs+v+\rho}{c^2}}{\pfrac{v+\rho}{c^2}^{(cs+v+\rho)/c^2}}
    \cdot
    \frac{\uppergamma\eparen{\frac{cs+v+\rho}{c^2}, \frac{v+\rho}{c^2}}}
         {\uppergamma\eparen{\frac{\rho}{c^2}, \frac{\rho}{c^2}}}
    \expebrace{\frac{v}{c^2}}.
\end{align}
This yields the closed-form mixture \eqref{eq:gamma_poisson_mixture}. As in the proof of
\cref{th:gamma_mixture}, since $\psi_P \geq 0$ the mixture integral is at most $1 < l_0 / \alpha$
at $s = 0$ and is increasing in $s$, so in applying \cref{th:basic_mixture} we may restrict to
$s \geq 0$. The identity above holds whenever $cs + v + \rho > 0$, in particular for all
$s \geq 0$; for $cs + v + \rho \leq 0$ the first argument of $\Gamma$ and $\uppergamma$ in
\eqref{eq:gamma_poisson_mixture} is nonpositive, outside the domain of their integral
definitions, though the mixture integral itself remains finite.

To verify that our choice of $\beta$ ensures that $\lambda$ has approximate precision
$\rho$ under the full (not truncated) mixture distribution, we use the delta
method to calculate the approximate variance of $\lambda$ for large $\beta$
based on the variance of $\theta$ under the full gamma mixture:
\begin{align}
  \Var \lambda \approx
    \ebracket{\frac{1}{c^2 \theta^2}}_{\theta = 1} \cdot \frac{1}{\beta}
    = \frac{1}{\rho}.
\end{align}
\end{proof}

\subsection{Proof of \cref{th:mixture_rate}}\label{sec:proof_mixture_rate}

Under the conditions of \cref{th:mixture_rate}, we have
\begin{align}
  m(s,v) = \int_0^{\lambda_{\max}}
    \expebrace{\lambda s - \psi(\lambda) v} f(\lambda) \d \lambda.
\end{align}
Note that $m(s,v)$ is nondecreasing in $s$ and nonincreasing in $v$ (since
$\psi \geq 0$ by our assumptions on $\psi$).

Choose $\delta \in (0, \lambda_{\max})$ so that $\psi$ has three continuous
derivatives and $f$ is continuous and positive on $[0,\delta)$; such a value of
$\delta$ must exist by conditions (i) and (ii). Before proving
\cref{th:mixture_rate}, we state several lemmas.

\begin{lemma}\label{th:mixture_linear}
  Under the conditions of \cref{th:mixture_rate}, for any
  $b \in (0, \psi'(\delta))$, we have $m(bv, v) < \infty$ and
  $m(bv, v) \to \infty$ as $v \to \infty$.
\end{lemma}
\begin{proof}
  Observe
  \begin{align}
    m(bv, v) = \int_0^{\lambda_{\max}}
      \expebrace{v [\lambda b - \psi(\lambda)]} f(\lambda) \d \lambda.
  \end{align}
  Note
  $\frac{\d}{\d \lambda}\ebracket{\lambda b - \psi(\lambda)} = b -
  \psi'(\lambda) < 0$ for all $\lambda \geq \delta$ by our condition on
  $b$. Hence the integrand $\expebrace{v [\lambda b - \psi(\lambda)]}$ is
  decreasing on $\lambda \geq \delta$ and bounded above by
  $e^{v \delta b}$ on $\lambda \leq \delta$ (since $\psi \geq 0$). The
  integrand is therefore uniformly bounded on $[0, \lambda_{\max})$, so that
  $m(bv, v) < \infty$.

  Now Laplace's asymptotic approximation \citep[Chapter VII.2, Theorem
  2b]{widder_laplace_1942} yields
  \begin{align}
    \int_0^\delta \expebrace{v \ebracket{\lambda b - \psi(\lambda)}}
      f(\lambda) \d \lambda
    \sim \frac{C e^{v \psi^\star(b)}}{\sqrt{v}},
    \quad \text{ as } v \to \infty, \label{eq:laplace_expansion}
  \end{align}
  where $C > 0$ is a constant not depending on $v$. (The condition
  $b < \psi'(\delta)$ ensures that the maximizer of $\lambda b - \psi(\lambda)$
  lies within $[0, \delta)$.) Since the LHS of \eqref{eq:laplace_expansion}
  lower bounds $m(bv, v)$ while the RHS diverges as $v \to \infty$, we must have
  $m(bv, v) \to \infty$ as $v \to \infty$.
\end{proof}

\begin{lemma}\label{th:mixture_root}
  Under the conditions of \cref{th:mixture_rate},
  $m(\Mcal_\alpha(v), v) = l_0/\alpha$ for all $v$ sufficiently large.
\end{lemma}
\begin{proof}
  Let $\Ccal(v) \defineas [0, \psi'(\delta) v)$ for $v >
  0$. \Cref{th:mixture_linear} shows that $m(s, v) < \infty$ for all
  $s \in \Ccal(v)$. Since $m(s,v)$ is nondecreasing in $s$, we may apply
  dominated convergence to find that $s \mapsto m(s, v)$ is continuous at all
  $s \in \Ccal(v)$. Condition (i) of \cref{th:mixture_rate} implies
  $\psi \geq 0$, so that $m(0, v) \leq 1 \leq l_0/\alpha$ for all $v$. Finally,
  \cref{th:mixture_linear} shows that $\sup_{s \in \Ccal(v)} m(s, v) \to \infty$
  as $v \to \infty$. Hence, for $v$ sufficiently large, there exists
  $s \in \Ccal(v)$ such that $m(s, v) > l_0/\alpha$.

  We have argued that, for any sufficiently large $v$,
  $m(0, v) \leq l_0/\alpha < m(\bar{s}, v) < \infty$ for some
  $\bar{s} < \psi'(\delta) v$, and $m(\cdot, v)$ is continuous on
  $[0, \bar{s}]$. The conclusion follows from the definition
  \eqref{eq:mixture_operator} of $\Mcal_\alpha$.
\end{proof}

\begin{lemma}\label{th:M_blows_up}
  Under the conditions of \cref{th:mixture_rate},
  $\lim_{v \to \infty} \Mcal_\alpha(v) = \infty$.
\end{lemma}
\begin{proof}
  Suppose for the sake of contradiction that there exists $a > 0$ such that
  $\Mcal_\alpha(v) \leq a$ for all $v$. Then, since $m(s,v)$ is nondecreasing in
  $s$, $m(\Mcal_\alpha(v), v) \leq m(a, v)$ for all $v$. But for sufficiently
  large $v$, we can write $a = bv$ for some $b < \psi'(\delta)$, so that
  \cref{th:mixture_linear} implies $m(a, v) < \infty$ for sufficiently large
  $v$. Since condition (i) of \cref{th:mixture_rate} implies $\psi \geq 0$, we
  have $m(s,v)$ is decreasing in $v$, and dominated convergence yields
  $m(a, v) \to 0$ as $v \to \infty$. But this implies
  $m(\Mcal_\alpha(v), v) \to 0$, contradicting \cref{th:mixture_root}.

  We have shown that $\limsup_{v \to \infty} \Mcal_\alpha(v) = \infty$. But
  since $m(s,v)$ is nondecreasing in $s$ and nonincreasing in $v$,
  \cref{th:mixture_root} implies $\Mcal_\alpha(v)$ must be nondecreasing in
  $v$. It follows that $\lim_{v \to \infty} \Mcal_\alpha(v) = \infty$.
\end{proof}

\begin{lemma}\label{th:laplace_lemma}
  Under the conditions of \cref{th:mixture_rate}, $\Mcal_\alpha(v) = o(v)$.
\end{lemma}
\begin{proof}
  Suppose for the sake of contradiction that $\Mcal_\alpha(v) \geq b v$ for all
  $v$ sufficiently large, for some $0 < b < \psi'(\delta)$. Then (again using
  the fact that $m(s,v)$ is nondecreasing in $s$)
  $\lim_{v \to \infty} m(\Mcal_\alpha(v), v) \geq \lim_{v \to \infty} m(b v, v)
  = \infty$ by \cref{th:mixture_linear}, contradicting \cref{th:mixture_root}.
\end{proof}

\begin{proof}[Proof of \cref{th:mixture_rate}]
  We invoke Theorem 4 of \citet{fulks_generalization_1951}, setting Fulks' $h$
  equal to our $v$, Fulks' $k$ equal to our $\Mcal_\alpha(v)$, Fulks' $\phi$
  equal to our $\psi$, Fulks' $\psi$ equal to the identity function, Fulks' $f$
  equal to our $f$, and Fulks' $b$ equal to our $\lambda_{\max}$. Fulks'
  assumptions (A1)-(A4) now read as follows.
  \begin{enumerate}
  \item[(A1)] requires $\psi(0) = \psi'(0_+) = 0$, $\psi''(0_+) > 0$, $\psi$ has
    three continuous derivatives in a neighborhood of the origin, and $\psi$ is
    positive and nondecreasing on $(0, \lambda_{\max})$.
  \item[(A2)] requires conditions on the identity function which are trivially
    satisfied.
  \item[(A3)] requires $f$ to be integrable and to be continuous and positive at
    the origin.
  \item[(A4)] requires $\Mcal_\alpha(v) \to \infty$ as $v \to \infty$ and
    $\Mcal_\alpha(v) = o(v)$.
  \end{enumerate}
  (A1) and (A3) are satisfied by conditions (i) and (ii) of
  \cref{th:mixture_rate}. (A4) is satisfied by
  \cref{th:M_blows_up,th:laplace_lemma}. For Fulks' Theorem 4, it remains to
  verify that $\sqrt{v} = o(\Mcal_\alpha(v))$. But if this were not true, then
  we could apply Theorem 1 or Theorem 2 of \citet{fulks_generalization_1951} to
  conclude that $m(\Mcal_\alpha(v), v) \to 0$ as $v \to \infty$, contradicting
  \cref{th:mixture_root}. So Fulks' Theorem 4 yields
  \begin{align}
    m(\Mcal_\alpha(v), v) \sim
      f(0) \sqrt{\frac{2\pi}{cv}} \expebrace{\frac{\Mcal_\alpha^2(v)}{2cv}}.
  \end{align}
  Using \cref{th:mixture_root} to set $m(\Mcal_\alpha(v), v) = l_0/\alpha$, we
  may write
  \begin{align}
    f(0) \sqrt{\frac{2\pi}{cv}} \expebrace{\frac{\Mcal_\alpha^2(v)}{2cv}}
    = \frac{l_0 e^{o(1)}}{\alpha},
  \end{align}
  which can be rearranged into the desired conclusion.
\end{proof}

We have proved the result for one-sided bounds, but a nearly-identical argument
applies to two-sided bounds such as \cref{th:two_sided_beta_mixture}.

\subsection{Proof of \cref{th:discrete_mixture}}
\label{sec:proof_discrete_mixture}

\begin{figure}
\centering
\includegraphics{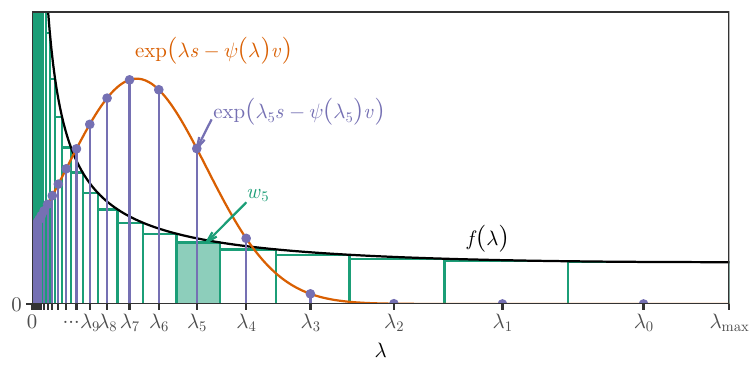}
\caption{Illustration of \cref{th:discrete_mixture}. Mixture density
  $f(\lambda)$ is discretized on a grid $(\lambda_k)_{k=0}^\infty$ which gets
  finer as $\lambda \downarrow 0$. Resulting discrete mixture weights are
  represented by areas within green bars. Integrand
  $\expebrace{\lambda s - \psi(\lambda) v}$ is evaluated at grid points
  $\lambda_k$, illustrated by purple points. Multiplying one integrand
  evaluation $\expebrace{\lambda_k s - \psi(\lambda_k) v}$ by the corresponding
  weight $w_k$ gives one term of the sum
  \eqref{eq:discrete_mixture_sum}. \label{fig:discrete_mixture}}
\end{figure}

Recall the discrete mixture support points and weights,
\begin{align}
    \lambda_k \defineas \frac{\overlambda}{\eta^{k+1/2}}
    \quad \text{and} \quad
    w_k \defineas
      \frac{\overlambda (\eta - 1) f(\lambda_k\sqrt{\eta})}{\eta^{k+1}}
    \quad \text{for} \quad k = 0, 1, 2, \dots.
\end{align}
Figure~\ref{fig:discrete_mixture} illustrates the construction. To see
heuristically why the exponentially-spaced grid $\lambda_k = \Ocal(\eta^{-k})$
makes sense, observe that the integrand
$\expebrace{\lambda s - \lambda^2 v / 2}$ is a scaled normal density in
$\lambda$ with mean $s/v$ and standard deviation $1/\sqrt{v}$. In the regime
relevant to our curved boundaries, $s$ is of order $\sqrt{v}$, ignoring
logarithmic factors. Hence the integrand at time $v$ has both center and spread
of order $1/\sqrt{v}$, so as $v \to \infty$, the relevant scale of the integrand
shrinks. With the grid $\lambda_k = \Ocal(\eta^{-k})$ we have
$\lambda_k - \lambda_{k+1} = \Ocal(\lambda_k)$, ensuring that the resolution of
the grid around the peak of the integrand matches the scale of the integrand as
$v \to \infty$.

The discrete mixture bound is a valid mixture boundary in its own right, based
on a discrete mixing distribution, but we may wish to know how well it
approximates the continuous-mixture boundary from which it is derived. To
illustrate the accuracy of the discrete mixture construction, we compare it to
the one-sided normal mixture bound, \cref{th:normal_mixture}.  By using the same
half-normal mixing density in \cref{th:discrete_mixture} and setting
$\eta = 1.05$, $\overlambda = 100$, we may evaluate a corresponding discrete
mixture bound $\Mtil_\alpha$. With $\rho = 14.3$, $\alpha = 0.05$ and $l_0 = 1$,
numerical calculations indicate that
$\Mtil_\alpha(v) / \NM_\alpha(v) \leq 1.004$ for $1 \leq v \leq 10^6$,
suggesting that \cref{th:discrete_mixture} gives an excellent conservative
approximation to the corresponding continuous mixture boundary over a large
practical range. Of course, when a closed form is available as in
\cref{th:normal_mixture}, one should use it in practice. But an exact closed
form integral is rarely available as it is in \cref{th:normal_mixture}, and
substantial looseness often accompanies closed-form approximations which
provably maintain crossing probability guarantees. In such cases, unless a
closed form is required, \cref{th:discrete_mixture} is preferable. See
figure~\ref{fig:finite_lil_all} for an example; in this figure, the bounds of
\citet{balsubramani_sharp_2014} and \citet{darling_further_1968} involve
closed-form mixture integral approximations.

\begin{proof}[Proof of \cref{th:discrete_mixture}]
Because $f$ is nonincreasing, $f(\lambda) \geq f(\lambda_k \sqrt{\eta})$ on the
interval $[\lambda_k / \sqrt{\eta}, \lambda_k \sqrt{\eta}]$, which has width
$\lambda_{\max} (\eta - 1) / \eta^{k+1} = w_k / f(\lambda_k \sqrt{\eta})$. Hence
$\sum_{k=0}^\infty w_k \leq \int_0^\infty f(\lambda) \d \lambda = 1$. Let $G$ be
a discrete distribution which places mass $w_k / \sum_{j=0}^\infty w_j$ at the
point $\lambda_k$. By \cref{th:basic_mixture}, we know the mixture bound
$\Mcal_\alpha$ applied to the discrete mixture distribution $G$ yields a
sub-$\psi$ uniform boundary with crossing probability $\alpha$. But
\begin{align}
\sum_{k=0}^\infty w_k \expebrace{\lambda_k s - \psi(\lambda_k) v}
  \leq \int \expebrace{\lambda s - \psi(\lambda) v} \d G(\lambda),
\end{align}
so $\Mtil_\alpha \geq \Mcal_\alpha$. That is, our discrete mixture approximation
$\Mtil_\alpha$ is a conservative overestimate of a corresponding exact mixture
boundary $\Mcal_\alpha$, and can only have a lower crossing probability. So the
discrete mixture bound $\Mtil_\alpha$ satisfies the desired probability
inequality $\P(\exists t : S_t \geq \Mtil_\alpha(V_t)) \leq \alpha$.
\end{proof}

\subsection{Stitching as a discrete mixture approximation}
\label{sec:stitching_mixture}

Suppose we wish to analytically approximate the discrete mixture boundary
$\Mtil_\alpha$ of \cref{th:discrete_mixture} in the sub-Gaussian case
$\psi = \psi_N$. Clearly the sum is lower bounded by the maximum summand, which
gives
\begin{align}
\Mtil_\alpha(v)
  &\leq \sup\ebrace{
    s \in \R :
    \sup_{k \geq 0}
      \ebracket{w_k \exp\ebrace{\lambda_k s - \psi_N(\lambda_k) v}}
    < \frac{l_0}{\alpha}
  } \\
  &= \min_{k \geq 0} \ebrace{
    \frac{\log(l_0 / w_k \alpha)}{\lambda_k} + \frac{\lambda_k}{2}\, v
  }.
\end{align}
The last expression is the pointwise minimum of a collection of linear
boundaries of the form presented in \cref{th:uniform_chernoff}, each chosen with
a different $\lambda_k$, and with nominal crossing rates $w_k \alpha$ so that a
union bound over crossing events yields total crossing probability
$\sum_k w_k \alpha \leq \alpha$. This is very similar to the stitching
construction, with a slightly different choice of the sequence $\lambda_k$.

By equating $w_k$ from \cref{th:discrete_mixture} with $1/h(k)$ from
\cref{th:stitching}, this observation allows us to view a stitched bound with
function $h(k)$ as an approximation to a mixture bound with mixture density
$f(\lambda) = \Theta(1 / \lambda h(\log \lambda^{-1}))$ as
$\lambda \downarrow 0$. For exponential stitching, this yields
$f(\lambda) = \Theta(1)$---densities approaching a nonzero constant as
$\lambda \downarrow 0$, including the half-normal distribution, correspond to
exponential stitched boundaries growing at a rate $\sqrt{V_t \log V_t}$. For
polynomial stitching, we have the corresponding mixture density
\begin{align}
\flil_s(\lambda)
  \defineas \frac{(s - 1) s^{s-1} \indicator{0 \leq \lambda \leq \exp(-s)}}
                 {\lambda \log^s \lambda^{-1}},
  \label{eq:lil_mixture}
\end{align}
matching the density from \citet[Lemma 12]{balsubramani_sharp_2014} (we truncate
at $\lambda = e^{-s}$ to ensure the density is nonincreasing). The ``slower''
function $h(k) \propto k \log^s k$ corresponds to
$f(\lambda) = \Theta(1 / \lambda (\log \lambda^{-1}) (\log \log
\lambda^{-1})^s)$, the density from example 3 of
\citet{robbins_statistical_1970}.

\subsection{Proof of \cref{th:inverted_stitching}}
\label{sec:proof_inverted_stitching}

The proof follows a straightforward idea. We break time into epochs
$\eta^k \leq V_t < \eta^{k+1}$. Within each epoch we consider the linear
boundary passing through the points $(\eta^k, g(\eta^k))$ and
$(\eta^{k+1}, g(\eta^{k+1}))$. This line lies below $g(V_t)$ throughout the
epoch, and its crossing probability is determined by its slope and intercept as
in \cref{th:uniform_chernoff}. Taking a union bound over epochs yields the
result.

We need the following lemma concerning $g$:
\begin{lemma}\label{th:concave_lemma}
  If $g$ is nonnegative and strictly concave on $\R_{\geq 0}$, then $g(v)$ is
  nondecreasing and $g(v) / v$ is strictly decreasing on $v > 0$.
\end{lemma}
\begin{proof}
  If $s < 0$ is a supergradient of $g$ at some point $t$, then
  $g(t + u) < g(t) + s u < 0$ for sufficiently large $u$, contradicting the
  non-negativity of $g$. So $g$ is nondecreasing. Now fix $0 < x < y$ and let
  $s$ be any supergradient of $g$ at $x$. From nonnegativity and concavity we
  have $0 \leq g(0) \leq g(x) - x s$, so that $s \leq g(x) / x$. Strict
  concavity then implies $g(y) < g(x) + s (y - x) \leq g(x) y / x$.
\end{proof}

\begin{proof}[Proof of \cref{th:inverted_stitching}]
Fix any $\eta > 1$. On $\eta^k \leq v < \eta^{k+1}$ we lower bound $g(v)$ by the
line $a_k + b_k v$ passing through the points $(\eta^k, g(\eta^k))$ and
$(\eta^{k+1}, g(\eta^{k+1}))$. This line has intercept and slope
\begin{align}
a_k &= \frac{\eta g(\eta^k) - g(\eta^{k+1})}{\eta - 1}, \\
b_k &= \frac{g(\eta^{k+1}) - g(\eta^k)}{\eta^k (\eta - 1)}.
\end{align}
Note $a_k > 0$ and $b_k \geq 0$ by \cref{th:concave_lemma}. We bound the
upcrossing probability of this linear boundary using \cref{th:uniform_chernoff}:
\begin{align}
\P(\exists t \geq 1 : S_t \geq a_k + b_k V_t)
  \leq l_0 e^{-2 a_k b_k}
  = l_0 \expebrace{
      -\frac{2(g(\eta^{k+1}) - g(\eta^k))(\eta g(\eta^k) - g(\eta^{k+1}))}
            {\eta^k(\eta-1)^2}
    }.
\end{align}
The conclusion follows from a union bound over epochs and from the arbitrary
choice of $\eta$.
\end{proof}

Inspection of the proof reveals that the crossing probability bound
\eqref{eq:inverted_stitching_sum} is valid not only for the boundary $u$ given
in \eqref{eq:inverted_stitching_boundary}, but also for a similar boundary which
is finite and linear for all $v < 1$ and $v > v_{\max}$. This follows by
extending the linear boundaries over the first and last epochs.

\subsection{Proof of \cref{th:empirical_variance}}
\label{sec:proof_empirical_variance}

For the proof, we take $a = 0, b = 1$ without loss of generality. Write
$Y_t \defineas X_t - \E_{t-1} X_t$ and
$\delta_t \defineas \Xhat_t - \E_{t-1} X_t$. Then
$Y_t - \delta_t = X_t - \Xhat_t \in [-1, 1]$. We will show that
$\expebrace{\lambda \sum_{i=1}^t Y_i - \psi_E(\lambda) \sum_{i=1}^t (Y_i -
  \delta_i)^2}$ is a supermartingale for each $\lambda \in [0,1)$, where we take
$c=1$ in $\psi_E$.

The proof of Lemma 4.1 in \citet{fan_exponential_2015} shows that
$\expebrace{\lambda \xi - \psi_E(\lambda) \xi^2} \leq 1 + \lambda \xi$ for all
$\lambda \in [0, 1)$ and $\xi \geq -1$. Applied to $\xi = y - \delta$, we have
\begin{align}
  \expebrace{\lambda y - \psi_E(\lambda)(y - \delta)^2}
  \leq e^{\lambda \delta}(1 + \lambda (y - \delta)).
\end{align}
Since $Y_t - \delta_t \geq -1$, $\E_{t-1} Y_t = 0$, and $\delta_t$ is
predictable, the above inequality implies
\begin{align}
  \E_{t-1} \expebrace{\lambda Y_t - \psi_E(\lambda)(Y_t - \delta_t)^2}
  \leq e^{\lambda \delta_t}(1 - \lambda \delta_t)
  \leq 1,
\end{align}
using $1 - x \leq e^{-x}$ in the final step.

This shows that $S_t = \sum_{i=1}^t Y_i = \sum_{i=1}^t X_i - t \mu_t$ is
sub-exponential with variance process
$V_t = \sum_{i=1}^t (Y_i - \delta_i)^2 = \sum_{i=1}^t (X_i - \Xhat_i)^2$ and
scale $c=1$. It follows that $\P(\exists t: S_t \geq u(V_t)) \leq \alpha$. A
similar argument applied with $-X_t$ in place of $X_t$ shows that
$\P(\exists t: -S_t \geq u(V_t)) \leq \alpha$, and a union bound finishes the
proof. \qed

We remark that the proofs of \citet[Theorem 11]{maurer_empirical_2009},
\citet[Theorem 1]{audibert_explorationexploitation_2009}, and
\citet{balsubramani_sequential_2016} follow very different arguments. All three
proofs involve a Bennett-type concentration bound for the sample mean with a
radius depending on the true variance, combined via a union bound with a
concentration bound for the sample
variance. \Citet{audibert_explorationexploitation_2009} and
\citet{balsubramani_sequential_2016} achieve the latter bound using another
Bennett/Bernstein-type inequality and the inequality $\E X^4 \leq \E X^2$ for
$\abs{X} \leq 1$, while \citet{maurer_empirical_2009} use a self-bounding
property to achieve a concentration inequality for the sample variance directly
\citep[Theorem 7]{maurer_empirical_2009}.

In contrast, our argument avoids the union bound over the sample mean and sample
variance bounds. We achieve this by constructing an exponential supermartingale
which directly relates the deviations of $S_t$ to the ``online'' empirical
variance $V_t$. In terms of proof technique, our method owes much more to the
literature on self-normalized bounds (\citealp{de_la_pena_general_1999},
\citealp{de_la_pena_self-normalized_2004}, \citealp{bercu_exponential_2008},
\citealp{delyon_exponential_2009} and especially \citealp{fan_exponential_2015})
than to the literature on empirical-Bernstein bounds.

\subsection{Proof of \cref{th:matrix_asymptotic_lil}}
\label{sec:proof_matrix_asymptotic_lil}

For case (1), Lemma 3(f) and Lemma 2 of \citet{howard_exponential_2018}
(cf. \citealp{delyon_exponential_2009}) show that $S_t = \gamma_{\max}(Y_t)$ is
sub-Gaussian with variance process
$\widetilde{V}_t = \gamma_{\max}\eparen{\sum_{i=1}^t \frac{\Delta Y_i^2 + 2 \E
    \Delta Y_i^2}{3}}$. Invoking \cref{th:asymptotic_lil}, we have
\begin{align}
  \limsup_{t \to \infty}
  \frac{S_t}{\sqrt{2 \widetilde{V}_t \log\log \widetilde{V}_t}} \leq 1
  \quad \text{a.s. on } \ebrace{\sup_t \widetilde{V}_t = \infty}.
\end{align}
Applying the strong law of large numbers elementwise, we have
$t^{-1} \sum_{i=1}^t \frac{\Delta Y_i^2 + 2 \E \Delta Y_i^2}{3} \convas \E
Y_1^2$ as $t \to \infty$, and the continuity of the maximum eigenvalue map over
the set of positive semidefinite matrices ensures that
$t^{-1} \widetilde{V}_t \convas \gamma_{\max}(\E Y_1^2) = t^{-1} V_t$. Hence, so
long as $\E Y_1^2 > 0$ we conclude that, with probability one,
$\sup_t \widetilde{V}_t = \infty$ and
$\sqrt{\widetilde{V}_t \log \log \widetilde{V}_t} \sim \sqrt{\gamma_{\max}(\E
  Y_1^2) t \log \log t}$, completing the proof for case (1). (If $\E Y_1^2 = 0$
then the event $\brace{\sup_t V_t = \infty}$ is empty and the result is
vacuous.)

In case (2), Fact 1(d) and Lemma 2 of \citet{howard_exponential_2018}
(cf. \citealp{tropp_user-friendly_2012}) show that $(S_t)$ defined as above is
sub-gamma with variance process $(V_t)$ and scale $c$. The conclusion now
follows directly from \cref{th:asymptotic_lil}. \qed

\subsection{Proof of \cref{th:empirical_covariance}}

The argument is adapted from \citet{tropp_introduction_2015}. Let
$X_i \defineas x_i x_i^T - \Sigma$. The triangle inequality implies
$\opnorm{X_i} \leq \opnorm{x_i x_i^T} + \opnorm{\Sigma} \leq 2b$. Hence, by Fact
1(c) and Lemma 2 of \citet{howard_exponential_2018}
(cf. \citealp{tropp_user-friendly_2012}),
$S_t = \gamma_{\max}\eparen{\sum_{i=1}^t X_i}$ is sub-Poisson with scale
$c = 2b$ and variance process
\begin{align}
V_t &= \gamma_{\max}\eparen{\sum_{i=1}^t \E X_i^2} \\
  &= \gamma_{\max}
       \eparen{\sum_{i=1}^t \ebracket{\E[(x_i x_i^T)^2] - \Sigma^2}
     } \\
  &\leq \sum_{i=1}^t \gamma_{\max}\eparen{\E[(x_i x_i^T)^2]}.
\end{align}
In the final step, we neglect the negative semidefinite term $-\Sigma^2$ and use
the fact that the maximum eigenvalue of a sum of positive semidefinite matrices
is bounded by the sum of the maximum eigenvalues. We continue by using
$\norm{x_i x_i^T} = \norm{x_i}_2^2 \leq b$ and the fact the expectation respects
the semidefinite order to obtain
\begin{align}
V_t &\leq \sum_{i=1}^t \gamma_{\max}\eparen{\E \norm{x_i}_2^2 x_i x_i^T} \\
  &\leq t b \opnorm{\Sigma}.
\end{align}
Plugging this upper bound on $V_t$ into the discrete mixture bound of
\cref{th:discrete_mixture} gives the result. \qed

\section{Implications among sub-$\psi$ boundaries}

\begin{table}
\centering
\renewcommand{\arraystretch}{1.1}
\begin{tabular}{lllll}
\toprule
& $\psi_1$ & $\psi_2$ & $a$ & Restriction \\
\midrule
(1) & $\psi_N$ & $\psi_{B,g,h}$ & $\frac{\varphi(g, h)}{gh}$ & any $g,h > 0$ \\
(2) & $\psi_N$ & $\psi_{B,g,h}$ & $\frac{(g+h)^2}{4gh}$ & any $g,h > 0$ \\
(3) & $\psi_{P,c}$ & $\psi_{B,g,g+c}$ & 1 & any $g > -c$ \\
(5) & $\psi_{G,c}$ & $\psi_{P,3c}$ & 1 & \\
(6) & $\psi_{E,c}$ & $\psi_{G,2c/3}$ & 1 & \\
(7) & $\psi_{G,c}$ & $\psi_{E,c}$ & 1 & $c \geq 0$ \\
(8) & $\psi_{G,c}$ & $\psi_{E,2c}$ & 1 & $c < 0$ \\
(9) & $\psi_{P,c}$ & $\psi_{G,c/2}$ & 1 & $c < 0$ \\
(10) & $\psi_N$ & $\psi_{P,c}$ & 1 & any $c < 0$ \\
(11) & $\psi_{B,g,h}$ & $\psi_{P,-g}$ & 1 & \\
\bottomrule
\end{tabular}
\caption{For each row, if $u$ is a sub-$\psi_1$ uniform boundary, subject to the
  given restriction, then $v \mapsto u(av)$ is a sub-$\psi_2$ uniform boundary.
  $\varphi(g,h)$ is defined in \eqref{eq:varphi_defn}. See
  \cref{th:psi_boundary_relations} for details. \label{tab:psi_relations}}
\end{table}

Together with \cref{tab:psi_relations}, the following proposition formalizes the
relationships illustrated in \cref{fig:psi_boundary_relations}, restating
Proposition 2 of \citet{howard_exponential_2018} in the language of uniform
boundaries. The first row of \cref{tab:psi_relations} uses the function
\begin{align}
  \varphi(g, h) \defineas \begin{cases}
      \frac{h^2 - g^2}{2 \log(h/g)}, & g < h \\
      gh, & g \geq h.
    \end{cases} \label{eq:varphi_defn}
\end{align}

\begin{proposition}\label{th:psi_boundary_relations}
  For each row in \cref{tab:psi_relations}, if $u$ is a sub-$\psi_1$ uniform
  boundary, and the given restrictions are satisfied, then $v \mapsto u(av)$ is
  a sub-$\psi_2$ uniform boundary for the given constant $a$. Furthermore, when
  we allow only transformations of the form $v \mapsto u(av)$, these capture all
  possible implications among the five sub-$\psi$ boundary types defined above,
  and the given constants are the best possible (in the case of row (2), the
  constant $(g+h)^2 / 4gh$ is the best possible of the form $k/gh$ where $k$
  depends only on the total range $g + h$).
\end{proposition}

A reader who is familiar with \citet{howard_exponential_2018} will note that the
arrows in \cref{fig:psi_boundary_relations} are reversed with respect to Figure
4 in their paper. Indeed, since any sub-Bernoulli process is also sub-Gaussian,
it follows that any sub-Gaussian uniform boundary is also a sub-Bernoulli
uniform boundary, and so on.

\section{Additional proofs}\label{sec:other_proofs}

\subsection{Proof of \cref{th:normal_mixture_rho}}
\label{sec:proof_normal_mixture_rho}

Let $k \defineas (l_0 / \alpha)^2$. For part (a), we will set the derivative
of the squared objective $u^2(v) / v$ to zero:
\begin{align}
  \frac{\d}{\d v} \ebracket{
    \eparen{1 + \frac{\rho}{v}}\eparen{\log\pfrac{k(v+\rho)}{\rho}}}
  = -\frac{\rho}{v^2} \log\pfrac{k(v+\rho}{\rho} + \frac{1}{v} &= 0
    \label{eq:v_rho_opt}.\\
  -\pfrac{v+\rho}{\rho} \expebrace{-\frac{v+\rho}{\rho}} &= -\frac{1}{ek}.
\end{align}
We solve this equation using the lower branch $W_{-1}$ since we know
$-(v+\rho)/\rho \leq -1$:
\begin{align}
\frac{v+\rho}{\rho} &= -W_{-1}\eparen{-\frac{1}{ek}},
\end{align}
which is equivalent to \eqref{eq:m_and_rho}.

For part (b), we optimize the squared boundary $u^2(v)$:
\begin{align}
  \frac{\d}{\d \rho} \ebracket{(v + \rho) \log \pfrac{k(v + \rho)}{\rho}}
  = \log\pfrac{k(v+\rho)}{\rho} - \frac{v}{\rho} &= 0.
\end{align}
which is equivalent to \eqref{eq:v_rho_opt}. \qed

\subsection{Proof of \cref{th:unimprovable}}\label{sec:proof_unimprovable}

First, \citet[Theorem 1]{robbins_boundary_1970} show that, for $B(t)$ a standard
Brownian motion,
\begin{align}
\P(\exists t \in (0, \infty) : B(t) \geq \Mcal_\alpha(t)) = \alpha.
\label{eq:brownian_mixture_exact}
\end{align}
Let $(X_t)_{t=1}^\infty$ be any i.i.d.\ sequence of mean-zero random variables
with unit variance and $\E e^{\lambda X_1} \leq e^{\lambda^2 / 2}$, for example
standard normal or Rademacher random variables. For each $m \in \N$, let
$S_t^{(m)} \defineas \sum_{i=1}^t X_i / \sqrt{m}$ and
$V_t^{(m)} \defineas t / m$, noting that $(S_t^{(m)})$ is sub-Gaussian with
variance process $(V_t^{(m)})$. Our proof rests upon a standard application of
Donsker's theorem, detailed below, which shows that, for any $T \in \N$,
\begin{align}
\lim_{m \to \infty}
\P\eparen{\exists t \in [mT] : S^{(m)}_t \geq \Mcal_\alpha(V^{(m)}_t)}
  = \P(\exists t \in (0, T] : B(t) \geq \Mcal_\alpha(t)).
 \label{eq:unimprovable_step_1}
\end{align}
To obtain the desired conclusion from \eqref{eq:unimprovable_step_1}, we write,
for any $m \in \N$ and $T \in \N$,
\begin{align}
  \P\eparen{\exists t \in \N: S_t^{(m)} \geq \Mcal_\alpha(V_t^{(m)})}
    &\geq \P\eparen{\exists t \in [mT]: S_t^{(m)} \geq \Mcal_\alpha(V_t^{(m)})}.
\end{align}
Take $m \to \infty$ and use \eqref{eq:unimprovable_step_1} to find, for any
$T \in \N$,
\begin{align}
  \liminf_{m \to \infty}
  \P\eparen{\exists t \in \N: S_t^{(m)} \geq \Mcal_\alpha(V_t^{(m)})}
    &\geq \P\eparen{\exists t \in (0, T] : B(t) \geq \Mcal_\alpha(t)}.
\end{align}
Now take $T \to \infty$ to obtain
\begin{align}
  \liminf_{m \to \infty}
  \P\eparen{\exists t \in \N: S_t^{(m)} \geq \Mcal_\alpha(V_t^{(m)})}
    &\geq \P\eparen{\exists t \in (0, \infty) : B(t) \geq \Mcal_\alpha(t)}
    = \alpha, \label{eq:liminf_lower_bound}
\end{align}
by \eqref{eq:brownian_mixture_exact}. But for each $m \in \N$, $S_t^{(m)}$ is
sub-Gaussian with variance process $V_t^{(m)}$, so that
\begin{align}
  \P\eparen{\exists t \in \N: S_t^{(m)} \geq \Mcal_\alpha(V_t^{(m)})}
  \leq \alpha. \label{eq:unimprovable_upper}
\end{align}
Together, \eqref{eq:liminf_lower_bound} and \eqref{eq:unimprovable_upper} yield
\begin{align}
  \lim_{m \to \infty}
    \P\eparen{\exists t \in \N: S_t^{(m)} \geq \Mcal_\alpha(V_t^{(m)})}
  = \alpha.
\end{align}
Since $(S_t^{(m)}, V_t^{(m)}) \in \mathbb{S}_{\psi_N}^1$ for each $m$, the
conclusion follows.

To prove \eqref{eq:unimprovable_step_1}, we will use the fact that
$\Mcal_\alpha: \R_{\geq 0} \to \R_{\geq 0}$ is continuous, increasing and
concave, as proved in \cref{th:mixture_properties} below. For each
$t \in \R_{> 0}$ let $S(mt)$ be equal to $S_{mt}$ for $mt \in \N$ and a linear
interpolation otherwise (with $S(0) = 0$). Let $C[0,T]$ denote the space of
continuous, real-valued functions on $[0,T]$ equipped with the sup-norm, and let
$\P_0$ denote the probability measure for standard Brownian motion. We first use
a corollary of Donsker's theorem: for any $\varphi: C[0,T] \to \R$ continuous
$\P_0$-a.s., we have \citep[Theorems 8.1.5, 8.1.11]{durrett_probability:_2017}
\begin{align}
  \varphi\pfrac{S(m\cdot)}{\sqrt{m}}
    \convdist \varphi(B(\cdot)) \quad \text{as } m \to \infty.
\end{align}
We let $\varphi(f) \defineas \sup_{t \in [0, T]} [f(t) - \Mcal_\alpha(t)]$, so
that by compactness of $[0,T]$ and continuity of $f$ and $\Mcal_\alpha$,
$\varphi(f) \geq 0$ if and only if $f(t) \geq \Mcal_\alpha(t)$ for some
$t \in [0, T]$. Now $\varphi(S(m\cdot) / \sqrt{m}) \convdist \varphi(B(\cdot))$,
and note that $\varphi(B(\cdot))$ has a continuous distribution: the
distribution when $\Mcal_\alpha(t) \equiv 0$ is well-known by the reflection
principle, and the measure for the Brownian motion with drift
$B(t) - \Mcal_\alpha(t) + \Mcal_\alpha(0)$ is equivalent to the measure for
$B(t)$ by the Cameron-Martin theorem \citep[Theorem
1.38]{morters_brownian_2010}. Hence
\begin{align}
  \P\eparen{\exists t \in [0, T] : \frac{S(mt)}{\sqrt{m}} \geq \Mcal_\alpha(t)}
    \to \P\eparen{\exists t \in [0, T] : B(t) \geq \Mcal_\alpha(t)}.
  \label{eq:step_1_part_1}
\end{align}
But because $\Mcal_\alpha(t)$ is concave, the linear interpolation of $S(\cdot)$
cannot add any new upcrossings beyond those in $(S_t)$:
\begin{align}
  \P\eparen{\exists t \in [0, T] : \frac{S(mt)}{\sqrt{m}} \geq \Mcal_\alpha(t)}
    &= \P\eparen{\exists x \in [mT] :
         \frac{S_x}{\sqrt{m}} \geq \Mcal_\alpha(x/m)} \\
    &= \P\eparen{\exists t \in [mT] : S^{(m)}_t \geq \Mcal_\alpha(V^{(m)}_t)}.
      \label{eq:step_1_part_2}
\end{align}
Combining \eqref{eq:step_1_part_2} with \eqref{eq:step_1_part_1} yields
\eqref{eq:unimprovable_step_1}, completing the proof.
\qed

\begin{lemma}\label{th:mixture_properties}
  The function $\Mcal_\alpha: \R_{\geq 0} \to \R_{\geq 0}$ is continuous,
  increasing and concave.
\end{lemma}
\begin{proof}
  Continuity of $\Mcal_\alpha(v)$ is clear from the continuity of
  $\expebrace{\lambda s - \psi(\lambda) v}$ in $s$ and $v$, which also implies
  \begin{align}
  \int \expebrace{\lambda \Mcal_\alpha(v) - \psi(\lambda) v} \d F(\lambda)
    = \frac{l_0}{\alpha}
  \end{align}
  for all $v > 0$. That is, the left-hand side is constant in $v$, hence has
  derivative with respect to $v$ equal to zero. We may exchange the derivative
  and integral by Theorem A.5.1 of \citet{durrett_probability:_2017}, noting
  that the integrand is positive and continuously differentiable in $v$ and $F$
  is a probability measure. This yields
  \begin{align}
  \Mcal_\alpha'(v) &= \frac{A(v)}{B(v)} > 0,\\
  \text{where } A(v) &\defineas
    \int \psi(\lambda) e^{\lambda \Mcal_\alpha(v) - \psi(\lambda) v}
      \d F(\lambda) \\
  \text{and } B(v) &\defineas
    \int \lambda e^{\lambda \Mcal_\alpha(v) - \psi(\lambda) v} \d F(\lambda).
  \end{align}
  Both $A(v) > 0$ and $B(v) > 0$ since the integrands are positive, which shows
  that $\Mcal_\alpha$ is increasing. Differentiating again yields, after some
  algebra,
  \begin{align}
  B^2(v) \Mcal_\alpha''(v) &=
    \int \eparen{-\frac{[\lambda A(v) - \psi(\lambda) B(v)]^2}{B(v)}}
        e^{\lambda \Mcal_\alpha(v) - \psi(\lambda) v} \d F(\lambda) \leq 0,
  \end{align}
  since the integrand is now nonpositive, showing that $\Mcal_\alpha$ is
  concave.
\end{proof}

\subsection{Proof of \cref{th:expo_family}}\label{sec:proof_expo_family}

Write $\mu^\star \defineas \E T(X_1)$.We have noted in the discussion preceding
the result that the exponential process
$\expebrace{\lambda S_t(\mu) - t \psi_\mu(\lambda)}$ is the likelihood ratio
testing $H_0: \theta = \theta(\mu)$ against
$H_1: \theta = \theta(\mu) + \lambda$. It is well-known that the likelihood
ratio is a martingale under the null. Hence $(S_t(\mu^\star))$ is
sub-$\psi_{\mu^\star}$ with variance process $V_t = t$, and it follows
immediately that
$\P(\exists t: S_t(\mu^\star) \geq u_{\mu^\star}(t)) \leq \alpha_1$. Apply the
same argument with $-X_t$ in place of $X_t$ to conclude that
$\P(\exists t: -S_t(\mu^\star) \geq \tilde{u}_{\mu^\star}(t)) \leq \alpha_2$. A
union bound completes the argument. \qed

\subsection{Proof of \cref{th:equiv_uniform_defns}}
\label{sec:proof_equiv_uniform_defns}

The implication $(a) \implies (b)$ follows from
\begin{align}
  A_T
  &= \ebracket{\eunion_{t=1}^\infty A_t \intersect \brace{T = t}}
    \union \ebracket{A_\infty \intersect \brace{T = \infty}}
    \subseteq \eunion_{t=1}^\infty A_t.
\end{align}
It is clear that $(b) \implies (c)$. For $(c) \implies (a)$, take
$\tau = \inf\brace{t \in \N: A_t \text{ occurs}}$, so that
$A_\tau = \eunion_{t = 1}^\infty A_t$. \qed

\section{Computing conjugate mixture bounds by root-finding}
\label{sec:root_finding}

In this section we demonstrate that our conjugate mixture boundaries, which
involve the supremum $\Mcal_\alpha(v)$ defined in \eqref{eq:mixture_operator},
can be computed via root-finding. We assume that $\psi$ is CGF-like, a property
which holds for all of the $\psi$ functions in \cref{sec:prelims}:
\begin{definition}[\citealp{howard_exponential_2018}, Definition 2]
  \label{th:cgf_like}
  A real-valued function $\psi$ with domain $[0, \lambda_{\max})$ is called
  \emph{CGF-like} if it is strictly convex and twice continuously differentiable
  with $\psi(0) = \psi'(0_+) = 0$ and
  $\sup_{\lambda \in [0, \lambda_{\max})} \psi(\lambda) = \infty$. For such a
  function, we write
  \begin{align}
    \bar{b} \defineas
      \sup_{\lambda \in [0, \lambda_{\max})} \psi'(\lambda) \in (0, \infty].
  \end{align}
\end{definition}

\Cref{th:basic_mixture} implies that, with probability at least $1 - \alpha$,
$m(S_t, V_t) < l_0 / \alpha$ for all $t$, where
\begin{align}
m(s, v) &= \int \expebrace{\lambda s - \psi(\lambda) v} \d F(\lambda).
\end{align}
We are interested in the set
$A(v) \defineas \brace{s \in \R: m(s, v) < l_0 / \alpha}$ for fixed
$v \geq 0$. It is clear that $m(0, v) \leq 1 < l_0 / \alpha$ whenever
$l_0 \geq 1$ (which holds in all cases we consider), since $\psi \geq 0$,
$v \geq 0$ and $F$ is a probability distribution. So $0 \in A(v)$ always. We
show below that, in addition, $A(v)$ is always an interval.

For one-sided boundaries, $F$ is supported on $\lambda \geq 0$, and so long as
$F$ is not a point mass at zero (which would be an uninteresting mixture),
$m(s, v)$ is strictly increasing in $s$ whenever $m(s, v) < \infty$. Hence
$m(s, v) = l_0 / \alpha$ for at most one value of $s^\star(v) > 0$, in which
case $A(v) = (-\infty, s^\star(v))$.

It is possible that $m(s, v) < l_0 / \alpha$ for all $s$ where the integral
converges. To examine this case, we fix $v > 0$, which is the interesting case
in practice, and make two observations:
\begin{itemize}
\item Whenever $s < \bar{b} v$, we have $m(s, v) < \infty$. Indeed, in this
  case, $\expebrace{\lambda s - \psi(\lambda) v} \to 0$ as $\lambda \to \infty$,
  and as the integrand is continuous in $\lambda$, it must be uniformly
  bounded. It follows immediately that we can have $m(s, v) = \infty$ only when
  $\bar{b} < \infty$.
\item Whenever $\bar{b} < \infty$, we have $S_t \leq \bar{b} V_t$ a.s., a
  consequence of Theorem 1(a) of \citet{howard_exponential_2018}, which shows
  that $\P(\exists t: S_t \geq a + \bar{b} V_t) = 0$ for all $a > 0$. (To verify
  this fact, note we must have $\lambda_{\max} = \infty$ when $\bar{b} < \infty$
  in order for the CGF-like condition
  $\sup_{\lambda \in [0, \lambda_{\max})} \psi(\lambda) = \infty$ to hold.)
\end{itemize}
Hence, when $\bar{b} = \infty$ we need not worry about $m(s, v) = \infty$. When
$\bar{b} < \infty$, it suffices to check $m(\bar{b} v, v)$, which may be
infinite. If $m(\bar{b} v, v) \geq l_0 / \alpha$, then we search for a root
of $m(s, v) = l_0 / \alpha$ in the interval $s \in [0, \bar{b} v]$. If
$m(\bar{b} v, v) < l_0 / \alpha$, it suffices to take
$\Mcal_\alpha(v) = \bar{b} v + \epsilon$ for any $\epsilon > 0$. In practice, it
seems more reasonable to take the upper bound $\bar{b} v$ and use a closed
confidence set instead of an open one.

For two-sided boundaries, when $F$ has support on both $\lambda > 0$ and
$\lambda < 0$, in general we require the technical condition
\begin{align}
  \int \abs{\lambda}^k \expebrace{\lambda s - \psi(\lambda) v} \d F(\lambda)
  < \infty, \quad \text{for $k = 1, 2$.} \label{eq:deriv_condition}
\end{align}
This ensures that we may differentiate $m(s, v)$ twice with respect to $s$,
exchanging the derivative and the integral both times \citep[Theorem
A.5.3]{durrett_probability:_2017}. Hence, whenever condition
\eqref{eq:deriv_condition} holds,
\begin{align}
\frac{\d^2}{\d s^2}\, m(s, v)
  = \int \lambda^2 \expebrace{\lambda s - \psi(\lambda) v} \d F(\lambda) \geq 0,
\end{align}
so that $m(s, v)$ is convex in $s$ for each $v \geq 0$. As
$m(0, v) < l_0 / \alpha$, we conclude that $m(s, v) = l_0 / \alpha$ for at
most one value $s^\star(v) > 0$ and one value $s_\star(v) < 0$, and
$A(v) = (s_\star(v), s^\star(v))$. A similar discussion as above applies when
$\bar{b} < \infty$ and we may have $m(s, v) = \infty$ for some values of $s$.

As \cref{th:two_sided_normal_mixture} yields a closed-form result, only
\cref{th:two_sided_beta_mixture} requires that we verify condition
\eqref{eq:deriv_condition}. From the proof of \cref{th:two_sided_beta_mixture}
in \cref{sec:proof_mixtures}, it suffices to show that
\begin{align}
  \int_0^1 \eabs{\log\pfrac{p}{1-p}}^k p^a (1 - p)^b \d p < \infty
\end{align}
for some $a, b > 0$ and $k = 1, 2$. This follows from the fact that the
integrand is continuous on $p \in (0,1)$ and approaches zero as $p \to 0$ and
$p \to 1$, so it is bounded.

\section{Tuning discrete mixture implementation}
\label{sec:discrete_mix_details}

In \cref{sec:tuning} we have discussed the choice of mixing precision in order
to tune a mixture bound for a particular range of sample sizes. For discrete
mixtures, the value $\overlambda$ must also be chosen, and this depends on the
minimum relevant value of $V_t$: making $\overlambda$ larger will make the
resulting bound tighter over smaller values of $V_t$ at the cost of a looser
bound for larger values of $V_t$. In practice, for $\psi = \psi_G$, setting
$\overlambda = [c + \sqrt{m / 2 \log \alpha^{-1}}]^{-1}$ will ensure the bound
is tight for $V_t \geq m$. Furthermore, when evaluating $\Mtil_\alpha(v)$ in
practice, the sum can be truncated after
$k_{\max} = \ceil{\log_\eta(\overlambda [c + \sqrt{5 v / \log \alpha^{-1}}])}$
terms. The remainder of this section explains these choices.

We wish to understand what range of values of $\lambda$ our discrete mixture
must cover to ensure we get a tight bound for all
$V_t \in [m, v_{\max}]$. At $V_t = m$ the value of $\lambda$ which yields
the optimal linear bound from \cref{th:uniform_chernoff} is found by optimizing
\begin{align}
\frac{\log \alpha^{-1}}{\lambda} + \frac{\psi(\lambda)}{\lambda} \cdot m,
\end{align}
yielding the first-order condition
\begin{align}
\lambda \psi'(\lambda) - \psi(\lambda) = \frac{\log \alpha^{-1}}{m}.
\end{align}
For $\psi = \psi_G$, this becomes
\begin{align}
\frac{\lambda^2}{2(1-c\lambda)^2} = \frac{\log \alpha^{-1}}{m},
\end{align}
which is solved by
\begin{align}
\lambda^\star(m) = \frac{1}{c + \sqrt{m / 2 \log \alpha^{-1}}}.
\end{align}
Large values of $\lambda$ are necessary to achieve tight bounds for small
$V_t$. Hence, to ensure good performance at $V_t = m$ we choose
$\overlambda = [c + \sqrt{m / 2 \log \alpha^{-1}}]^{-1}$. Similarly, to ensure
the sum safely covers $V_t = v$ we ensure
$\lambda_{k_{\max}} \leq [c + \sqrt{10 v / 2 \log \alpha^{-1}}]^{-1}$ (using an
arbitrary ``fudge factor'' of ten), which yields
$k_{\max} = \ceil{\log_\eta(\lambda_{\max} [c + \sqrt{5 v / \log
    \alpha^{-1}}])}$.

We note that $\eta$ must also be chosen, but the only tradeoff here is
computational. Smaller values of $\eta$ lead to more accurate approximations of
the discrete mixture to the target continuous mixture, but require more terms to
be summed. We have found $\eta = 1.1$ to provide excellent approximations in the
examples we have examined.

\section{Intrinsic time, change of units and minimum time conditions}
\label{sec:change_units}

In this section we point out that a bound expressed in terms of intrinsic time
yields an infinite family of related bounds via scaling, and that ``minimum
time'' conditions in such bounds (such as $m \bmax V_t$ in \cref{th:stitching})
can be freely scaled as well. Suppose we have a uniform bound of the form
\begin{align}
  \P\eparen{\exists t \geq 1 : S_t \geq u_c(m \bmax V_t)}
    \leq \alpha,
  \label{eq:change_units_bound}
\end{align}
where intrinsic time $V_t$ has the same units as $S_t^2$, as usual, and $c$ is
some parameter with the same units as $S_t$. Then, fixing any $\gamma > 0$ and
applying the bound \eqref{eq:change_units_bound} to the scaled observations
$X_t / \sqrt{\gamma}$, which amounts to a change of units, we have
\begin{align}
\alpha
  &\geq \P\eparen{
    \exists t \geq  1 :
    \frac{S_t}{\sqrt{\gamma}} \geq
    u_{c / \sqrt{\gamma}}\eparen{m \bmax \frac{V_t}{\gamma}}
  } \\
  &= \P\eparen{\exists t \geq 1 : S_t \geq h_c(\gamma m \bmax V_t)},
    \quad \text{where } h_c(v) \defineas
      \sqrt{\gamma} u_{c / \sqrt{\gamma}}\pfrac{v}{\gamma}.
\end{align}
By changing units we have obtained a new bound on $S_t$ with different minimum
time $\gamma m$ and a different shape. For example, applying this change of
units to the stitched boundary \eqref{eq:stitching_operator} with $m = 1$ yields
the family of bounds
\begin{align}
\P\eparen{
  \exists t \geq 1 :
  S_t \geq k_1 \sqrt{(\gamma \bmax V_t)\ell\pfrac{\gamma \bmax V_t}{\gamma}}
  + c k_2 \ell\pfrac{\gamma \bmax V_t}{\gamma}
} \leq \alpha
\end{align}
for any $\gamma > 0$, with the definition of $\ell$ unchanged from
\eqref{eq:stitching_operator}. Note only the argument of $\ell$ has been
scaled. We started with a single bound \eqref{eq:stitching_operator} expressed
in terms of $V_t$ and ended up with a family of bounds on the same process
$S_t$, one for each value of $\gamma$. Indeed, the tuning parameter $m$ in
\cref{th:stitching} is obtained by exactly this argument. The effect is more
clear if we let $c = 0$ and examine the upper bound on the normalized process
$S_t / \sqrt{V_t}$: then for any $\gamma > 0$, with probability at least
$1 - \alpha$,
\begin{align}
\frac{S_t}{\sqrt{V_t}} \leq \begin{cases}
  k_1 \sqrt{\ell\pfrac{V_t}{\gamma}}, & \text{when } V_t \geq \gamma, \\
  k_1 \sqrt{\frac{\gamma \ell(1)}{V_t}}, & \text{when } V_t < \gamma.
\end{cases}
\end{align}
Now the right-hand depends on $V_t$ only through $V_t / \gamma$, so that the
effect of changing $\gamma$ is simply to multiplicatively shift the bound
backwards or forwards in time without changing the bounded process.

\section{Detailed comparison of finite LIL bounds}
\label{sec:finite_lil_details}

\begin{figure}
\centering
\includegraphics{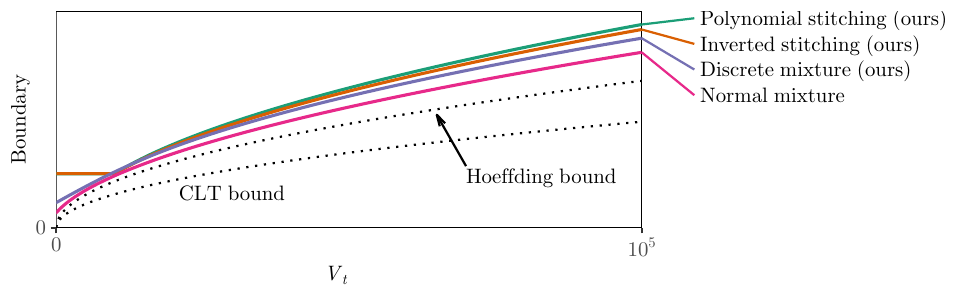}
\caption{Pointwise and uniform bounds for independent 1-sub-Gaussian
  observations, $\alpha = 0.025$. All tuning parameters are chosen to optimize
  roughly for time $V_t = 10^4$. The dotted lines show the Hoeffding bound
  $\sqrt{2 V_t \log \alpha^{-1}}$, which is nonasymptotically pointwise valid,
  and the CLT bound $z_{1-\alpha} \sqrt{V_t}$, which is asymptotically pointwise
  valid. Polynomial stitching uses \cref{th:stitching} with $\eta=2.04$,
  $m = 10^4$, and $h(k) = (k+1)^{1.4} \zeta(1.4)$. The inverted stitching
  boundary is
  $1.7 \sqrt{(V_t \bmax 10^4) (\log(1 + \log((V_t \bmax 10^4) / 10^4) + 3.5)}$,
  using \cref{th:inverted_stitching} with $\eta = 2.99$, $v_{\max} = 10^{20}$,
  and error rate $0.82\alpha$ to account for finite horizon and ensure a fair
  comparison. Discrete mixture applies \cref{th:discrete_mixture} to the density
  $f(\lambda) = 0.4 \cdot \indicator{0 \leq \lambda \leq \lambda_{\max}} /
  [\lambda \log^{1.4}(\lambda_{\max} e / \lambda)]$ with $\eta = 1.1$ and
  $\lambda_{\max} = 0.044$; see \cref{sec:stitching_mixture} for motivation. The
  normal mixture bound \eqref{eq:normal_mixture_closed} uses $\rho =
  1260$. \label{fig:finite_lil_boundaries}}
\end{figure}

\begin{figure}
\centering
\includegraphics{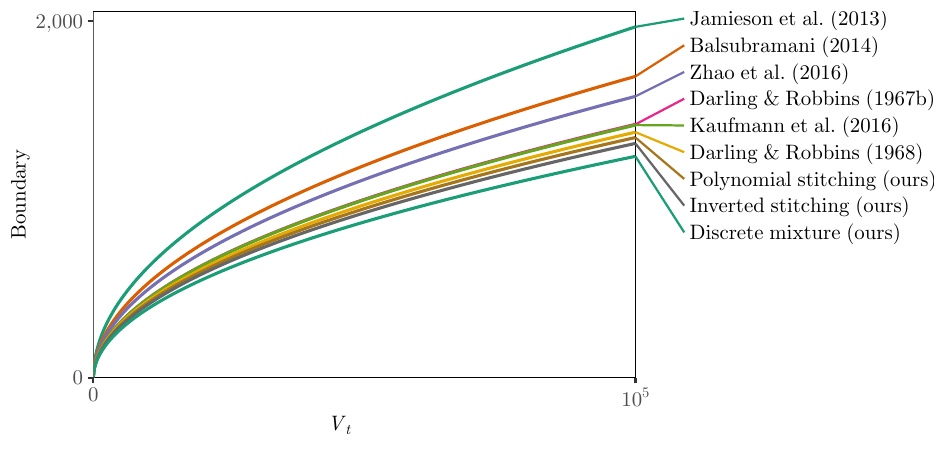}
\caption{Finite LIL bounds for independent 1-sub-Gaussian observations,
  $\alpha = 0.025$. The dotted lines show the Hoeffding bound
  $\sqrt{2 V_t \log \alpha^{-1}}$, which is nonasymptotically pointwise valid,
  and the CLT bound $z_{1-\alpha} \sqrt{V_t}$, which is asymptotically pointwise
  valid. Polynomial stitching uses \cref{th:stitching} with $\eta=2.04$ and
  $h(k) = (k+1)^{1.4} \zeta(1.4)$. The inverted stitching boundary is
  $1.7 \sqrt{V_t (\log(1 + \log V_t) + 3.5)}$, using
  \cref{th:inverted_stitching} with $\eta = 2.99$, $v_{\max} = 10^{20}$, and
  error rate $0.82\alpha$ to account for finite horizon. Discrete mixture
  applies \cref{th:discrete_mixture} to the density
  $f(\lambda) = 0.4 \cdot \indicator{0 \leq \lambda \leq 4} / [\lambda
  \log^{1.4}(4e / \lambda)]$ with $\eta = 1.1$, and $\lambda_{\max} = 4$; see
  \cref{sec:stitching_mixture} for motivation. The normal mixture bound
  \eqref{eq:normal_mixture_closed} uses $\rho = 0.129$. See
  \cref{sec:finite_lil_details} for details. \label{fig:finite_lil_all}}
\end{figure}

\Cref{fig:finite_lil_boundaries,fig:finite_lil_all} compare our finite LIL
bounds to several existing bounds. Below we restate the original results from
the various papers giving finite LIL bounds included in
\cref{fig:finite_lil_all}. In table~\ref{tab:finite_lil_bounds}, for ease of
  comparison, we write all bounds in the form
\begin{align}
\P(\exists t \geq 1 : S_t \geq A \sqrt{t (\log \log B t + C)},
\end{align}
valid for independent 1-sub-Gaussian observations. When the original bound holds
only for $t \geq n$ instead of $t \geq 1$, we apply a change of units argument
to replace $\log \log B t$ with $\log \log B n t$ and $t \geq n$ with
$t \geq 1$, so that all bounds are comparable (see
\cref{sec:change_units}). When bounds are expressed in terms of intrinsic time
$V_t$ \citep{balsubramani_sharp_2014}, this is formally justified. When they are
expressed in terms of nominal time
\citep{darling_iterated_1967,darling_further_1968} this is only a heuristic
argument, but we conjecture that proofs of such bounds could be generalized to
justify this scaling. When observations are i.i.d.\ from an infinitely divisible
distribution, the change is formally justified by replacing each observation
$X_i$ with a sum of $n$ i.i.d.\ ``pseudo-observations'' $Z_i$ such that
$\sum_{i=1}^n Z_i \sim X_1$.

\begin{itemize}
\item \citet{jamieson_best-arm_2014}, Lemma 1: for i.i.d.\ sub-Gaussian
  observations with variance parameter $\sigma^2$,
  \begin{align}
    \P\eparen{
      \exists t \geq 1 :
      S_t \geq (1 + \sqrt{\epsilon}) \sqrt{
        2 \sigma^2 (1 + \epsilon) t \log\pfrac{\log((1 + \epsilon) t)}{\delta}
      }
    } \leq 1 - \frac{2 + \epsilon}{\epsilon}
        \pfrac{\delta}{\log(1 + \epsilon)}^{1 + \epsilon}.
  \end{align}
\item \citet{zhao_adaptive_2016}, Theorem 1: for sub-Gaussian observations with
  variance parameter $1/4$,
  \begin{align}
    \P\eparen{\exists t \geq 1 : S_t \geq \sqrt{a t \log(\log_c t + 1) + b t}}
      \leq \zeta(2a / c) e^{-2b / c}.
  \end{align}
\item \citet{kaufmann_complexity_2014}, Lemma 7: for independent sub-Gaussian
  observations with variance parameter $\sigma^2$,
  \begin{align}
    \P\eparen{
      \exists t \geq 1 : S_t \geq \sqrt{2 \sigma^2 t (x + \eta \log \log(e t))}
    } \leq \sqrt{e} \zeta\eparen{\eta \eparen{1 - \frac{1}{2x}}}
           \eparen{\frac{\sqrt{x}}{2\sqrt{2}} + 1}^\eta e^{-x}
    \label{eq:kaufmann_error}
  \end{align}
\item \citet{balsubramani_sharp_2014}, Theorem 4: for $\abs{X_t} \leq c_t$
  a.s. and $V_t = \sum_{i=1}^t c_i^2$,
  \begin{align}
    \P\eparen{
      \exists t \geq 1 : V_t \geq 173 \log\pfrac{2}{\alpha} :
      S_t \geq \sqrt{3 V_t (2 \log \log (3V_t / 2 S_t) + \log \alpha^{-1})}
    } \leq \alpha.
  \end{align}
  Though the bound is stated for bounded observations, the proof holds for any
  observations sub-Gaussian with variance parameters $(c_t^2)$, as noted in
  section 5.2 of \citet{balsubramani_sharp_2014}. Balsubramani suggests removing
  the initial time condition by imposing a constant bound over
  $t \leq 173 \log(2/\alpha)$ (section 5.3). We instead remove the condition by
  a change of units, as discussed in \cref{sec:change_units}.
\item \citet{darling_iterated_1967}, eq. 22: for i.i.d.\ observations
  sub-Gaussian with variance parameter 1,
  \begin{align}
    \P\eparen{
      \exists t \geq \eta^j :
      S_t \geq \frac{1 + \eta}{2 \sqrt{\eta}}
      \sqrt{t(2 c \log \log t - 2 c \log \log \eta + 2 \log a)}
    } \leq \frac{1}{a (c - 1)(j - 1/2)^{c - 1}}.
  \end{align}
  Darling and Robbins consider results for a general bound $\varphi(\lambda)$ on
  the moment-generating function of the observations. The result involves the
  term $h(v_t)$ where the function
  $h(\lambda) \defineas 1/2 + \lambda^{-2} \log \varphi(\lambda)$ and $v_t$ is
  unspecified but bounded.
\item \citet{darling_further_1968}, eq. 2.2 and the example that follows: for
  i.i.d.\ observations sub-Gaussian with variance parameter 1,
  \begin{align}
    \P\eparen{\exists t \geq 3 : S_t \geq A \sqrt{t(\log \log t + C)}} \leq
    \int_m^\infty \frac{A \sqrt{\log\log t + C}}{t}
      \expebrace{-\frac{A^2(\log\log t + C)}{2}} \d t.
    \label{eq:dr_1968_error}
  \end{align}
  Darling and Robbins give a closed-form upper bound for the right-hand side of
  \eqref{eq:dr_1968_error}. We instead evaluate it numerically, using
  readily-available implementations of the upper incomplete gamma function:
  \begin{multline}
    \int_m^\infty \frac{A \sqrt{\log\log t + C}}{t}
    \expebrace{-\frac{A^2(\log\log t + C)}{2}} \d t \\
    = \frac{\sqrt{2\pi} A e^{-C}}{(A-2)^{3/2}}\,
      \P\eparen{G \geq \frac{A^2 - 2}{2} (\log \log m + C)},
  \end{multline}
  where $G \sim \Gamma(3/2, 1)$.
\item Polynomial stitching as in \eqref{eq:poly_stitching} with $c = 0$.
\item Inverted stitching with $g(v) = A \sqrt{v (\log \log(e v) + C)}$ as in
  \eqref{eq:lil_inverted}. We set $v_{\max} = 10^{20}$ which covers 42 epochs
  with $\eta = 2.994$. To make for a fair comparison with polynomial stitching,
  observe that in 42 epochs with $s = 1.4$, polynomial stitching ``spends''
  $\sum_{k=1}^{42} k^{-1.4} / \zeta(1.4) \approx 0.820$ of its crossing
  probability $\alpha$, so we run inverted stitching with
  $\alpha = 0.820 \cdot 0.025$.
\item Normal mixture as in \eqref{eq:normal_mixture_closed} with
  $\rho \approx 0.13$:
  \begin{align}
    u(v) \approx
      \sqrt{2 (v + 0.13) \log\eparen{20\sqrt{1+\frac{v}{0.13}} + 1}}.
  \end{align}
  This is not a LIL boundary, so is not included in
  \cref{tab:finite_lil_bounds}.
\end{itemize}

\begin{table}[h]
\centering
\begin{tabular}{rccc}
\toprule
Source and parameter settings & $A$ & $B$ & $C$ \\
\midrule
\citet{jamieson_best-arm_2014} &
  $(1 + \sqrt{\epsilon})\sqrt{2(1 + \epsilon)}$ &
  $1 + \epsilon$ &
  $\frac{1}{1 + \epsilon}
   \log\pfrac{2 + \epsilon}{\alpha\epsilon\log^{1 + \epsilon}(1+\epsilon)}$\\
$\epsilon = 0.033$ & (1.7) & (1.033) & (10.966) \\
\midrule
\citet{balsubramani_sharp_2014} &
  $\sqrt{6}$ &
  $\frac{865}{2} \log\pfrac{2}{\delta}$ &
  $(\log \alpha^{-1}) / 2$ \\
& (2.45) & (1137) & (1.844) \\
\midrule
\citet{zhao_adaptive_2016} &
  $2\sqrt{a}$ &
  $c$ &
  $\frac{c}{2a} \log\pfrac{\zeta(2a/c)}{\alpha \log^{2a/c} c}$ \\
$a=0.7225, c=1.1$ & (1.7) & (1.1) & (6.173) \\
\midrule
\citet{darling_iterated_1967} &
  $(1 + \eta) \sqrt{\frac{c}{2\eta}}$ &
  $\eta^j$ &
  $\frac{1}{c} \log\pfrac{1}{\alpha (c - 1)(j - 1/2)^{c-1} \log^c \eta}$ \\
$j=1, c=1.4, \eta=1.429$ & (1.7) & (1.429) & (4.518) \\
\midrule
\citet{kaufmann_complexity_2014} &
  $\sqrt{2 \eta}$ &
  $e$ &
  $x(\alpha, \eta) / \eta$ \\
$\eta = 1.3$ & (1.7) & (2.718) & (4.427) \\
\midrule
\citet{darling_further_1968} & $A$ & 3 & $C(\alpha, A)$ \\
$A = 1.7$ & $(1.7)$ & $(3)$ & $(3.945)$ \\
\midrule
Polynomial stitching \eqref{eq:poly_stitching} &
  $(\eta^{1/4} + \eta^{-1/4})\sqrt{\frac{s}{2}}$ &
  $\eta$ &
  $\frac{1}{s} \log \frac{\zeta(s)}{\alpha \log^s \eta}$ \\
$s=1.4, \eta = 2.041$ & (1.7) & (2.041) & (3.782) \\
\midrule
Inverted stitching (\cref{th:inverted_stitching}) &
  $A$ & $e$ & $C(\alpha, A, \eta)$ \\
$\eta = 2.994$, nominal error rate $0.82 \alpha$ & (1.7) & (2.718) & (3.454) \\
\bottomrule
\end{tabular}
\caption{Comparison of parameters $A,B,C$ for finite LIL boundaries expressed in
  the form
  ${\P(\exists t \geq 1 : S_t \geq A \sqrt{t (\log \log B t + C)}) \leq \alpha}$
  for sums of independent 1-sub-Gaussian observations, with $\alpha = 0.025$. Functions $x(\alpha, \eta)$ and $C(\alpha, \dots)$ are given
  by numerical root-finding to set the corresponding error bound equal to
  $\alpha$. \label{tab:finite_lil_bounds}}
\end{table}

\section{Details of Example 1}\label{sec:example_details}

Write $X_t = \mu + \sigma Z_t$ for $t = 1, 2, \dots$ where $Z_1, Z_2, \dots$ are
i.i.d.\ $\Normal(0,1)$ random variables. Substituting into the definition of
$S_t$, we find
\begin{align}
  S_t = \sum_{i=1}^{t+1} \eparen{Z_i - \Zbar_{t+1}}^2 - t,
\end{align}
where $\Zbar_t \defineas t^{-1} \sum_{i=1}^t Z_i$. Evidently the distribution of
$S_t$ depends on neither $\mu$ nor $\sigma^2$. Furthermore, direct calculation
shows that the increments of $(S_t)$ may be written as
\begin{align}
  \Delta S_t = S_t - S_{t-1} &= \frac{t}{t+1} \eparen{Z_{t+1} - \Zbar_t}^2 - 1\\
    &\defineright Y_t^2 - 1,
\end{align}
where we define $Y_t \defineas \sqrt{t/(t+1)} (Z_{t+1} - \Zbar_t)$ for
$t = 1, 2, \dots$ (and take $S_0 = 0$ by convention). Noting that
$Z_{t+1} \sim \Normal(0,1)$ is independent of $\Zbar_t \sim \Normal(0, t^{-1})$,
we see that $Y_t \sim \Normal(0, 1)$ for each $t$. Finally, a straightforward
calculation shows that $\E Y_i Y_j = 0$ for all $i \neq j$, so that
$Y_1, Y_2, \dots$ are i.i.d. It follows that $\Delta S_1, \Delta S_2, \dots$ are
i.i.d.\ centered Chi-squared random variables each with one degree of
freedom. The CGF of this distribution is
\begin{align}
  \log \E e^{\lambda \Delta S_1} = -\frac{\log(1 - 2\lambda)}{2} - \lambda,
  \quad \text{for all } \lambda < \frac{1}{2}. \label{eq:chisq_cgf}
\end{align}
which is equal to $2 \psi_E(\lambda)$ with scale $c = 2$. As the increments of
$(S_t)$ are i.i.d., it suffices for Definition 1 to have
$\log \E e^{\lambda \Delta S_t} \leq \psi(\lambda) \Delta V_t$, and we have
shown this holds with equality.

We have shown that $(S_t)$ is sub-exponential with scale $c = 2$ and variance
process $V_t = 2t$. Recall that \cref{th:canonical_assumption} depends only on
$\lambda \geq 0$. However, since \eqref{eq:chisq_cgf} holds for all
$\lambda < 1/2$ and not just $0 \leq \lambda < 1/2$, replacing $\Delta S_t$ with
$-\Delta S_t$ shows that $(-S_t)$ is sub-exponential with scale $c = -2$.

\section{Extension to smooth Banach spaces and continuous-time processes}
\label{sec:banach}

Though we have focused on discrete-time processes taking values in $\R$ or
$\Scal^d$, our uniform boundaries also apply to discrete-time martingales in
general smooth Banach spaces and to real-valued, continuous-time martingales. In
this section we briefly review concepts from \citet[Sections
3.4-3.5]{howard_exponential_2018} to highlight the possibilities.
First, let $(Y_t)_{t \in \N}$ be a martingale taking values in a separate Banach
space $(\Xcal, \norm{\cdot})$. Our uniform boundaries apply to any function
$\Psi: \Xcal \to \R$ satisfying the following property:

\begin{definition}[\citep{pinelis_optimum_1994}]
  A function $\Psi : \Xcal \to \R$ is called \emph{$(2,D)$-smooth} for some
  $D > 0$ if, for all $x, v \in \Xcal$, we have (a) $\Psi(0) = 0$, (b)
  $\abs{\Psi(x + v) - \Psi(x)} \leq \norm{v}$, and (c)
  $\Psi^2(x + v) - 2 \Psi^2(x) + \Psi^2(x - v) \leq 2D^2 \norm{v}^2$.
\end{definition}

For example, the norm induced by the inner product in any Hilbert space is
$(2,1)$-smooth, and the Schatten $p$-norm is $(2,\sqrt{p-1})$-smooth for
$p \geq 2$.

\begin{corollary}\label{th:banach}
  Let $(Y_t)_{t \in \N}$ be a martingale taking values in a separable Banach
  space $(\Xcal,\norm{\cdot})$, and $\Psi: \Xcal \to \R$ is $(2,D)$-smooth; denote
  $D_\star \defineas 1 \bmax D$.
  \begin{enumerate}[label=(\alph*)]
  \item Suppose $\smash{\norm{\Delta Y_t} \leq c_t}$ a.s. for all $t$ for
    constants $(c_t)$. Then, for any sub-Gaussian boundary $f$ with crossing
    probability $\alpha$ and $\smash{l_0 = 2}$, we~have
    \begin{align}
    \small
      \P\eparen{
        \exists t \geq 1: \Psi(Y_t) \geq f\eparen{D_\star^2 \sum_{i=1}^t c_i^2}
      } \leq \alpha.
    \end{align}
  \item Suppose $\norm{\Delta Y_t} \leq c$ a.s. for all $t$ for
     $\smash{c > 0}$. Then, for any sub-Poisson boundary $f$ with crossing
    probability $\alpha$, $l_0 = 2$, and scale $c$, we have
    \begin{align}
    \small
      \P\eparen{
        \exists t \geq 1: \Psi(Y_t)
        \geq f\eparen{D_\star^2 \sum_{i=1}^t \E_{i-1} \norm{X_i}^2}
      } \leq \alpha.
    \end{align}
  \end{enumerate}
\end{corollary}
The result follows directly from the proof of Corollary 10 in
\citet{howard_exponential_2018}, which shows that $S_t = \Psi(Y_t)$ is
sub-Gaussian or sub-Poisson with appropriate variance process $(V_t)$ for each
case, building upon the work of
\citet{pinelis_approach_1992,pinelis_optimum_1994}. For example, let $(Y_t)$ be
a martingale taking values in any Hilbert space, with $\norm{\cdot}$ the induced
norm, and suppose $\norm{\Delta Y_t} \leq 1$ a.s. for all $t$. Then
\cref{th:banach}(a) with a normal mixture bound yields
\begin{align}
\P\eparen{
  \exists t \geq 1: \norm{Y_t} \geq
  \sqrt{(t + \rho) \log\eparen{\frac{4 (t + \rho)}{\alpha^2 \rho}}}
} \leq \alpha.
\end{align}

Next, let $(S_t)_{t \in \R_{\geq 0}}$ be a continuous-time, real-valued
process. Replacing discrete-time processes in \cref{th:canonical_assumption}
with continuous-time processes, and invoking the continuous-time version of
Ville's inequality, our stitched, mixture and inverted stitching results extend
straightforwardly to continuous time. Below we give 
two examples which follow from Fact 2 of
\citet{howard_exponential_2018}. Here $\eangle{S}_t$ denotes the predictable
quadratic variation of $(S_t)$.

\begin{corollary}\label{th:continuous_time}
  Let $(S_t)_{t \in \R_{\geq 0}}$ be a real-valued process.
  \begin{enumerate}[label=(\alph*)]
  \item If $(S_t)$ is a locally square-integrable martingale with
    a.s. continuous paths, and $f$ is a sub-Gaussian stitched, mixture or
    inverted stitching uniform boundary, then
    $\P(\exists t \in (0, \infty): S_t \geq f(\eangle{S}_t)) \leq e^{-2ab}$.
  \item If $(S_t)$ is a local martingale with $\Delta S_t \leq c$ for all $t$,
    and $f$ is a sub-Poisson mixture bound for scale $c$ or a sub-gamma
    stitched bound for scale $c/3$, then
    $\P(\exists t \in (0, \infty): S_t \geq f(\eangle{S}_t)) \leq \alpha$.
  \end{enumerate}
  
\end{corollary}

For example, if $(S_t)$ is a standard Brownian motion, then
\cref{th:continuous_time}(a) with a polynomial stitched boundary yields, for
any $\eta > 1, s > 1$,
{\small
\begin{align*}
  \P\eparen{
    \exists t \in (0, \infty): S_t \geq
    \frac{\eta^{1/4} + \eta^{-1/4}}{\sqrt{2}} \sqrt{
      (1 \bmax t) \eparen{s \log\log(\eta (1 \bmax t))
      + \log \frac{\zeta(s)}{\alpha \log^s \eta}}
    }
  } \leq \alpha.
\end{align*}
}

\section{Sufficient conditions for Definition 1}
\label{sec:reference_tables}

\Cref{tab:scalar_suff_cond} offers a summary of sufficient conditions for
\cref{th:canonical_assumption} to hold when $(S_t)$ is a scalar process, while
\cref{tab:matrix_suff_cond} gives conditions for matrix-valued processes. See
\citet[Section 2]{howard_exponential_2018} for details.

\begin{table}[p!]
\renewcommand{\arraystretch}{1.4}
\begin{tabular}{llll}
\toprule
 & Condition on $S_t$ & $\psi$ & $V_t$ \\
\midrule
\multicolumn{4}{l}{\emph{Discrete time, $S_t = \sum_{i=1}^t X_i$, one-sided}} \\
\qquad Bernoulli II &
  $X_t \leq h, \E_{t-1} X_t^2 \leq gh$ & $\psi_B$ & $ght$ \\
\qquad Bennett & $X_t \leq c $ & $\psi_P$ & $\sum_{i=1}^t \E_{i-1} X_i^2$ \\
\qquad Bernstein
  & $\E_{t-1} (X_t)^k \leq \frac{k!}{2} c^{k-2} \E_{t-1} X_t^2$
  & $\psi_G$ & $\sum_{i=1}^t \E_{i-1} X_i^2$ \\
\qquad $^*$Heavy on left
  & $\E_{t-1} T_a(X_t) \leq 0 \text{ for all } a > 0$ & $\psi_N$
  & $\sum_{i=1}^t X_i^2$ \\
\qquad Bounded below & $X_t \geq -c $ & $\psi_E$ & $\sum_{i=1}^t X_i^2$ \\
\multicolumn{4}{l}{\emph{Discrete time, $S_t = \sum_{i=1}^t X_i$, two-sided}} \\
\qquad Parametric & $X_t \asiid F$ &
  $\log \E e^{\lambda X_1}$ &  $t$ \\
\qquad Bernoulli I & $-g  \leq X_t \leq h $ & $\psi_B$ &
  $ght $ \\
\qquad Hoeffding-KS & $-g_t  \leq X_t \leq h_t $ & $\psi_N$
  & $\sum_{i=1}^t \varphi(g_i, h_i) $ \\
\qquad Hoeffding I & $-g_t  \leq X_t \leq h_t $ & $\psi_N$
  & $\sum_{i=1}^t \pfrac{g_i+h_i}{2}^2 $ \\
\qquad $^*$Symmetric & $X_t \sim -X_t \mid \Fcal_{t-1}$ & $\psi_N$ &
  $\sum_{i=1}^t X_i^2$ \\
\qquad Self-normalized I & $\E_{t-1} X_t^2 < \infty$ & $\psi_N$ &
  $\frac{1}{3} \sum_{i=1}^t \eparen{X_i^2 + 2 \E_{i-1} X_i^2}$ \\
\qquad Self-normalized II & $\E_{t-1} X_t^2 < \infty$ & $\psi_N$ &
  $\frac{1}{2} \sum_{i=1}^t \eparen{(X_i)_+^2 + \E_{i-1} (X_i)_-^2}$ \\
\qquad Cubic self-normalized & $\E_{t-1} \abs{X_t}^3 < \infty$ &
  $\psi_G$ & $\sum_{i=1}^t \eparen{X_i^2 + \E_{i-1} \abs{X_i}^3}$ \\
\multicolumn{4}{l}{\emph{Continuous time, one-sided}} \\
\qquad Bennett & $\Delta S_t \leq c$ & $\psi_P$ & $\eangle{S}_t$ \\
\qquad Bernstein & $W_{m,t} \leq \frac{m!}{2} c^{m-2} V_t$ & $\psi_G$ &
  $V_t$ \\
\multicolumn{4}{l}{\emph{Continuous time, two-sided}} \\
\qquad L\'evy & $\E e^{\lambda S_1} < \infty$ & $\log \E e^{\lambda S_1}$ &
  $t$ \\
\qquad Continuous paths & $\Delta S_t \equiv 0$ & $\psi_N$ & $\eangle{S}_t$ \\
\bottomrule
\end{tabular}
\caption{Summary of sufficient conditions a real-valued, discrete- or
  continuous-time process $(S_t)$ to be sub-$\psi$ with the given variance
  process. We assume $(S_t)$ is a martingale in every case except
  the starred ones ($^*$), when the first moment $\E \eabs{X_t}$ need not exist.
  See \citet[Section 2]{howard_exponential_2018} for details. One-sided
  conditions yield a bound on right-tail deviations only, while two-sided
  conditions yield bounds on both tails. For continuous-time cases, $\Delta S_t$
  denotes the jumps of $(S_t)$ and $\eangle{S}_t$ denotes the predictable
  quadratic variation. For the heavy on left case, the truncation function is
  defined as $T_a(y) \defineas (y \bmin a) \bmax -a$ for $a > 0$
  \citep{bercu_exponential_2008}. The function $\varphi$ used in the
  Hoeffding-KS case is defined in \eqref{eq:varphi_defn}. The process $W_{m,t}$
  in the continuous-time Bernstein case is defined in Fact 2(c) of
  \citet{howard_exponential_2018} (cf. \citet{van_de_geer_exponential_1995}).
  \label{tab:scalar_suff_cond}}
\end{table}

\begin{table}[p!]
\renewcommand{\arraystretch}{1.4}
\begin{tabular}{llll}
\toprule
 & Condition on $Y_t=\sum_{i=1}^t X_t$ & $\psi$ & $Z_t$ \\
\midrule
\multicolumn{4}{l}{\emph{One-sided}} \\
\qquad Bernoulli II &
  $X_t \preceq h I_d, \E_{t-1} X_t^2 \preceq gh I_d$ & $\psi_B$
  & $ght I_d$ \\
\qquad Bennett & $X_t \preceq c I_d$ & $\psi_P$
  & $\sum_{i=1}^t \E_{i-1} X_i^2$ \\
\qquad Bernstein
  & $\E_{t-1} (X_t)^k \preceq \frac{k!}{2} c^{k-2} \E_{t-1} X_t^2$
  & $\psi_G$ & $\sum_{i=1}^t \E_{i-1} X_i^2$ \\
\qquad Bounded below & $X_t \succeq -c I_d$ & $\psi_E$
  & $\sum_{i=1}^t X_i^2$ \\
\multicolumn{4}{l}{\emph{Two-sided}} \\
\qquad Bernoulli I & $-g I_d \preceq X_t \preceq h I_d$ & $\psi_B$ &
  $ght I_d$ \\
\qquad Hoeffding-KS & $-G_t I_d \preceq X_t \preceq H_t I_d$ & $\psi_N$
  & $\sum_{i=1}^t \varphi(G_i, H_i) I_d$ \\
\qquad Hoeffding I & $-G_t I_d \preceq X_t \preceq H_t I_d$ &
  $\psi_N$
  & $\sum_{i=1}^t \pfrac{G_i+H_i}{2}^2 I_d$ \\
\qquad Hoeffding II & $X_t^2 \preceq A_t^2$ & $\psi_N$ &
  $\sum_{i=1}^t A_i^2$ \\
\qquad $^*$Symmetric & $X_t \sim -X_t \mid \Fcal_{t-1}$ & $\psi_N$
  & $\sum_{i=1}^t X_i^2$ \\
\qquad Self-normalized I & $\E_{t-1} X_t^2 < \infty$ & $\psi_N$
  & $\frac{1}{3} \sum_{i=1}^t \eparen{X_i^2 + 2 \E_{i-1} X_i^2}$
  \\
\qquad Self-normalized II & $\E_{t-1} X_t^2 < \infty$ & $\psi_N$
  & $\frac{1}{2}
     \sum_{i=1}^t \eparen{(X_i)_+^2 + \E_{i-1} (X_i)_-^2}$ \\
\qquad Cubic self-normalized & $\E_{t-1} \abs{X_t}^3 < \infty$ &
  $\psi_G$ & $\sum_{i=1}^t \eparen{X_i^2 + \E_{i-1} \abs{X_i}^3}$
  \\
\bottomrule
\end{tabular}
\caption{Summary of sufficient conditions for \cref{th:canonical_assumption}
  when $Y_t = \sum_{i=1}^t X_i$ with $X_t \in \Hcal^d$, the space of
  Hermitian, $d \times d$ matrices, taking $S_t = \gamma_{\max}(Y_t)$ and
  $V_t=\gamma_{\max}(Z_t)$. We assume
  $\E X_t = 0$ and hence $(Y_t)$ is a martingale in every case except the
  symmetric$^*$ case, when the first moment $\E \abs{X_t}$ need not exist. See
  \citet[Section 2]{howard_exponential_2018} for details. One-sided conditions
  yield a bound on right-tail deviations only, while two-sided conditions yields
  bounds on both tails. The function $\varphi$ used in the Hoeffding-KS case is
  defined in \eqref{eq:varphi_defn}.
  \label{tab:matrix_suff_cond}}
\end{table}

\end{document}